\pdfoutput=1
\RequirePackage{ifpdf}
\ifpdf 
\documentclass[pdftex]{sigma}
\else
\documentclass{sigma}
\fi

\usepackage{ae,aecompl}
\usepackage{tikz}
\usepackage{tikz-cd}
\usepackage{ytableau}
\usetikzlibrary{decorations.markings}
\usepackage{microtype}

\numberwithin{equation}{section}

\newcommand\myeq{\mathrel{\stackrel{\makebox[0pt]{\mbox{\normalfont\tiny def}}}{=}}}

\newtheorem{thm}{Theorem}[section]
\newtheorem{lem}[thm]{Lemma}

\newtheorem{dfprop}[thm]{Proposition/Definition}
\newtheorem{prop}[thm]{Proposition}
\newtheorem{cor}[thm]{Corollary}
\newtheorem{claim}[thm]{Claim}

{\theoremstyle{definition}
\newtheorem{df}[thm]{Definition}
\newtheorem{exam}[thm]{Example}
\newtheorem{rem}[thm]{Remark}
}

\newcommand{\C}{\mathbb{C}}

\newcommand{\R}{\mathbb{R}}
\newcommand{\Z}{\mathbb{Z}}
\newcommand{\N}{\mathbb{N}}

\newcommand{\U}{\mathcal{U}}

\newcommand{\AAA}{\mathcal{A}}
\newcommand{\CC}{\mathcal{C}}
\newcommand{\Nu}{\mathcal{N}}

\newcommand{\GT}{\mathbb{GT}}
\newcommand{\GTp}{\mathbb{GT}^+}

\newcommand{\PPP}{\mathcal{P}}

\newcommand{\TT}{\mathbb{T}}
\newcommand{\DD}{\mathcal{D}}
\newcommand{\DDD}{\mathbb{D}}
\newcommand{\Tau}{\mathcal{T}}
\newcommand{\Mart}{\Omega_{q, t}^{\rm Martin}}
\newcommand{\NN}{\mathbf{N}}

\begin{document}

\allowdisplaybreaks

\newcommand{\arXivNumber}{1704.02429}

\renewcommand{\PaperNumber}{001}

\FirstPageHeading

\ShortArticleName{Asymptotic Formulas for Macdonald Polynomials}

\ArticleName{Asymptotic Formulas for Macdonald Polynomials and\\ the Boundary of the $\boldsymbol{(q, t)}$-Gelfand--Tsetlin Graph}

\Author{Cesar CUENCA}

\AuthorNameForHeading{C.~Cuenca}

\Address{Department of Mathematics, Massachusetts Institute of Technology, USA}
\Email{\href{mailto:cuenca@mit.edu}{cuenca@mit.edu}}
\URLaddress{\url{http://math.mit.edu/~cuenca/}}

\ArticleDates{Received April 21, 2017, in f\/inal form December 09, 2017; Published online January 02, 2018}

\Abstract{We introduce Macdonald characters and use algebraic properties of Macdonald polynomials to study them. As a result, we produce several formulas for Macdonald cha\-racters, which are generalizations of those obtained by Gorin and Panova in [\textit{Ann. Probab.} \textbf{43} (2015), 3052--3132], and are expected to provide tools for the study of statistical mechanical models, representation theory and random matrices. As f\/irst application of our formulas, we characterize the boundary of the $(q,t)$-deformation of the Gelfand--Tsetlin graph when $t = q^{\theta}$ and $\theta$ is a positive integer.}

\Keywords{Branching graph; Macdonald polynomials; Gelfand--Tsetlin graph}

\Classification{33D52; 33D90; 60B15; 60C05}

\section{Introduction}\label{introduction}

Macdonald polynomials are remarkable two-parameter $q$, $t$ generalizations of Schur polynomials. They were f\/irst introduced by Ian G.~Macdonald in~\cite{M1}; the canonical reference is his classical book \cite{M2}. The Macdonald polynomials are very interesting objects for representation theory and integrable systems, due to their connections with quantum groups, e.g.,~\cite{EK, N}, double af\/f\/ine Hecke algebras, e.g.,~\cite{Ch, Ki}, etc. More recently, and releveant for us, Macdonald polynomials have been heavily used to study probabilistic models arising in mathematical physics and random matrix theory. The important work~\cite{BC} of Borodin and Corwin showed how to use the algebraic properties of these polynomials to obtain analytic formulas that allow asymptotic analysis of the so-called Macdonald processes. Remarkably, by specialization or degeneration of the parameters~$q$,~$t$ def\/ining Macdonald processes, the paper~\cite{BC} yields tools that can be used to analyze interacting particle systems~\cite{BCS}, beta Jacobi corners processes~\cite{BG}, probabilistic models from asymptotic representation theory~\cite{BBuO}, among others; see the survey~\cite{BP} and references therein.

It should be noted that the special case $t = q$ of Macdonald processes are known as Schur processes and they have the special property of being determinantal point processes, thus allowing much more control over their asymptotics. Schur processes were introduced by Okounkov and Reshetikhin, as generalizations of the classical Plancherel measures, several years before the work of Borodin and Corwin~\cite{Ok1, OR}. The Schur processes, though a very special case of Macdonald processes, produced various applications to statistical models of plane partitions and random matrices, see for example~\cite{J, OR2}. However, most of the physical models that were studied with the Macdonald processes do not have a determinantal structure, and therefore they could not have been analyzed solely by means of the Schur processes machinery.

The conclusion from the story of Schur and Macdonald processes is that studying the more complicated object can allow one to tackle more complicated questions, despite losing some integrability (such as the determinantal structure in the case of Schur processes). We follow this philosophy in our work, by introducing and studying \textit{Macdonald characters}, two-parameter $q$, $t$ generalizations of normalized characters of unitary groups.
The normalized characters of the unitary groups are expressible in terms of Schur polynomials, reason why we will call them Schur characters, whereas the generalization we present involves Macdonald polynomials.

Our main results are asymptotic formulas for Macdonald characters, which are generalizations of those for Schur characters, proved in~\cite{GP} by dif\/ferent methods. As it is expected, the asymptotic formulas for Schur characters are simpler and they involve certain determinantal structure, whereas the formulas for Macdonald polynomials are more complicated and the determinantal structure is no longer present. However the advantage of our work on Macdonald polynomials, much like the advantage of Macdonald processes over Schur processes, is that we are able to access a number of asymptotic questions that are more general than those given in~\cite{GP}. The tools we obtain in this paper are therefore very exciting, given that the formulas for Schur characters have already produced several applications to stochastic discrete particle systems, lozenge and domino tiling models, and asymptotic representation theory~\cite{BuG1, BuG2,BuK,G0, GP,Pa}.

The paper \cite{Cu} is this article's companion, in which the author studies \textit{Jack characters}, a~natural degeneration of Macdonald characters and obtains their asymptotics in the Vershik--Kerov limit regime. The approach to the study of asymptotics of Jack characters is dif\/ferent from the approach we use here to study the asymptotics of Macdonald characters; in particular, it relies heavily on the \textit{Pieri integral formula}, see~\cite{Cu} for further details. The tools from this paper and~\cite{Cu} af\/ford us a very strong control over the asymptotics of Macdonald characters, Jack characters and Bessel functions~\cite{OO0, Op}, if the number of variables remains f\/ixed and the rank tends to inf\/inity. As another application of the developed toolbox, we have studied a~Jack--Gibbs model of lozenge tilings in the spirit of~\cite{BGG, GS}. The author was able to prove the weak convergence of statistics of the Jack--Gibbs lozenge tilings model near the edge of the boundary to the well-known \textit{Gaussian beta ensemble}, see, e.g., Forrester~\cite[Chapter~20]{Ox} and references therein. This result, and its rational limit concerning corner processes of Gaussian matrix ensembles, will appear in a forthcoming publication.

We proceed with a more detailed description of the results of the present paper.

\subsection{Description of the formulas}

The main object of study in this paper are the Macdonald characters, which we def\/ine as follows. For integers $1\leq m\leq N$, a \textit{Macdonald character of rank $N$ and $m$ variables} is a polynomial, with coef\/f\/icients in $\C(q, t)$, of the form
\begin{gather*}
P_{\lambda}(x_1, \dots, x_m; N, q, t) \myeq \frac{P_{\lambda}\big(x_1, \dots, x_m, 1, t, \dots, t^{N-m-1}; q, t\big)}{P_{\lambda}\big(1, t, t^2, \dots, t^{N-1}; q, t\big)},
\end{gather*}
where $P_{\lambda}(x_1, \dots, x_N; q, t)$ is the Macdonald polynomial of $N$ variables parametrized by the signature $\lambda = (\lambda_1 \geq \lambda_2 \geq \dots \geq \lambda_N)\in\Z^N$. Macdonald characters, under the specialization $t = q$, turn into \textit{$q$-Schur characters}, which have appeared previously in \cite{G, GP}. The reason behind the use of the word ``character'' is that $q$-Schur characters turn into normalized characters of the irreducible rational representations of unitary groups, after the degeneration $q \rightarrow 1$.

We make one further comment about terminology. Macdonald characters, as def\/ined here, are two-parameter $q$,~$t$ degenerations of normalized and irreducible characters of unitary groups. One could also consider a two-parameter degeneration of the characters of the symmetric groups, in the spirit of Lassalle's work~\cite{L}, where a one-parameter degeneration of symmetric group characters was considered. Thus a better name for our object would be \textit{Macdonald unitary character}. For convenience, we simply will use the name Macdonald character. We remark that Macdonald symmetric group characters have not been considered yet, to the author's best knowledge. However, there have been many articles studying the structural theory and asymptotics on Jack symmetric group characters, notably several recent works by Maciej Do{\l}\c{e}ga, Valentin F\'{e}ray and Piotr \'{S}niady, see, e.g., \cite{DF13, DF16, DFS, Sn}.

The main theorems of this paper fall into two categories:
\begin{enumerate}\itemsep=0pt
\item[(A)] Integral representations for Macdonald characters of one variable and arbitrary rank~$N$.
\end{enumerate}

The initial idea that led to the integral representations in this paper is due to Andrei Okounkov, see~\cite[remark following Theorem~3.6]{GP}. An example of the integral formulas we prove is the following theorem. Observe that the integrand is a simple expression in terms of $q$-Gamma functions and can be analyzed by well known methods of asymptotic analysis, such as the method of steepest descent or the saddle-point method~\cite{C}.
In this paper, we study the regime in which the signatures grow to inf\/inity, whereas the remaining parameters are f\/ixed.

\begin{thm}[consequence of Theorem~\ref{macdonaldthm2}]
Assume $q\in (0, 1)$ and $\theta > 0$. Let $N\in\N$, $\lambda\in\GT_N$ and $x\in\C\setminus\{0\}$, $|x|\leq q^{\theta(1-N)}$. The integral below converges absolutely and the identity holds
\begin{gather}
P_{\lambda}\big(x; N, q, q^{\theta}\big) = \frac{\ln q}{1 - q}\frac{\big(xq^{\theta}; q\big)_{\infty}}{\big(xq^{1 + \theta(1 - N)}; q\big)_{\infty}}\frac{\Gamma_q(\theta N)}{2\pi\sqrt{-1}}\nonumber\\
\hphantom{P_{\lambda}\big(x; N, q, q^{\theta}\big) =}{} \times\int_{\CC^+}{\big(xq^{\theta(1-N)}\big)^z\prod_{i=1}^N{\frac{\Gamma_q(\lambda_i + \theta(N-i)-z)}{\Gamma_q(\lambda_i + \theta(N-i+1)-z)}}{\rm d}z},\label{macdonaldthm2eqnreform}
\end{gather}
where $\CC^+$ is a certain contour described in Theorem~{\rm \ref{macdonaldthm2}}, and which looks as in Fig.~{\rm \ref{fig:C}}.
In the formula above, we used the $q$-Pochhammer symbol $(z; q)_{\infty}$ and the $q$-Gamma function $\Gamma_q(z)$; see Appendix~{\rm \ref{app:qtheory}}.
\end{thm}

\begin{enumerate}\itemsep=0pt
\item[(B)] Formulas expressing Macdonald characters of~$m$ variables (and rank~$N$) in terms of Macdonald characters of one variable (and rank~$N$).
\end{enumerate}

These formulas will involve certain $q$-dif\/ference operators. Formulas of this kind will be called \textit{multiplicative formulas}.\footnote{The reason for the name is that analogous formulas for Schur characters were used to prove statements of the form $\lim\limits_{N\rightarrow\infty}{F_{\lambda(N)}(x_1, \dots, x_m)} = \prod\limits_{i=1}^m{\lim\limits_{N\rightarrow\infty}{F_{\lambda(N)}(x_i)}}$, where the functions $F$ are certain normalizations of Schur characters, see, e.g., \cite[Corollaries 3.10 and 3.12]{GP}} One of the simplest multiplicative formulas we prove is the one below that expresses a Macdonald character of two variables in terms of those of one variable; the general formula is given below in Theorem~\ref{macdonaldthm3}.

\begin{thm}[reformulation of Corollary \ref{macdonaldcor3}]
Let $\theta\in\N$, $N\in\N$, $\lambda\in\GT_N$. Then
\begin{gather*}
P_{\lambda}\big(x_1, x_2; N, q, q^{\theta}\big) = \frac{q^{-(N-\frac{3}{2})\theta^2 + \frac{1}{2}\theta} (1 - q)^{\theta} }{\prod\limits_{i=1}^{\theta}{(1 - q^{\theta N - i})}}\frac{1}{\prod\limits_{i=1}^{\theta(N - 1) - 1}{(x_1 - q^{i-\theta})(x_2 - q^{i-\theta})}}\nonumber\\
\qquad{}\times \left( \frac{1}{x_1 - x_2} \circ (D_{q, x_2} - D_{q, x_1}) \right)^{\theta} \left\{ \prod_{i=1}^2{\left( P_{\lambda}\big(x_i; N, q, q^{\theta}\big)\prod_{j=1}^{\theta N-1}{\big(x_i - q^{j-\theta}\big)}\right) }\right\},
\end{gather*}
where $D_{q, x_i}$, $i = 1, 2$, are the linear operators in $\C(q)[x_1, x_2]$ acting on monomials by $D_{q, x_i}(x_1^{m_1}x_2^{m_2})$ $= \frac{1 - q^{m_i}}{1 - q} (x_1^{m_1}x_2^{m_2})$, $i = 1, 2$.
\end{thm}

Observe that the multiplicative formula above requires $\theta\in\N$. It is somewhat surprising that all the identities we prove in this paper, even the integral representations, behave better for $\theta\in\N$.

\subsection[The boundary of the $(q, t)$-Gelfand--Tsetlin graph]{The boundary of the $\boldsymbol{(q, t)}$-Gelfand--Tsetlin graph}

As an application of our formulas, we characterize the space of central probability measures in the path-space of the $(q, t)$-Gelfand--Tsetlin graph, when $t = q^{\theta}$ and $\theta$ is a positive integer.
To state our result, we f\/irst introduce a few notions. Assume that $q, t\in (0, 1)$ are generic real parameters for the moment.

The \textit{Gelfand--Tsetlin graph}, or simply GT graph, is an undirected graph whose vertices are the signatures of all lengths $\GT = \bigsqcup\limits_{N\geq 0}{\GT_N}$; we also include the empty signature $\varnothing$ as the only element of $\GT_0$ for convenience. The set of edges is determined by the interlacing constraints, namely the edges in the GT graph can only join signatures whose lengths dif\/fer by $1$ and $\mu\in\GT_N$ is joined to $\lambda\in\GT_{N+1}$ if and only if
\begin{gather*}
\lambda_{N+1} \leq \mu_N \leq \lambda_N \leq \dots \leq \lambda_2 \leq \mu_1 \leq \lambda_1.
\end{gather*}
If the above inequalities are satisf\/ied, we write $\mu \prec \lambda$. If $\mu\in\GT_N$ and $\lambda\in\GT_{N+1}$ are joined by an edge, then we consider the expression $\Lambda^{N+1}_N(\lambda, \mu)$ given by
\begin{gather*}
\Lambda^{N+1}_N(\lambda, \mu) = \psi_{\lambda/\mu}(q, t)\frac{P_{\mu}\big(t^N, \dots, t^2, t; q, t\big)}{P_{\lambda}\big(t^N, \dots, t, 1; q, t\big)},
\end{gather*}
where $\psi_{\lambda/\mu}(q, t)$ is given in the branching rule for Macdonald polynomials, see Theorem~\ref{branchingmacdonald} below. If $\mu\in\GT_N$ is not joined to $\lambda\in\GT_{N+1}$, set $\Lambda^{N+1}_N(\lambda, \mu) = 0$.
One can easily show, see, e.g., Theorem~\ref{evaluation} below,
\begin{gather*}
\Lambda^{N+1}_N(\lambda, \mu) \geq 0, \qquad \forall\, \lambda\in\GT_{N+1}, \quad \mu\in\GT_N,\\
\sum_{\mu\in\GT_N}{\Lambda^{N+1}_N(\lambda, \mu)} = 1, \qquad \forall\, \lambda\in\GT_{N+1}.
\end{gather*}
Thus, for any $N\in\N$, $\lambda\in\GT_{N+1}$, $\Lambda^{N+1}_N(\lambda, \cdot)$ is a probability measure on $\GT_N$. For this reason, we will call the expressions $\Lambda^{N+1}_N(\lambda, \mu)$ \textit{cotransition probabilities}.

Next we def\/ine the \textit{path-space} $\Tau$ of the GT graph as the set of inf\/inite paths in the GT graph that begin at $\varnothing\in\GT_0$: $\Tau = \{\tau = (\varnothing = \tau^{(0)} \prec \tau^{(1)} \prec \tau^{(2)} \prec \cdots) \colon \tau^{(n)}\in\GT_n \ \forall\, n\in\Z_{\geq 0}\}$. Each f\/inite path of the form $\phi = \big(\varnothing = \phi^{(0)} \prec \phi^{(1)} \prec \cdots \prec \phi^{(n)}\big)$ def\/ines a cylinder set $S_{\phi} = \big\{\tau\in\Tau \colon \tau^{(1)} = \phi^{(1)}, \dots, \tau^{(n)} = \phi^{(n)}\big\} \subset \Tau$. We equip~$\Tau$ with the $\sigma$-algebra generated by the cylinder sets~$S_{\phi}$, over all f\/inite paths~$\phi$.
Equivalently, the $\sigma$-algebra of~$\Tau$ is its Borel $\sigma$-algebra if we equip~$\Tau$ with the topology it inherits as a subspace of the product $\prod\limits_{n\geq 0}{\GT_n}$. Each probability measure~$M$ on~$\Tau$ admits a pushforward to a probability measure on~$\GT_m$ via the obvious projection map
\begin{gather*}
\operatorname{Proj}_m\colon \ \Tau\subset\prod_{n\geq 0}{\GT_n} \longrightarrow \GT_N,\\
\hphantom{\operatorname{Proj}_m\colon}{} \ \tau = \big(\tau^{(0)}\prec \tau^{(1)}\prec \tau^{(2)} \prec \cdots\big) \mapsto \tau^{(N)}.
\end{gather*}
We say that a probability measure $M$ on $\Tau$ is a $(q, t)$-\textit{central measure} if
\begin{gather*}
M\big( S\big(\phi^{(0)} \prec \phi^{(1)} \prec \cdots\prec \phi^{(N-1)} \prec \phi^{(N)}\big) \big)\\
\qquad{} =
\Lambda^{N}_{N-1}\big(\phi^{(N)}, \phi^{(N-1)}\big) \cdots \Lambda^{1}_0\big(\phi^{(1)}, \phi^{(0)}\big) M_N\big(\phi^{(N)}\big),
\end{gather*}
for all $N \geq 0$, all f\/inite paths $\phi^{(0)} \prec \cdots \prec \phi^{(N)}$, and for some probability measures $M_N$ on $\GT_N$.
It then automatically follows that $M_N = (\operatorname{Proj}_N)_*M$ are the pushforwards of $M$; moreover, they satisfy the coherence relations
\begin{gather*}
M_N(\mu) = \sum_{\lambda\in\GT_{N+1}}{M_{N+1}(\lambda)\Lambda^{N+1}_N(\lambda, \mu)}, \qquad \forall\, N \geq 0, \quad \forall\, \mu\in\GT_N.
\end{gather*}

We denote by $M_{\rm prob}(\Tau)$ the set of $(q, t)$-central (probability) measures on $\Tau$; it is clearly a~convex subset of the Banach space of all f\/inite and signed measures on $\Tau$. Let us denote by $\Omega_{q, t} = \textrm{Ex}(M_{\rm prob}(\Tau))$ the set of its extreme points. From a general theorem, we can deduce that $\Omega_{q, t} \subset M_{\rm prob}(\Tau)$ is a Borel subset. We call $\Omega_{q, t}$, with its inherited topology, the \textit{boundary of the $(q, t)$-Gelfand--Tsetlin graph}. The theorem stated below, which is our main application, completely characterizes the topological space $\Omega_{q, t}$. Before stating it, let us make a couple of relevant def\/initions.

Consider the set of weakly increasing integers $\Nu = \{\nu = (\nu_1 \leq \nu_2 \leq \cdots)\colon \nu_1, \nu_2, \dots \in\Z\}$ and equip it with the topology inherited from the product $\Z^{\infty} = \Z\times\Z\times\cdots$ of countably many discrete spaces. For each $k\in\Z$, we can def\/ine the automorphism $A_k$ of $\Nu$ by $\nu \mapsto A_k\nu = (\nu_1 + k \leq \nu_2 + k \leq \cdots)$. Clearly $A_k$ has inverse $A_{-k}$. There is a similar automorphism of $\GT$, given by $\lambda \mapsto A_k\lambda = (\lambda_1 + k \geq \lambda_2 + k \geq \cdots)$, $\varnothing \mapsto A_k\varnothing = \varnothing$, which restricts to automorphisms $\GT_m \rightarrow \GT_m$, for each $m\in\Z_{\geq 0}$.

For any $k\in\Z$, one can easily show that $\mu\in\GT_m$ interlaces with $\lambda\in\GT_{m+1}$ if\/f $A_k\mu\in\GT_m$ interlaces with $A_k\lambda\in\GT_{m+1}$, that is, $\mu\prec\lambda$ if\/f $A_k\mu\prec A_k\lambda$.
This allows us to def\/ine automorphisms $A_k$ of $\Tau$ by
\begin{gather*}
A_k\colon \ \Tau \longrightarrow \Tau,\\
\hphantom{A_k\colon}{} \ \tau = \big(\varnothing \prec \tau^{(1)} \prec \tau^{(2)} \prec \cdots\big) \mapsto A_k\tau = \big(\varnothing \prec A_k\tau^{(1)} \prec A_k\tau^{(2)} \prec \cdots\big).
\end{gather*}
One can similarly obtain maps $A_k$ on the set of f\/inite paths of length $n$ by $\phi = (\phi^{(0)} \prec \phi^{(1)} \prec \cdots \prec \phi^{(n)}) \mapsto A_k\phi = (A_k\phi^{(0)} \prec A_k\phi^{(1)}\prec \cdots\prec A_k\phi^{(n)})$. Consequently we can also def\/ine automorphisms on cylinder sets by $A_k S_{\phi} = S_{A_k\phi}$, for all f\/inite paths $\phi = (\phi^{(0)} \prec \phi^{(1)} \prec \cdots \prec \phi^{(n)})$, in the natural way.

We named several maps above by the same letter $A_k$, but there should be no risk of confusion.

For our main theorem, we make the assumption $\theta\in\N$, $t = q^{\theta}$. We believe the theorem can be generalized for any $\theta>0$, but we do not have a proof at the moment.

\begin{thm}\label{thm:mainapplication}
Assume $q\in (0, 1)$, $\theta\in\N$ and set $t = q^{\theta}$.
\begin{enumerate}\itemsep=0pt
	\item[$1.$] There exists a homeomorphism $\NN\colon \Nu \rightarrow \Omega_{q, t}$ sending each $\nu\in\Nu$ to the $(q, t)$-central probability measure $M^{\nu} \in \Omega_{q, t}$ determined by the relations
\begin{gather}
\sum_{\lambda\in\GT_m}{M_m^{\nu}(\lambda)\frac{P_{\lambda}\big(x_1, x_2t, \dots, x_mt^{m-1}; q, t\big)}{P_{\lambda}\big(1, t, \dots, t^{m-1}; q, t\big)}} = \Phi^{\nu}\big(x_1 t^{1-m}, \dots, x_{m-1}t^{-1}, x_m; q, t\big),\nonumber \\
\forall\, m\in\N, \qquad \forall\, (x_1, \dots, x_m)\in\TT^m.\label{eq:macdonaldgenerating}
\end{gather}
In \eqref{eq:macdonaldgenerating}, we denoted by $\{M^{\nu}_m\}_{m\geq 1}$ the corresponding sequence of pushforwards of~$M^{\nu}$ under the projection maps $\operatorname{Proj}_m\colon \Tau\rightarrow\GT_m$.
The left side in~\eqref{eq:macdonaldgenerating} is absolutely convergent on~$\TT^m$, $\mathbb{T} = \{z\in\C \colon |z| = 1\}$, and the functions $\Phi^{\nu}$ in the right side are defined in~\eqref{eqn:Phi1} and~\eqref{def:Phinu}. The probability measure $M^{\nu}$ is determined uniquely by the relations~\eqref{eq:macdonaldgenerating}.

	\item[$2.$] For each $k\in\Z$, the probability measures $M^{\nu}$ and $M^{A_k\nu}$ are related by
\begin{gather}
M^{A_k\nu}(S_{A_k\phi}) = M^{\nu}(S_{\phi}), \quad \textrm{for all finite paths} \quad \phi = \big(\phi^{(0)} \prec \phi^{(1)} \prec \cdots \prec \phi^{(n)}\big).
\end{gather}
Moreover the $(q, t)$-coherent sequences $\{M^{\nu}_m\}_{m \geq 0}$ and $\{M^{A_k\nu}_m\}_{m \geq 0}$ are related by
\begin{gather}\label{eq:shifts2}
M_m^{A_k \nu}(A_k \lambda) = M^{\nu}_m(\lambda), \qquad \forall \, m \geq 0, \qquad \lambda\in\GT_m.
\end{gather}
\end{enumerate}
\end{thm}

Another main result of this article is Theorem~\ref{thm:characterization}, where we characterize the \textit{Martin boundary} of the $(q, t)$-Gelfand--Tseltin graph for $t = q^{\theta}$ and $\theta\in\N$. In fact, we f\/irst prove that the Martin boundary is homeomorphic to $\Nu$ and then show that the minimal boundary $\Omega_{q, t}$ coincides with the Martin boundary.
See Sections~\ref{sec:martindef} and~\ref{sec:martinchar} for the def\/inition and characterization of the Martin boundary.

\subsection{Comments on Theorem~\ref{thm:mainapplication} and connections to existing literature}

Our f\/irst comment is that Theorem~\ref{thm:mainapplication} is a generalization of the main theorem in the article of Vadim Gorin~\cite{G}, which is the special case $\theta = 1$ of our theorem, and characterizes the boundary of the \textit{$q$-Gelfand--Tsetlin graph}. Some ideas in the proofs are the same, especially the overall scheme of using the ergodic method of Vershik--Kerov, see~\cite{VK}, but we need many new arguments as well. For example, \cite{G} makes heavy use of the \textit{shifted Macdonald polynomials}, in particular the binomial formula for shifted Macdonald polynomials at $t = q$, \cite{Ok0}, while we do not use them at all. Moreover, in order to prove that the boundary of the $q$-Gelfand--Tsetlin graph is homeomorphic to $\Nu$, \cite{G} made use of the following closed formula for the \textit{shifted-Schur generating function} of $M_N^{\nu, \theta = 1}$, in the case that $\nu_1 \geq 0$:
\begin{gather}\label{eqn:shiftedgenerating}
\sum_{\lambda\in\GTp_N}{M_N^{\nu, \theta = 1}(\lambda)\frac{s_{\lambda}^*\big(q^{N-1}x_1, \dots, q^{N-1}x_N; q^{-1}\big)}{s_{\lambda}^*\big(0, \dots, 0; q^{-1}\big)}} = H^{\nu}(x_1)\cdots H^{\nu}(x_N),
\end{gather}
where
\begin{gather*}
H^{\nu}(x) = \frac{\prod\limits_{i=0}^{\infty}{(1 - q^i t)}}{\prod\limits_{j=1}^{\infty}{(1 - q^{\nu_j + j - 1}t)}}.
\end{gather*}
In the formula above, $s_{\lambda}^*(x_1, \dots, x_N; q)$ is the shifted Macdonald polynomial at $t = q$. In addition to the usefulness of the closed formula~(\ref{eqn:shiftedgenerating}) above, the multiplicative structure is surprising. It would be interesting to f\/ind a closed formula for the shifted Macdonald generating function of the measures~$M_N^{\nu}$, for general $\theta\in\N$, and f\/ind out if the multiplicative structure still holds in this generality.

It is shown in~\cite{G} that their main statement is equivalent to the characterization of certain Gibbs measures on lozenge tilings. A conjectural characterization of positive $q$-Toeplitz matrices is also given in that paper. Finally, it is mentioned that the asymptotics of $q$-Schur functions is related to quantum traces and the representation theory of $U_{\epsilon}(\mathfrak{gl}_{\infty})$. It would be interesting to extend some of these statements to the Macdonald case, especially to connect the asymptotics of Macdonald characters to the representation theory of inductive limits of quantum groups.

Several other ``boundary problems'' have appeared in the literature in various contexts. For instance, in the limiting case $t = q \rightarrow 1$, the problem of characterizing the boundary $\Omega_{q, t}$ becomes equivalent to characterizing the space of extreme characters of the inf\/inite-dimensional unitary group ${\rm U}(\infty) = \lim\limits_{\rightarrow}{{\rm U}(N)}$. The answer also characterizes totally positive Toeplitz matrices \cite{Ed, VK, Voi}. Also in the degenerate case $t = q^2$ or $t = q^{1/2}$ and $q\rightarrow 1$, the boundary problem becomes equivalent to characterizing the space of extreme spherical functions of the inf\/inite-dimensional Gelfand pairs $({\rm U}(2\infty), {\rm Sp}(\infty))$ and $({\rm U}(\infty), {\rm O}(\infty))$, respectively. This question, and in fact a more general one-parameter ``Jack''-degeneration, was solved in~\cite{OO1}. A similar degenerate question in the setting of random matrix theory was studied in~\cite{OlV}. Some of the tools in this paper can be degenerated easily to these scenarios and they may provide an alternative approach to their proof as well; for example, the special case of our toolbox in the case $t = q \rightarrow 1$ was used in \cite{GP} to study the corresponding boundary problem, and in \cite{Cu} we also study ref\/ine the asymptotic result that is needed to solve the boundary problem of \cite{OO1}.

Another similar boundary problem in a somewhat dif\/ferent direction is the following. Assume we consider the \textit{Young graph} instead of the Gelfand--Tsetlin graph, e.g., see~\cite{BOl}. Assume also that the cotransition probabilities coming from the branching rule of Macdonald polynomials are replaced by the cotransition probabilities coming from the \textit{Pieri-rule}. In this setting, the boundary problem has not been solved yet, but it is expected that the answer is given by \textit{Kerov's conjecture}, which characterizes \textit{Macdonald-positive specializations}, see \cite[Section~2]{BC}.

Finally, it was brought to my attention, after I completed the results of this paper, that Grigori Olshanski has obtained a characterization of the extreme set of $(q, t)$-central measures in the \textit{extended Gelfand--Tsetlin graph} for more general parameters $q$,~$t$ by dif\/ferent methods. His work follows the setting of the paper~\cite{GOl} of Gorin--Olshanski, which is some sort of analytic continuation to our proposed boundary problem. Interestingly, new features arise, e.g., two copies of the space $\Nu$ characterize the boundary in his context, one can def\/ine and work with suitable analogues of $zw$-measures, etc. Another related work in the $t = q$ case is his recent article~\cite{Ol0}.

\subsection{Organization of the paper}

The present work is organized as follows. In Section~\ref{macdonaldsection}, we brief\/ly recall some important algebraic properties of Macdonald polynomials that will be used to obtain our main results. We prove integral representations for Macdonald characters of one variable in Section~\ref{macdonaldintegralsec}. Next, in Section~\ref{macdonaldmultiplicativesec}, we obtain multiplicative formulas for Macdonald characters of a given number of variables $m\in\N$ in terms of those of one variable.
By making use of our formulas, in Section~\ref{sec:asymptotics} we obtain asymptotics of Macdonald characters as the signatures grow to inf\/inity in a specif\/ic limit regime. In Sections~\ref{sec:qtpreliminaries} and~\ref{sec:qtboundary}, we def\/ine and characterize the boundary of the $(q, t)$-Gelfand--Tsetlin graph in the case that $\theta\in\N$ and $t = q^{\theta}$. The asymptotic statements of Section~\ref{sec:asymptotics} play the key role in the characterization of the boundary.

In Appendix~\ref{app:qtheory}, we have bundled the necessary language and results of $q$-theory that are used throughout the paper. In Appendix~\ref{app:cfunctions}, we make some computations with expressions that appear in the multiplicative formulas for Macdonald polynomials.

\section{Symmetric Laurent polynomials}\label{macdonaldsection}

A canonical reference for symmetric polynomials is~\cite{M2}. We choose to give a brief overview of the tools that we need from~\cite{M2}, in order to f\/ix terminology and to introduce lesser known objects, such as signatures and Macdonald Laurent polynomials.

\subsection{Partitions, signatures and symmetric Laurent polynomials}

A \textit{partition} is a f\/inite sequence of weakly decreasing nonnegative integers $\lambda = (\lambda_1 \geq \lambda_2 \geq \cdots\geq \lambda_k)$, $\lambda_i\in\Z_{\geq 0}$ $\forall\, i$. We identify partitions that dif\/fer by trailing zeroes; for example, $(4, 2, 2, 0, 0)$ and $(4, 2, 2)$ are the same partition. We def\/ine the \textit{size} of $\lambda$ to be the sum $|\lambda| \myeq \lambda_1 + \dots + \lambda_k$, and its \textit{length} $\ell(\lambda)$ to be the number of strictly positive elements of it.
The \textit{dominance order} for partitions is a partial order given by letting $\mu \leq \lambda$ if $|\mu| = |\lambda|$ and $\mu_1 + \cdots + \mu_i \leq \lambda_1 + \cdots + \lambda_i$ for all $i$. As usual, we let $\mu < \lambda$ if $\mu \leq \lambda$ and $\mu \neq \lambda$.

Partitions can be graphically represented by their \textit{Young diagrams}. The Young diagram of partition $\lambda$ is the array of boxes with coordinates $(i, j)$ with $1\leq j\leq \lambda_i$, $1\leq i\leq\ell(\lambda)$, where the coordinates are in matrix notation (row labels increase from top to bottom and column labels increase from left to right), see Fig.~\ref{youngdiagram1}.

\begin{figure}[h!]
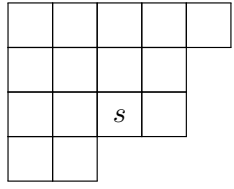
\centering
\ytableausetup{centertableaux}
\begin{ytableau}
\ & & & & \\
 & & & \\
 & & s & \\
&
\end{ytableau}
\caption{Young diagram for the partition $\lambda = (5, 4, 4, 2)$. Square $s = (3, 3)$ has arm length, arm colength, leg length and leg colength given by $a(s) = 1$, $a'(s) = 2$, $l(s) = 0$, $l'(s) = 2$.}\label{youngdiagram1}
\end{figure}

A \textit{signature} is a sequence of weakly decreasing integers $\lambda = (\lambda_1 \geq \lambda_2 \geq \dots\geq \lambda_k)$, $\lambda_i\in\Z$ $\forall\, i$. A \textit{positive signature} is a signature whose elements are all nonnegative. The \textit{length} of a~signature, or positive signature, is the number~$k$ of elements of it. Positive signatures which dif\/fer by trailing zeroes are not identif\/ied, in contrast to partitions; for example, $(4, 2, 2, 0, 0)$ and $(4, 2, 2)$ are dif\/ferent positive signatures, the f\/irst of length $5$ and the second of length $3$. We shall denote~$\GT_N$ (resp.~$\GTp_N$) the set of signatures (resp. positive signatures) of length~$N$. Evidently~$\GTp_N$ can be identif\/ied with the set of all partitions of length $\leq N$. Under this identif\/ication, we are allowed to talk about the Young diagram of a positive signature $\lambda\in\GTp_N$, its size, the dominance order, and other attributes that are typically associated to partitions. Note, however, that length is def\/ined dif\/ferently for partitions and for positive signatures.

Let us now switch to notions pertaining to symmetric (Laurent) polynomials. Fix a positive integer $N$. Consider the f\/ield $F = \C(q, t)$ and recall the algebra $\Lambda_F[x_1, \dots, x_N]$ of symmetric polynomials on the variables $x_1, \dots, x_N$ with coef\/f\/icients in $F$. For any $m\in\Z_{\geq 0}$, recall also the subalgebra $\Lambda_F^m[x_1, \dots, x_N]$ of symmetric polynomials on $x_1, \dots, x_N$ that are homogeneous of degree $m$; then
\begin{gather*}
\Lambda_F[x_1, \dots, x_N] = \bigoplus_{m\geq 0}{\Lambda_F^m[x_1, \dots, x_N]}.
\end{gather*}
We also denote by $\Lambda_F[x_1^{\pm}, \dots, x_N^{\pm}]$ the algebra of symmetric (with respect to the transpositions $x_i \leftrightarrow x_{i+1}$ for $i = 1, \dots, N-1$) Laurent polynomials in the variables $x_1, \dots, x_N$.

The connection between partitions/signatures and symmetric polynomials comes from the observation that $\dim_F \left(\Lambda_F^m[x_1, \dots, x_N]\right)$ is the number of partitions of size $m$ and length $\leq N$, or equivalently the number of positive signatures of size $m$ and length $N$. A basis for the space $\Lambda_F^m[x_1, \dots, x_N]$ is given by the \textit{monomial symmetric polynomials} $m_{\lambda}(x_1, \dots, x_N)$, with $|\lambda| = m$, $\ell(\lambda)\leq N$, def\/ined by
\begin{gather*}
m_{\lambda}(x_1, \dots, x_N) \myeq \sum_{\mu\in S_N\cdot\lambda}{x_1^{\mu_1}\cdots x_N^{\mu_N}},
\end{gather*}
where $S_N\cdot\lambda$ is the orbit of $\lambda$ under the permutation action of $S_N$, and the sum runs over distinct elements $\mu$ of that orbit. It is implied that $\{m_{\lambda}(x_1, \dots, x_N)\colon \ell(\lambda) \leq N\}$ is a basis of $\Lambda_F[x_1, \dots, x_N]$.

\subsection{Macdonald polynomials and Macdonald characters}\label{macdonaldpolysection}

\begin{dfprop}[{\cite[Chapter~VI, Sections~3, 4, 9]{M2}}]\label{def:macdonaldpolys}
The \textit{Macdonald polynomials} $P_{\lambda}(x_1, \dots, x_N; q, t)$, for partitions $\lambda$ with $\ell(\lambda) \leq N$, are the unique elements of $\Lambda_F[x_1, \dots, x_N]$ satisfying the following two properties
\begin{itemize}\itemsep=0pt
	\item \textit{Triangular decomposition}: $\displaystyle P_{\lambda}(x_1, \dots, x_N; q, t) = m_{\lambda} + \sum\limits_{\mu\colon \mu < \lambda}{c_{\lambda, \mu}m_{\mu}}$, for some $c_{\lambda, \mu}\in F$, and the sum is over partitions $\mu$ with $\ell(\mu) \leq N$, and $\mu < \lambda$ in the dominance order.

	\item \textit{Orthogonality relation}: Let $[\cdot]_0 \colon F[x_1^{\pm}, \dots, x_N^{\pm}] \rightarrow F$ be the constant term map
\begin{gather*}
\left[ \sum_{\lambda = (\lambda_1 \geq \cdots \geq \lambda_N) \in \Z^N}{a_{\lambda}x_1^{\lambda_1}\cdots x_N^{\lambda_N}} \right]_0 = a_{(0, \dots, 0)}.
\end{gather*}
The Macdonald polynomials are orthogonal with respect to the inner product $(\cdot, \cdot)_{q, t}$ on $\Lambda_F[x_1^{\pm}, \dots, x_N^{\pm}]$ given by $\displaystyle (f, g)_{q, t} \myeq [f(x_1, \dots, x_N)g\big(x_1^{-1}, \dots, x_N^{-1}\big)\Delta_{q, t}]_0$, where
\[
\Delta_{q, t} \myeq \prod_{1\leq i\neq j\leq N}\prod_{k=0}^{\infty}{\frac{1 - q^{k}x_ix_j^{-1}}{1 - q^{k}tx_ix_j^{-1}}}.
\]
\end{itemize}
\end{dfprop}

Note that $P_{\varnothing}(q, t) = 1$. If $N < \ell(\lambda)$, we set $P_{\lambda}(x_1, \dots, x_N; q, t) \myeq 0$ for convenience.

When we are talking about Macdonald polynomials and some of their properties which hold regardless of the number $N$ of variables, as long as $N$ is large enough, we simply write $P_{\lambda}(q, t)$ instead of $P_{\lambda}(x_1, \dots, x_N; q, t)$.

From the triangular decomposition of Macdonald polynomials, $P_{\lambda}(q, t)$ is a homogeneous polynomial of degree $|\lambda|$. Moreover $\{P_{\lambda}(x_1, \dots, x_N) \colon \ell(\lambda) \leq N\}$ is a basis of $\Lambda_F[x_1, \dots, x_N]$. Finally, we have the following \textit{index stability} property:
\begin{gather}\label{macdonaldsignatures1}
P_{(\lambda_1 + 1, \dots, \lambda_N + 1)}(x_1, \dots, x_N; q, t) = (x_1\cdots x_N)\cdot P_{\lambda}(x_1, \dots, x_N; q, t).
\end{gather}

As pointed out before, the set of partitions of length $\leq N$ is in bijection with $\GTp_N$. Thus we can index the Macdonald polynomials by positive signatures rather than by partitions: for any $\lambda\in\GTp_N$, we let $P_{\lambda}(x_1, \dots, x_N; q, t)$ be the Macdonald polynomial corresponding to the partition associated to $\lambda$. We can slightly extend the def\/inition above and introduce Macdonald Laurent polynomials $P_{\lambda}(x_1, \dots, x_N; q, t)$ for any $\lambda\in\GT_N$. Let $\lambda\in\GT_N$ be arbitrary. If $\lambda_N \geq 0$, then $\lambda\in\GTp_N$ and $P_{\lambda}(x_1, \dots, x_N; q, t)$ is already def\/ined. If $\lambda_N < 0$, choose $m\in\N$ such that $\lambda_N + m \geq 0$ and so $(\lambda_1 + m, \dots, \lambda_N + m)\in\GTp_N$. Then def\/ine
\begin{gather}\label{macdonaldsignatures11}
P_{\lambda}(x_1, \dots, x_N; q, t) \myeq (x_1\cdots x_N)^{-m}\cdot P_{(\lambda_1 + m, \dots, \lambda_N + m)}(x_1, \dots, x_N; q, t).
\end{gather}
By virtue of the index stability property, the Macdonald (Laurent) polynomial $P_{\lambda}(x_1, \dots, x_N; q, t)$ is well-def\/ined and does not depend on the value of $m$ that we choose. For simplicity, we call $P_{\lambda}(x_1, \dots, x_N; q, t)$ a Macdonald polynomial, whether $\lambda\in\GTp_N$ or not. In a similar fashion, we can def\/ine monomial symmetric polynomials $m_{\lambda}$, for any $\lambda\in\GT_N$.

Recall the def\/initions of the \textit{arm-length, arm-colength, leg-length, leg-colength} $a(s)$, $a'(s)$, $l(s)$, $l'(s)$ of the square $s = (i, j)$ of the Young diagram of $\lambda$, given by $a(s) = \lambda_i - j$, $a'(s) = j - 1$, $l(s) = \lambda_j' - i$, $l'(s) = i - 1$;
we note that $\lambda_j' = |\{ i\colon \lambda_i \geq j \} |$ is the length of the $j$th part of the \textit{conjugate partition} $\lambda'$, see Fig.~\ref{youngdiagram1}.

We use terminology from $q$-analysis, see Appendix~\ref{app:qtheory}; particularly we use the def\/inition of $q$-Pochhammer symbols $(z; q)_n \myeq \prod\limits_{i=0}^{n-1}{(1 - zq^i)}$ and $(z; q)_{\infty} \myeq \prod\limits_{i=0}^{\infty}{(1 - zq^i)}$.

For $\lambda\in\bigsqcup\limits_{N\geq 0}{\GTp_N}$, def\/ine the \textit{dual Macdonald polynomials} $Q_{\lambda}(q, t)$ as the following normaliza\-tion of Macdonald polynomials
\begin{gather}\label{PtoQ}
Q_{\lambda}(q, t) \myeq b_{\lambda}(q, t)P_{\lambda}(q, t), \qquad b_{\lambda}(q, t) \myeq \prod_{s\in\lambda}{\frac{1 - q^{a(s)}t^{l(s) + 1}}{1 - q^{a(s) + 1}t^{l(s)}}}.
\end{gather}

The \textit{complete homogeneous symmetric $($Macdonald$)$ polynomials} $g_0 = 1, g_1, g_2, \dots$ are the one-row dual Macdonald polynomials:
\begin{gather}\label{gtoQ}
g_n(q, t) \myeq Q_{(n)}(q, t) = \frac{(q; q)_n}{(t; q)_n}P_{(n)}(q, t).
\end{gather}
For convenience, we also set $g_n(q, t) \myeq 0$, $\forall\, n < 0$.

Now we come to several important theorems on Macdonald polynomials, which will be our main tools.

\begin{thm}[{index-argument symmetry; \cite[Chapter~VI, Property~6.6]{M2}}]\label{symmetry}
Let $N\in\N$, $\lambda, \mu\in\GTp_N$, then
\begin{gather*}
\frac{P_{\lambda}\big(q^{\mu_1}t^{N-1}, q^{\mu_2}t^{N-2}, \dots, q^{\mu_N}; q, t\big)}{P_{\lambda}\big(t^{N-1}, t^{N-2}, \dots, 1; q, t\big)} = \frac{P_{\mu}\big(q^{\lambda_1}t^{N-1}, q^{\lambda_2}t^{N-2}, \dots, q^{\lambda_N}; q, t\big)}{P_{\mu}\big(t^{N-1}, t^{N-2}, \dots, 1; q, t\big)}.
\end{gather*}
\end{thm}

\begin{thm}[{evaluation identity; \cite[Chapter~VI, equations~(6.11) and~(6.11$'$)]{M2}}]\label{evaluation}
Let $N\in\N$, $\lambda\in\GTp_N$, then
\begin{gather*}
P_{\lambda}\big(t^{N-1}, t^{N-2}, \dots, 1; q, t\big) = t^{n(\lambda)}\prod_{1\leq i<j\leq N}{\frac{\big(q^{\lambda_i - \lambda_j}t^{j-i}; q\big)_{\infty}(t^{j-i+1}; q)_{\infty} }{ \big(q^{\lambda_i - \lambda_j}t^{j-i+1}; q\big)_{\infty} (t^{j-i}; q)_{\infty}}}\\
\hphantom{P_{\lambda}\big(t^{N-1}, t^{N-2}, \dots, 1; q, t\big)}{} = t^{n(\lambda)}\prod_{s\in\lambda}{\frac{1 - q^{a'(s)}t^{N - l'(s)}}{1 - q^{a(s)}t^{l(s)+1}}},
\end{gather*}
where and $n(\lambda) \myeq \lambda_2 + 2\lambda_3 + \cdots + (N-1)\lambda_N$.
The first equality holds, more generally, for any signature $\lambda\in\GT_N$ by virtue of the definition~\eqref{macdonaldsignatures11} of Macdonald Laurent polynomials $P_{\lambda}(x_1, \dots, x_N; q, t)$.
\end{thm}

From the second equality of Theorem~\ref{evaluation} and the def\/inition of dual Macdonald polynomials, we obtain

\begin{cor}\label{evaluationcor}
Let $N\in\N$, $\lambda\in\GTp_N$, then
\begin{gather*}
Q_{\lambda}\big(t^{N-1}, t^{N-2}, \dots, 1; q, t\big) = t^{n(\lambda)}\prod_{s\in\lambda}{\frac{1 - q^{a'(s)}t^{N - l'(s)}}{1 - q^{a(s)+1}t^{l(s)}}}.
\end{gather*}
\end{cor}

Since the Macdonald polynomial $P_{\lambda}(x_1, x_2, \dots, x_N; q, t)$ is symmetric in $x_1, x_2, \dots, x_N$, it is also a symmetric polynomial on $x_2, \dots, x_N$; thus it is a linear combination of Macdonald polynomials $P_{\mu}(x_2, \dots, x_N; q, t)$ with coef\/f\/icients in $F[x_1]$. More precisely, we have the so-called \textit{branching rule for Macdonald polynomials}:

\begin{thm}[{branching rule; \cite[Chapter~VI, equation~(7.13$'$), Example~2(b) on p.~342]{M2}}]\label{branchingmacdonald}
Let $N\in\N$, $\lambda\in\GTp_N$, then
\begin{gather*}
P_{\lambda}(x_1, x_2, \dots, x_N; q, t) = \sum_{\mu \in\GTp_{N-1}\colon \mu\prec\lambda}{\psi_{\lambda/\mu}(q, t)x_1^{|\lambda| - |\mu|}P_{\mu}(x_2, \dots, x_N; q, t)},
\end{gather*}
where the branching coefficients are
\begin{gather*}
\psi_{\lambda/\mu}(q, t) \myeq \prod_{1\leq i \leq j\leq N-1}\frac{\big(q^{\mu_i - \mu_j}t^{j-i+1}; q\big)_{\infty}\big(q^{\lambda_i - \lambda_{j+1}}t^{j-i+1}; q\big)_{\infty} }
{\big(q^{\lambda_i - \mu_j}t^{j-i+1}; q\big)_{\infty}\big(q^{\mu_i - \lambda_{j+1}}t^{j-i+1}; q\big)_{\infty} }\\
\hphantom{\psi_{\lambda/\mu}(q, t) \myeq \prod_{1\leq i \leq j\leq N-1}}{} \times \frac{\big(q^{\lambda_i - \mu_j+1}t^{j-i}; q\big)_{\infty}\big(q^{\mu_i - \lambda_{j+1} + 1}t^{j-i}; q\big)_{\infty}}{\big(q^{\mu_i - \mu_j+1}t^{j-i}; q\big)_{\infty}\big(q^{\lambda_i - \lambda_{j+1}+1}t^{j-i}; q\big)_{\infty}} ,
\end{gather*}
and the sum is over positive signatures $\mu\in\GTp_{N-1}$ that satisfy the interlacing constraint
\begin{gather*}
\lambda_N \leq \mu_{N-1} \leq \lambda_{N-1} \leq \cdots \leq \mu_1 \leq \lambda_1,
\end{gather*}
which is written succinctly as $\mu\prec\lambda$.
\end{thm}

Observe that $\psi_{\lambda/\mu}(q, t) > 0$, whenever $q, t \in (0, 1)$. By applying the branching rule several times, we can deduce the following.

\begin{cor}\label{cor:branchingrule}
The coefficients $c_{\lambda, \mu} = c_{\lambda, \mu}(q, t)$ in the expansion
\begin{gather*}
P_{\lambda}(q, t) = \sum_{\mu}{c_{\lambda, \mu}m_{\mu}}
\end{gather*}
are such that $c_{\lambda, \lambda} = 1$ and $c_{\lambda, \mu} \geq 0$, whenever $q, t \in (0, 1)$. Moreover, $c_{\lambda, \mu} = 0$ unless $\lambda \geq \mu$.
\end{cor}

We come to our f\/inal tool on Macdonald polynomials. It is the main theorem of~\cite{LS}, and is called the \textit{Jacobi--Trudi formula} for Macdonald polynomials. For any $n\in\N$, nonnegative integers $\tau_1, \dots, \tau_n$, variables $u_1, \dots, u_n$, def\/ine the rational functions $C_{\tau_1, \dots, \tau_n}^{(q, t)}(u_1, \dots, u_n)$ by
\begin{gather}
C_{\tau_1, \dots, \tau_n}^{(q, t)}(u_1, \dots, u_n) \myeq \prod_{k=1}^n{t^{\tau_k}\frac{(q/t; q)_{\tau_k}}{(q; q)_{\tau_k}}\frac{(qu_k; q)_{\tau_k}}{(qtu_k; q)_{\tau_k}}}\prod_{1\leq i<j\leq n}{\frac{(qu_i/tu_j; q)_{\tau_i}}{(qu_i/u_j; q)_{\tau_i}}\frac{(tu_i/(q^{\tau_i}u_j); q)_{\tau_i}}{(u_i/(q^{\tau_i}u_j); q)_{\tau_i}}}\nonumber\\
\qquad{} \times\frac{1}{\Delta(q^{\tau_1}u_1, \dots, q^{\tau_n}u_n)}\det_{1\leq i, j\leq n}{\left[ (q^{\tau_i}u_i)^{n-j}\!\left(\! 1 - t^{j-1}\frac{1 - tq^{\tau_i}u_i}{1 - q^{\tau_i}u_i}\prod_{k=1}^n{\frac{u_k - q^{\tau_i}u_i}{tu_k - q^{\tau_i}u_i}} \right) \!\right]},\!\!\!\label{Ccoeff}
\end{gather}
where $\Delta(z_1, \dots, z_n) \myeq \prod\limits_{1\leq i < j\leq n}{(z_i - z_j)} = \det\big[ z_i^{n-j} \big]_{i, j=1}^n$ is known as the Vandermonde determinant.

\begin{thm}[{Jacobi--Trudi formula; \cite[Theorem~5.1]{LS}}]\label{jacobitrudi}
Let $N\in\N$, $\lambda\in\GTp_N$, then
\begin{gather*}
Q_{\lambda}(x_1, \dots, x_N; q, t)\\
 \qquad{} = \sum_{\tau\in M^{(N)}}\left\{ \prod_{s = 1}^{N-1}{C^{(q, t)}_{\tau_{1, s+1}, \dots, \tau_{s, s+1}}\left( u_i = q^{\lambda_i - \lambda_{s+1} + \sum\limits_{j=s+2}^N{(\tau_{i, j} - \tau_{s+1, j})}}t^{s-i}\colon 1\leq i\leq s \right) }\right.\\
\left.\qquad\quad{} \times\prod_{s=1}^N{g_{\lambda_s + \tau_s^+ - \tau_s^-}(x_1, \dots, x_N; q, t)} \right\},
\end{gather*}
where $M^{(N)}$ is the set of strictly upper-triangular matrices with nonnegative entries, and for each $1\leq s\leq N$, the integers $\tau_s^+$, $\tau_s^-$, depend only on the indexing matrix $\tau$ and are defined by
\begin{gather}\label{tauplusminus}
\tau_s^+ \myeq \sum_{i=s+1}^N{\tau_{s, i}}, \qquad \tau_s^- \myeq \sum_{i=1}^{s - 1}{\tau_{i, s}}.
\end{gather}
\end{thm}

\begin{rem}
Observe that, even though $M^{(N)}$ is an inf\/inite set, the only nonvanishing terms in the sum above are those $\tau\in M^{(N)}$ such that $\lambda_s + \tau_s^+ - \tau_s^- \geq 0$ $\forall\, s = 1, \dots, N$. In other words, the sum is indexed by points of the discrete $\frac{N(N-1)}{2}$-dimensional simplex with coordinates $\{\tau_{i, j}\}_{1\leq i<j\leq N}$ satisfying
\begin{gather*}
\tau_{i, j} \geq 0,\qquad \textrm{for all} \quad 1\leq i < j\leq N,\\
\lambda_n + \sum_{i=n+1}^N{\tau_{n, i}} - \sum_{i=1}^{n - 1}{\tau_{i, n}} \geq 0, \qquad \textrm{for all} \quad n = 1, \dots, N.
\end{gather*}
\end{rem}

Let us introduce the last piece of terminology and main object of study in this paper.

\begin{df}
For any $m, N\in\N$ with $1\leq m\leq N$, $\lambda\in\GT_N$, def\/ine
\begin{gather}\label{normalmacdonald}
P_{\lambda}(x_1, \dots, x_m; N, q, t) \myeq \frac{P_{\lambda}\big(x_1, \dots, x_m, 1, t, \dots, t^{N-m-1}; q, t\big)}{P_{\lambda}\big(1, t, t^2, \dots, t^{N-1}; q, t\big)}
\end{gather}
and call $P_{\lambda}(x_1, \dots, x_m; N, q, t)$ the \textit{Macdonald unitary character of rank $N$, number of variables $m$ and parametrized by $\lambda$}.
For simplicity of terminology, we call $P_{\lambda}(x_1, \dots, x_m; N, q, t)$ a \textit{Macdonald character} rather than a Macdonald unitary character. Observe that if $q, t \in \C$ are such that $|q|, |t|\in (0, 1)$, the evaluation identity for Macdonald polynomials, Theorem~\ref{evaluation}, shows that the denominator of~\eqref{normalmacdonald} is nonzero.
\end{df}

\section{Integral formulas for Macdonald characters of one variable}\label{macdonaldintegralsec}

In this section, assume $q$ is a real number in the interval $(0, 1)$. There will also be a parameter $\theta$, typically $\theta > 0$, but we also consider cases when $\theta$ is a complex number with $\Re\theta > 0$. In either case, the parameter $t = q^{\theta}$ satisf\/ies $|t| < 1$.

\subsection{Statements of the theorems}

The simplest contour integral representation is the following, which works only when $t = q^{\theta}$, $\theta\in\N$, and involves a closed contour around f\/initely many singularities.

\begin{thm}\label{macdonaldthm1}
Let $\theta\in\N$, $t = q^{\theta}$, $N\in\N$, $\lambda\in\GT_N$ and $x\in\C \setminus \big\{0, q, q^2, \dots, q^{\theta N - 1}\big\}$. Then
\begin{gather}
\frac{P_{\lambda}\big(x, t, t^2, \dots, t^{N-1}; q, t\big)}{P_{\lambda}\big(1, t, t^2, \dots, t^{N-1}; q, t\big)} \nonumber\\
\qquad {} = \ln(1/q)\prod_{i=1}^{\theta N - 1}{\frac{1 - q^i}{x - q^i}}\frac{1}{2\pi\sqrt{-1}}\oint_{\CC_0}{\frac{x^z}{\prod\limits_{i=1}^N\prod\limits_{j=0}^{\theta -1}{\big(1 - q^{z - (\lambda_i + \theta(N-i) + j)}\big)}}{\rm d}z},\label{macdonaldthm1eqn}
\end{gather}
where $\CC_0$ is a closed, positively oriented contour enclosing the real poles $\{\lambda_i + \theta(N-i) + j \colon i = 1, \dots, N, \, j = 0, \dots, \theta - 1\}$ of the integrand. For instance, the rectangular contour with vertices $-M-r\sqrt{-1}$, $ -M+r\sqrt{-1}$, $M+r\sqrt{-1}$ and $M-r\sqrt{-1}$, for any $-\frac{2\pi}{\ln{q}} > r > 0$ and any $M > \max\{0, -\lambda_N, \, \lambda_1 + \theta N - 1\}$, is a suitable contour.
\end{thm}

The following two theorems are analytic continuations, in the variable $\theta$, of Theorem~\ref{macdonaldthm1} above.

\begin{thm}\label{macdonaldthm2}
Let $\theta > 0$, $t = q^{\theta}$, $N\in\N$, $\lambda\in\GT_N$ and $x\in\C \setminus \{0\}$, $|x|\leq 1$. The integral below converges absolutely and the equality holds
\begin{gather}
\frac{P_{\lambda}\big(xt^{N-1}, t^{N-2}, \dots, t, 1; q, t\big)}{P_{\lambda}\big(t^{N-1}, t^{N-2}, \dots, t, 1; q, t\big)}\nonumber\\
 \qquad{} = \frac{\ln q}{1 - q}\frac{\big(xt^N; q\big)_{\infty}}{(xq; q)_{\infty}}\frac{\Gamma_q(\theta N)}{2\pi\sqrt{-1}}\int_{\CC^+}{x^z\prod_{i=1}^N{\frac{\Gamma_q(\lambda_i + \theta(N-i)-z)}{\Gamma_q(\lambda_i + \theta(N-i+1)-z)}}{\rm d}z}.\label{macdonaldthm2eqn}
\end{gather}
Contour $\CC^+$ is a positively oriented contour consisting of the segment $[M+r\sqrt{-1}, M-r\sqrt{-1}]$ and the horizontal lines $[M+r\sqrt{-1}, +\infty+r\sqrt{-1})$, $[M-r\sqrt{-1}, +\infty-r\sqrt{-1})$, for some $-\frac{\pi}{2\ln{q}} > r > 0$ and $\lambda_N > M$, see Fig.~{\rm \ref{fig:C}}. Observe that $\CC^+$ encloses all real poles of the integrand $($which accumulate at~$+\infty)$ and no other poles.
\end{thm}

The reader is referred to Appendix~\ref{app:qtheory} for a reminder of the def\/inition of the $q$-Gamma function, its zeroes and poles.

\begin{thm}\label{macdonaldthm25}
Let $\theta > 0$, $t = q^{\theta}$, $N\in\N$, $\lambda\in\GT_N$ and $x\in\C$, $|x|\geq 1$. The integral below converges absolutely and the equality holds
\begin{gather}
\frac{P_{\lambda}\big(x, t, t^2, \dots, t^{N-1}; q, t\big)}{P_{\lambda}\big(1, t, t^2, \dots, t^{N-1}; q, t\big)} \nonumber\\
\qquad{} = \frac{\ln q}{q - 1}\frac{\big(x^{-1}t^N; q\big)_{\infty}}{(x^{-1}q; q)_{\infty}}\frac{\Gamma_q(\theta N)}{2\pi\sqrt{-1}}\int_{\CC^-}{x^z\prod_{i=1}^N{\frac{\Gamma_q(z -( \lambda_i - \theta i + \theta))}{\Gamma_q(z - (\lambda_i - \theta i))}}{\rm d}z}.\label{macdonaldthm25eqn}
\end{gather}
Contour $\CC^-$ is a positively oriented contour consisting of the segment $[M-r\sqrt{-1}, M+r\sqrt{-1}]$ and the horizontal lines $[M-r\sqrt{-1}, -\infty-r\sqrt{-1})$, $[M+r\sqrt{-1}, -\infty+r\sqrt{-1})$, for some $-\frac{\pi}{2\ln{q}} > r > 0$ and $M > \lambda_1$, see Fig.~{\rm \ref{fig:Cminus}}. Observe that $\CC^-$ encloses all real poles of the integrand $($which accumulate at $-\infty)$ and no other poles.
\end{thm}

\begin{figure}[t]\centering
\begin{tikzpicture}[decoration={markings,
mark=at position 1.5cm with {\arrow[line width=1pt]{>}},
mark=at position 3.2cm with {\arrow[line width=1pt]{>}},
mark=at position 4.7cm with {\arrow[line width=1pt]{>}},
mark=at position 6.5cm with {\arrow[line width=1pt]{>}},
mark=at position 8.3cm with {\arrow[line width=1pt]{>}},
mark=at position 10cm with {\arrow[line width=1pt]{>}},
mark=at position 11.8cm with {\arrow[line width=1pt]{>}}
}
]
\draw[help lines,->] (-4,0) -- (4,0) coordinate (xaxis);
\draw[help lines,->] (0,-1) -- (0,1) coordinate (yaxis);

\path[draw,line width=1pt,postaction=decorate] (4,0.5) -- (-2,0.5) -- (-2,-0.5) -- (4,-0.5);

\node[below] at (xaxis) {$\Re z$};
\node[left] at (yaxis) {$\Im z$};
\node[below left] {};
\end{tikzpicture}
\caption{Contour $\CC^+$.}\label{fig:C}
\end{figure}

\begin{figure}[t]\centering
\begin{tikzpicture}[decoration={markings,
mark=at position 1.5cm with {\arrow[line width=1pt]{>}},
mark=at position 3.2cm with {\arrow[line width=1pt]{>}},
mark=at position 5cm with {\arrow[line width=1pt]{>}},
mark=at position 6.5cm with {\arrow[line width=1pt]{>}},
mark=at position 8cm with {\arrow[line width=1pt]{>}},
mark=at position 9.8cm with {\arrow[line width=1pt]{>}},
mark=at position 11.7cm with {\arrow[line width=1pt]{>}}
}
]
\draw[help lines,->] (-4,0) -- (4,0) coordinate (xaxis);
\draw[help lines,->] (0,-1) -- (0,1) coordinate (yaxis);

\path[draw,line width=1pt,postaction=decorate] (-4,-0.5) -- (2,-0.5) -- (2,0.5) -- (-4, 0.5);

\node[below] at (xaxis) {$\Re z$};
\node[left] at (yaxis) {$\Im z$};
\node[below left] {};
\end{tikzpicture}
\caption{Contour $\CC^-$.}\label{fig:Cminus}
\end{figure}

\begin{rem}
In the formulas above, $x^z = \exp(z\ln{x})$. If $x \notin (-\infty, 0)$, we can use the principal branch of the logarithm to def\/ine $\ln{x}$, and if $x \in (-\infty, 0)$, then we can def\/ine the logarithm in the complex plane cut along $(-\sqrt{-1} \infty, 0]$ such that $\Im \ln{a} = 0$ for all $a \in (0, \infty)$.
\end{rem}

\begin{rem}
The formulas in Theorems~\ref{macdonaldthm2} and~\ref{macdonaldthm25} probably hold for more general con\-tours~$\CC^+$,~$\CC^-$, but we do not need more generality for our purposes.
\end{rem}

\begin{rem}
When $\theta = 1$, Theorem~\ref{macdonaldthm1} recovers \cite[Theorem~3.6]{GP}.
\end{rem}

\begin{rem}\label{macdonaldthetaNinteger}
The inf\/inite contours in Theorems~\ref{macdonaldthm2} and~\ref{macdonaldthm25} are needed because there are inf\/initely many real poles in the integrand and the contour needs to enclose all of them. When $\theta\in\N$, the integrands have f\/initely many poles, and we can therefore close the contours, obtaining eventually Theorem~\ref{macdonaldthm1}. More generally, if $\theta > 0$ is such that $\theta N \in\N$, a similar remark applies. In fact, we can write the product of $q$-Gamma function ratios appearing in the integrand~\eqref{macdonaldthm2eqn} as
\begin{gather}
\prod_{i=1}^N{\frac{\Gamma_q(\lambda_i + \theta(N-i)-z)}{\Gamma_q(\lambda_i + \theta(N-i+1)-z)}}\nonumber\\
 \qquad{} = \frac{\Gamma_q(\lambda_1 + \theta(N-1) - z)}{\Gamma_q(\lambda_2 + \theta(N-1)-z)}\cdots\frac{\Gamma_q(\lambda_{N-1} + \theta - z)}{\Gamma_q(\lambda_N + \theta-z)}\frac{\Gamma_q(\lambda_N - z)}{\Gamma_q(\lambda_1 + \theta N - z)},\label{ratioqgammas}
\end{gather}
and since $\Gamma_q(t+1) = \frac{1 - q^t}{1 - q}\Gamma_q(t)$, we conclude that the product above is a rational function in $q^{-z}$ with f\/initely many real poles. Thus formula~\eqref{macdonaldthm2eqn} is true if we replaced contour~$\CC^+$ by a closed contour $\CC_0$ containing all f\/initely many real poles of the integrand. Similarly, we can replace $\CC^-$ by a closed contour $\CC_0$ in~\eqref{macdonaldthm25eqn}.
\end{rem}

\subsection{An example}

Before carrying out the proofs of the theorems above in full generality, we prove some very special cases, by means of the residue theorem and the $q$-binomial formula. For simplicity, let $|x| < 1$ be a complex number, and consider the empty partition $\lambda = \varnothing$, or equivalently the $N$-signature $\lambda = \big(0^N\big) = (0, 0, \dots, 0)$. As remarked in Section~\ref{macdonaldpolysection}, we have $P_{(0^N)}(q, t) = 1$, and therefore the left-hand sides of identities~\eqref{macdonaldthm1eqn} and~\eqref{macdonaldthm2eqn} are both equal to~$1$, when $\lambda = (0^N)$. Let us prove that the right-hand sides of~\eqref{macdonaldthm1eqn} and~\eqref{macdonaldthm2eqn} also equal $1$, for $\lambda = (0^N)$.

Let us begin with the case $\theta\notin\N$, i.e., the right-hand side of~\eqref{macdonaldthm2eqn}.
Since the contour~$\CC^-$ encloses all real poles in the integrand in its interior, then the right-hand side of~\eqref{macdonaldthm2eqn} equals
\begin{gather}\label{rhsidentity2}
\frac{\ln{q}}{1 - q}\frac{\big(xt^N; q\big)_{\infty}}{(xq; q)_{\infty}}\times\Gamma_q(\theta N)\times\sum_{n = 0}^{\infty}{x^n\frac{\operatorname{Res}_{z = n}{\Gamma_q(-z)}}{\Gamma_q(\theta N - n)}}.
\end{gather}
From the def\/inition of $q$-Gamma functions, see Appendix~\ref{app:qtheory}, it is evident that, for any $n\in\Z_{\geq 0}$, we have
\begin{gather*}
\operatorname{Res}_{z = n}{\Gamma_q(-z)} = \frac{(-1)^n(1-q)^{n+1}}{\ln{q}}\frac{q^{{n+1 \choose 2}}}{(q; q)_n}.
\end{gather*}
Furthermore, $\Gamma_q(t+1) = \frac{1 - q^t}{1 - q}\Gamma_q(t)$ gives $\frac{\Gamma_q(\theta N)}{\Gamma_q(\theta N - n)} = (1 - q)^{-n}(q^{\theta N - n}; q)_n$, so~\eqref{rhsidentity2} equals
\begin{gather*}
\frac{\big(xt^N; q\big)_{\infty}}{(xq; q)_{\infty}}\times\sum_{n = 0}^{\infty}{\frac{(-1)^nq^{{n+1 \choose 2}}(q^{\theta N - n}; q)_n}{(q; q)_n}x^n}.
\end{gather*}
The latter indeed equals $1$ because of the $q$-binomial theorem, Theorem~\ref{qtheory2}, applied to $z = xq^{\theta N} = xt^N$, $a = q^{1 - \theta N}$, and the equality $q^{\theta N n}(q^{1 - \theta N}; q)_n = (-1)^nq^{{n+1 \choose 2}}(q^{\theta N - n}; q)_n$ $\forall\, n\geq 0$.

Second, let us consider the case $\theta\in\N$, i.e., the right-hand side of~\eqref{macdonaldthm1}. Observe that for $\lambda=(0^N)$, the integrand in~\eqref{macdonaldthm1} can be rewritten as $x^z\cdot\prod\limits_{i=0}^{\theta N-1}{(1-q^{z-i})^{-1}}$, whose set of poles enclosed in the interior of $\CC_0$ is $\{0, 1, 2, \dots, \theta N-1\}$. Since $\operatorname{Res}_{z=n}{(1 - q^{z-n})^{-1}} = -(\ln{q})^{-1} = (\ln{(1/q)})^{-1}$, similar considerations as above lead us to conclude that the right side of~\eqref{macdonaldthm1} is equal to the \textit{finite} sum
\begin{gather*}
\prod_{i=1}^{\theta N - 1}{\frac{1 - q^i}{x - q^i}}\times\sum_{n=0}^{\theta N-1}{\frac{x^n}{\prod\limits_{\substack{0\leq i\leq \theta N - 1 \\ i \neq n}}{(1 - q^{n-i})}}} = \prod_{i=1}^{\theta N - 1}{\frac{1 - q^{-i}}{1 - xq^{-i}}}\times\sum_{n=0}^{\theta N-1}{\frac{x^n}{ \big(q; q)_n (q^{-1}; q^{-1}\big)_{\theta N - n - 1} }}\\
\qquad{} = \frac{1}{\big(xq^{-1}; q^{-1}\big)_{\theta N - 1}}\times\sum_{n=0}^{\theta N-1}{\frac{ (-1)^n q^{-{n + 1 \choose 2}} \big(q^{-1}; q^{-1}\big)_{\theta N - 1} }{ \big(q^{-1}; q^{-1}\big)_n (q^{-1}; q^{-1})_{\theta N - n - 1} }x^n}.
\end{gather*}
The latter equals $1$, because of Corollary~\ref{qtheory25} of the $q$-binomial formula, applied to $z = xq^{-1}$ and $M = \theta N - 1$.

A simple argument involving the index stability for Macdonald polynomials, see \eqref{macdonaldsignatures1}, shows that if Theorems~\ref{macdonaldthm1} and~\ref{macdonaldthm2} hold for $\lambda\in\GT_N$, then they hold for $(\lambda_1 + n \geq \lambda_2 + n \geq \dots \geq \lambda_N + n)\in\GT_N$, and any $n\in\Z$, cf.\ Step~5 in Section~\ref{proofthm1macdonald} below. Thus the present example shows how to prove Theorems~\ref{macdonaldthm1} and~\ref{macdonaldthm2} for signatures of the form $(n, n, \dots, n)\in\GT_N$, only by use of the classical $q$-binomial theorem.

\subsection[Integral formula when $t = q^{\theta}$, $\theta\in\N$: Proof of Theorem~\ref{macdonaldthm1}]{Integral formula when $\boldsymbol{t = q^{\theta}}$, $\boldsymbol{\theta\in\N}$: Proof of Theorem~\ref{macdonaldthm1}}\label{proofthm1macdonald}

Assume $\theta\in\N$ and $t = q^{\theta}$. The proof of Theorem~\ref{macdonaldthm1} is broken down into several steps. In the f\/irst four steps, we prove the statement for positive signatures $\lambda\in\GTp_N$ (when all coordinates are nonnegative: $\lambda_1 \geq \cdots \geq \lambda_N \geq 0$), and in step 5 we extend it for all signatures $\lambda\in\GT_N$ (when some coordinates of $\lambda$ could be negative).

\textbf{Step 1.} We derive a contour integral formula for the ratio $\frac{P_{\lambda}(q^rt^{N-1}, t^{N-2}, \dots, t, 1; q, t)}{P_{\lambda}(t^{N-1}, \dots, t, 1; q, t)}$ of Macdonald polynomials, and any $r\in\N$. The index-argument symmetry, Theorem~\ref{symmetry}, applied to $\lambda = (\lambda_1, \dots, \lambda_N)$ and $\mu = (r) = (r, 0^{N-1})$, gives
\begin{gather}\label{mac2eq1}
\frac{P_{\lambda}\big(q^rt^{N-1}, t^{N-2}, \dots, t, 1; q, t\big)}{P_{\lambda}\big(t^{N-1}, \dots, t, 1; q, t\big)} = \frac{P_{(r)}\big(q^{\lambda_1}t^{N-1}, \dots, q^{\lambda_{N-1}}t, q^{\lambda_N}; q, t\big)}{P_{(r)}\big(t^{N-1}, \dots, t, 1; q, t\big)}.
\end{gather}
The denominator $P_{(r)}(t^{N-1}, \dots, t, 1; q, t)$ has a simple expression due to the evaluation identity, Theorem~\ref{evaluation}; it is particularly simple for the row partition $(r)$:
\begin{gather*}
P_{(r)}\big(t^{N-1}, \dots, t, 1; q, t\big) = \prod_{j=2}^N{\frac{\big(q^{r}t^{j-1}; q\big)_{\infty}(t^j; q)_{\infty}}{(q^rt^j; q)_{\infty}(t^{j-1}; q)_{\infty}}} = {\frac{(q^rt; q)_{\infty}(t^N; q)_{\infty}}{(q^rt^N; q)_{\infty}(t; q)_{\infty}}} = \frac{\big(t^N; q\big)_r}{(t; q)_r}.
\end{gather*}
Since we also have $P_{(r)}(q, t) = \frac{(q; q)_r}{(t; q)_r}g_r(q, t)$, see (\ref{gtoQ}), identity (\ref{mac2eq1}) becomes
\begin{gather}\label{mac2eq0}
\frac{P_{\lambda}\big(q^rt^{N-1}, t^{N-2}, \dots, t, 1; q, t\big)}{P_{\lambda}\big(t^{N-1}, \dots, t, 1; q, t\big)} = \frac{(q; q)_r}{\big(t^N; q\big)_r} g_r\big(q^{\lambda_1}t^{N-1}, \dots, q^{\lambda_N}; q, t\big).
\end{gather}

The symmetric polynomials $g_r(q, t)$, in addition to being essentially one-row Macdonald polynomials, can be def\/ined in terms of their generating function as follows, see \cite[Chapter~VI]{M2}:
\begin{gather}\label{generatingfunction}
\prod_{i=1}^N{\frac{(tx_iy; q)_{\infty}}{(x_iy; q)_{\infty}}} = \sum_{r\geq 0}g_r(x_1, \dots, x_N; q, t)y^r.
\end{gather}
The relation~(\ref{generatingfunction}) holds formally in the ring $\Lambda_F[x_1, \dots, x_N][[y]]$. If we f\/ix nonzero values $x_1, \dots, x_N\in\C\setminus\{0\}$, the identity above is an equality of real analytic functions in the domain $\{y\in\C \colon \max_{1\leq i\leq N}{|x_i y|} < 1\}$. Thus we have the following contour integral representation
\begin{gather*}
g_r(x_1, \dots, x_N; q, t) = \frac{1}{2\pi\sqrt{-1}}\oint_C{\frac{1}{y^{r+1}}\prod_{i=1}^N{\frac{(tx_iy; q)_{\infty}}{(x_iy; q)_{\infty}}{\rm d}y}},
\end{gather*}
where $C$ is any circle around the origin and radius smaller than $(\max_i{|x_i|})^{-1}$. Let $x_i = q^{\lambda_i}t^{N-i}$ for $i = 1, 2, \dots, N$ in the integral representation of $g_r(q, t)$, and replace it into the right-hand side of~(\ref{mac2eq0}); then
\begin{gather}\label{mac2eq2}
\frac{P_{\lambda}\big(q^rt^{N-1}, t^{N-2}, \dots, t, 1; q, t\big)}{P_{\lambda}\big(t^{N-1}, \dots, t, 1; q, t\big)} = \frac{(q; q)_r}{\big(t^N; q\big)_r}\frac{1}{2\pi\sqrt{-1}}\oint_C{\frac{1}{y^{r+1}}\prod_{i=1}^N{\frac{\big(q^{\lambda_i}t^{N-i+1}y; q\big)_{\infty}}{\big(q^{\lambda_i}t^{N-i}y; q\big)_{\infty}}{\rm d}y}},
\end{gather}
where $C$ can be taken to be any circle around the origin of radius smaller than $1$ (we need here that $\lambda\in\GTp_N$ implies $q^{\lambda_i}t^{N-i} < 1$ $\forall\, i$). For $t = q^{\theta}$, $\theta\in\N$, we can simplify~(\ref{mac2eq2}) to
\begin{gather}\label{mac2eq3}
\frac{P_{\lambda}\big(q^rt^{N-1}, t^{N-2}, \dots, t, 1; q, t\big)}{P_{\lambda}\big(t^{N-1}, \dots, t, 1; q, t\big)} = \frac{(q; q)_r}{\big(q^{\theta N}; q\big)_r}\frac{1}{2\pi\sqrt{-1}}\oint_C
{\frac{y^{-(r+1)} {\rm d}y}{\prod\limits_{i=1}^N{\prod\limits_{j=0}^{\theta-1}{(1 - q^{\lambda_i + \theta(N - i) + j}y)}}}}.
\end{gather}

\textbf{Step 2.} We obtain a new contour integral representation by modifying~(\ref{mac2eq3}). The resulting contour integral representation involves an open contour~$\CC^+$, that looks like that of Fig.~\ref{fig:C} (but has a slight dif\/ference from that in Theorem~\ref{macdonaldthm2}).

Observe that the absolute value of the integrand in~(\ref{mac2eq3}) is of order $o(R^{-r-1}) = o(R^{-2})$, if $|y| = R$ is large. An application of Cauchy's theorem yields that the value of the integral is unchanged if the closed contour $C$ is deformed into the ``keyhole'' contour $\CC'$ shown in Fig.~\ref{fig:Cprime}.
Let us describe the contour $\CC'$ in words: it is a positively oriented contour, formed by two lines away from the origin, of arguments $\pm 3\pi/4$, and the portion of a semicircle of some radius $0 < \delta < 1$. Evidently, the straight lines are part of the level lines $\Im(\ln(y)) = \pm 3\pi/4$, while the portion of the semicircle is part of the level line $\Re(\ln(y)) = \delta$ (where $\ln$ is def\/ined in $\C \setminus (-\infty, 0]$).

\begin{figure}[t]\centering
\resizebox{7cm}{6cm}{
\begin{tikzpicture}[decoration={markings,
mark=at position 1.1cm with {\arrow[line width=1pt]{>}},
mark=at position 2.6cm with {\arrow[line width=1pt]{>}},
mark=at position 5cm with {\arrow[line width=1pt]{>}},
mark=at position 6.4cm with {\arrow[line width=1pt]{>}},
mark=at position 8.5cm with {\arrow[line width=1pt]{>}},
mark=at position 10cm with {\arrow[line width=1pt]{>}}
}
]
\draw[help lines,->] (-4,0) -- (4,0) coordinate (xaxis);
\draw[help lines,->] (0,-3.5) -- (0,3.5) coordinate (yaxis);

\path[draw,line width=1pt,postaction=decorate] (-3,-3) -- (-0.707,-0.707) arc (-135:135:1) -- (-3,3);

\node[below] at (xaxis) {$\Re z$};
\node[left] at (yaxis) {$\Im z$};
\node[below left] {$O$};
\end{tikzpicture}
}
\caption{Contour $\CC'$.}\label{fig:Cprime}
\end{figure}

Next, we make the change of variables $y = q^{-z}$, or $z = -\frac{\ln{y}}{\ln{q}}$, where $\ln$ is def\/ined on its principal branch. Based on the previous observations about the contour $\CC'$ being composed by level lines, it is clear that the resulting contour for $z$ is a \textit{negatively oriented} contour formed by one segment and two straight lines, but we can easily reverse the orientation of the contour at the cost of switching signs. Let us call the positively oriented contour $\CC^+$, see Fig.~\ref{fig:C}; the integral formula becomes
\begin{gather}\label{mac2eq4}
\frac{P_{\lambda}\big(q^rt^{N-1}, t^{N-2}, \dots, t, 1; q, t\big)}{P_{\lambda}\big(t^{N-1}, \dots, t, 1; q, t\big)} = \frac{(q; q)_r}{\big(q^{\theta N}; q\big)_r}\frac{1}{2\pi\sqrt{-1}}\int_{\CC^+}\!
{\frac{q^{(r+1)z} q^{-z}\ln{q} {\rm d}z}{\prod\limits_{i=1}^N{\prod\limits_{j=0}^{\theta-1}{\big(1 - q^{\lambda_i + \theta(N - i) + j- z}\big)}}}}.\!\!\!\!
\end{gather}
Note that points in the horizontal lines of contour $\CC^+$ in~\eqref{mac2eq4} have imaginary parts $\pm\frac{3\pi}{4\ln{q}}$, while points in the vertical segment of $\CC^+$ have real part $-\ln{\delta}/\ln{q} < 0$. Thus contour~$\CC^+$ encloses exactly all the real poles of the integrand in~\eqref{mac2eq4} and no other poles (it also encloses the origin, though this not important).

\textbf{Step 3.} We make some f\/inal modif\/ications to formula~\eqref{mac2eq4}; the resulting contour integral representation will include a closed contour $\CC_0$ as in the statement of the theorem.

Observe that the integrand in~(\ref{mac2eq4}) has f\/initely many real poles and they are all enclosed by contour $\CC^+$; also the integrand is exponentially small as $|z|\rightarrow\infty$ along the contour $\CC^+$. Therefore, as an application of Cauchy's theorem, we can replace $\CC^+$ by a closed contour $\CC_0$ that encloses all f\/initely many real poles of the integrand. Also note that for $r > \theta N$, we can write $\frac{(q; q)_r}{(q^{\theta N}; q)_r} = \prod\limits_{i=1}^{\theta N - 1}{\frac{1 - q^i}{1 - q^{r+i}}}$. Thus for $r > \theta N$, equation~\eqref{mac2eq4} can be rewritten as
\begin{gather}
\frac{P_{\lambda}\big(q^rt^{N-1}, t^{N-2}, \dots, t, 1; q, t\big)}{P_{\lambda}\big(t^{N-1}, \dots, t, 1; q, t\big)}\nonumber\\
 \qquad{} = \frac{\ln{q}}{2\pi\sqrt{-1}} \prod_{i=1}^{\theta N - 1}{\frac{1 - q^i}{1 - q^{r+i}}}\oint_{\CC_0}{q^{rz}\prod_{i=1}^N{\prod_{j=0}^{\theta - 1}{\frac{1}{1 - q^{\lambda_i + \theta(N - i) + j - z}}}{\rm d}z}}.
\label{mac2eq5}\end{gather}
We let $x = q^r t^N = q^{r + \theta N}$ and replace all instances of $q^r$ in~(\ref{mac2eq5}) above by $x/t^N$; then multiply both sides of the identity by $t^{|\lambda|}$. We claim that the resulting equation is exactly equality~(\ref{macdonaldthm1eqn}). In fact, the left-hand side of our equation is
\begin{gather*}
t^{|\lambda|}\frac{P_{\lambda}\big(x/t, t^{N-2}, \dots, t, 1; q, t\big)}{P_{\lambda}\big(t^{N-1}, \dots, t, 1; q, t\big)} = \frac{P_{\lambda}\big(x, t^{N-1}, t^{N-2}, \dots, t; q, t\big)}{P_{\lambda}\big(t^{N-1}, t^{N-2}, \dots, 1; q, t\big)} = \frac{P_{\lambda}\big(x, t, t^2, \dots, t^{N-1}; q, t\big)}{P_{\lambda}\big(1, t, t^2, \dots, t^{N-1}; q, t\big)},
\end{gather*}
by homogeneity of the Macdonald polynomials. On the other hand, the right-hand side of our equation is
\begin{gather*}
t^{|\lambda|}\frac{\ln{q}}{2\pi\sqrt{-1}} \prod_{i=1}^{\theta N - 1}{\frac{1 - q^i}{1 - xq^{i - \theta N}}}\oint_{\CC_0}{x^z t^{-Nz}\prod_{i=1}^N{\prod_{j=0}^{\theta - 1}{\frac{1}{1 - q^{\lambda_i + \theta(N - i) + j - z}}}{\rm d}z}},
\end{gather*}
which can be shown to be equal to the right-hand side of~(\ref{macdonaldthm1eqn}), by simple algebraic manipulations.

The conclusion is that we have proved identity~(\ref{macdonaldthm1eqn}) for all $x = q^m$ with $m\in\N$ large enough (to be precise, $m$ is of the form $r + \theta N$ and $r > \theta N$, so the statement was proved for all integers $m > 2\theta N$).

\textbf{Step 4.} Still assuming $\lambda\in\GTp_N$, we prove Theorem~\ref{macdonaldthm1} for all $x\in\C \setminus (\{0\}\cup\{q^i \colon i = 1, 2, \dots, \theta N - 1\})$. Observe that $x\neq 0$ is imposed to make sense of the term $x^z = \exp(z\ln{x})$ and $x\notin\{q, q^2, \dots, q^{\theta N - 1}\}$ is necessary so that the denominator of the right-hand side of~(\ref{macdonaldthm1eqn}) is nonzero.

We claim that both sides of~(\ref{macdonaldthm1eqn}) are rational functions of $x$. This would prove the desired result, given that we have shown in step~3 above that~(\ref{macdonaldthm1eqn}) holds for inf\/initely many points $q^m$, where $m$ is large enough. The left-hand side of~(\ref{macdonaldthm1eqn}) is obviously a polynomial on $x$. In the right-hand side of~(\ref{macdonaldthm1eqn}), we have the product $\prod\limits_{i=1}^{\theta N - 1}{(x - q^i)^{-1}}$ which is a rational function. We only need to check that the contour integral is also a rational function on~$x$. In fact, this follows from the residue theorem and the fact that there are f\/initely many poles in the interior of~$\CC_0$, all of these being simple and integral. The fact that the poles considered above are simple and integral can be easily checked and is equivalent to fact that all the values $\lambda_i + \theta(N - i) + j$, for $1\leq i\leq N$, $0\leq j\leq \theta - 1$, are pairwise distinct integers.

\textbf{Step 5.} We extend equality~(\ref{macdonaldthm1eqn}) to all signatures $\lambda\in\GT_N$.

Let $\lambda\in\GT_N$ be arbitrary and we aim to prove~(\ref{macdonaldthm1eqn}). If $\lambda\in\GTp_N$, the result is already proved in the f\/irst four steps above. Otherwise, choose $m\in\N$ such that $\lambda_N + m \geq 0$, and so $\widetilde{\lambda} \myeq \lambda + \big(m^N\big) = (\lambda_1 + m, \lambda_2 + m, \dots, \lambda_N + m)\in\GTp_N$.

By the index stability~(\ref{macdonaldsignatures1}), $(x_1\cdots x_N)^m P_{\lambda}(x_1, \dots, x_N; q, t) = P_{\widetilde{\lambda}}(x_1, \dots, x_N; q, t)$. Thus multiplying the left-hand side of the equality~(\ref{macdonaldthm1eqn}) by $x^m = \big(x\cdot t\cdot t^2\cdots t^{N-1}\big)^m(1\cdot t\cdot t^2\cdots t^{N-1})^{-m}$ gives
\begin{gather*}
\frac{\big(x\cdot t\cdot t^2\cdots t^{N-1}\big)^m \cdot P_{\lambda}\big(x, t, t^2, \dots, t^{N-1}; q, t\big)}{\big(1\cdot t\cdot t^2\cdots t^{N-1}\big)^m \cdot P_{\lambda}\big(1, t, t^2, \dots, t^{N-1}; q, t\big)} = \frac{P_{\widetilde{\lambda}}\big(x, t, t^2, \dots, t^{N-1}; q, t\big)}{P_{\widetilde{\lambda}}\big(1, t, t^2, \dots, t^{N-1}; q, t\big)}.
\end{gather*}
If we multiply the right-hand side of equality~(\ref{macdonaldthm1eqn}) by $x^m$, and also make the change of variables $z\mapsto z-m$ in the contour integral, we obtain
\begin{gather*}
\ln(1/q)\prod_{i=1}^{\theta N - 1}{\frac{1 - q^i}{x - q^i}}\frac{1}{2\pi\sqrt{-1}}\oint_{\CC_0}{\frac{x^z}{\prod\limits_{i=1}^N\prod\limits_{j=0}^{\theta -1}{\big(1 - q^{z - (\widetilde{\lambda}_i + \theta(N-i) + j)}\big)}}{\rm d}z}.
\end{gather*}
Since we have proved equality~(\ref{macdonaldthm1eqn}) for the positive signature $\widetilde{\lambda}$ already, then~(\ref{macdonaldthm1eqn}) also holds for~$\lambda$, since we have multiplied both sides by $x^m$ and obtained equal expressions.

\subsection[Integral formula for general $q$,~$t$ I: Outline of proof of Theorems~\ref{macdonaldthm2} and \ref{macdonaldthm25}]{Integral formula for general $\boldsymbol{q}$,~$\boldsymbol{t}$ I:\\ Outline of proof of Theorems~\ref{macdonaldthm2} and \ref{macdonaldthm25}}

Theorems~\ref{macdonaldthm2} and~\ref{macdonaldthm25} are analytic continuations of Theorem~\ref{macdonaldthm1}, with respect to the variable~$\theta$.
We shall need the weak version of Carlson's lemma below, which is proved in \cite[Theorem~2.8.1]{AAR} (in this reference, the statement is given for functions that are analytic and bounded on $\{z\in\C\colon \Re z \geq 0\}$, but a simple change of variables $z\mapsto z+M$, for some large positive integer $M\in\N$ leads us to the version below).

\begin{lem}[Carlson's lemma]\label{carlsonlemma}
Let $f(z)$ be a holomorphic function on the right half plane $\{z\in\C\colon \Re z > 0\}$, such that $f(z)$ is uniformly bounded on the domain $\{z\in\C \colon \Re z \geq M\}$, for some $M>0$. If $f(n) = 0$ for all $n\in\N$, then $f$ is identically zero.
\end{lem}

We outline the steps of the proof of Theorem~\ref{macdonaldthm2} below, and we shall carry out the detailed proof of each step in the next section. Theorem~\ref{macdonaldthm25} can be proved analogously and we leave the details to the reader.
\begin{itemize}\itemsep=0pt
	\item \textbf{Step 0.} Prove that the right-hand side of~(\ref{macdonaldthm2eqn}) is well-def\/ined, i.e., the contour integral is absolutely convergent. We prove the absolute convergence for all $|x|\leq 1$, $x\neq 0$, and $\theta>0$.
	\item \textbf{Step 1.} Reduce the general statement to the case $|x| < 1$, $x\neq 0$, and $\lambda\in\GTp_N$. Assume the latter conditions are in place for the remaining steps in the proof.
	\item \textbf{Step 2.} Prove that the equation~(\ref{macdonaldthm2eqn}) of Theorem~\ref{macdonaldthm2} holds when $t = q^{\theta}$ and $\theta\in\N$.
	\item \textbf{Step 3.} Prove that both sides of equation~(\ref{macdonaldthm2eqn}) are holomorphic functions of $\theta$ on the right half plane $\{\theta\in\C\colon \Re\theta > 0\}$.
	\item \textbf{Step 4.} Prove that, for some (suf\/f\/iciently large) $M > 0$, both sides of equation~(\ref{macdonaldthm2eqn}) are uniformly bounded functions of $\theta$ on $\{\theta\in\C\colon \Re\theta \geq M\}$.
\end{itemize}

Then we can conclude Theorem~\ref{macdonaldthm2} as follows. From steps 2--4 and Lemma~\ref{carlsonlemma} above, identity~(\ref{macdonaldthm2eqn}) is proved for all $x\in\C$ with $|x| < 1$, $x\neq 0$, $\lambda\in\GTp_N$, $\theta\in\{z\in\C \colon \Re z > 0\}$, and in particular for all $\theta > 0$. Thus step 1 shows the theorem holds for the more general case $x\in\C \setminus \{0\}$, $|x|\leq 1$, and $\lambda\in\GT_N$.

\subsection[Integral formula for general $q$,~$t$ II: Proof of Theorem \ref{macdonaldthm2}]{Integral formula for general $\boldsymbol{q}$,~$\boldsymbol{t}$ II: Proof of Theorem \ref{macdonaldthm2}}

In this section, we prove Theorem~\ref{macdonaldthm2}; as we already mentioned, Theorem~\ref{macdonaldthm25} can be proved in an analogous way, and we leave the details to the reader. From the last paragraph of the previous section, Theorem~\ref{macdonaldthm2} will be proved if we verify steps 0--4 stated there.

\textbf{Proof of Step 0.} Fix $\theta>0$, $x\in\C \setminus \{0\}$, $|x|\leq 1$, and let
\begin{gather*}
F_q(z; \theta) \myeq x^z\prod\limits_{i=1}^N{\frac{\Gamma_q(\lambda_i + \theta(N-i) - z)}{\Gamma_q(\lambda_i + \theta(N-i+1) - z)}}.
 \end{gather*} From the def\/inition of the $q$-Gamma function, $F_q(z; \theta)$ is a holomorphic function in a neighborhood of the contour $\CC^+$, thus $\int_{\CC^+}{F_q(z; \theta){\rm d}z}$ is a well-posed integral. To prove its absolute convergence, it suf\/f\/ices to show that $|F_q(z; \theta)| \leq c_1\cdot q^{c_2 \Re z}$, for some $c_1, c_2 > 0$ and all $z\in\CC^+$ with $\Re z$ large enough.

Because $|x|\leq 1$ and all $z$ in the contour $\CC^+$ have bounded imaginary parts, we have $|x^z| = \exp(\Re z \ln{|x|} - \Im z \arg(x))$, $z\in\CC^+$, is upper-bounded by a constant. Thus it will suf\/f\/ice to show
\begin{gather*}
\left|\frac{\Gamma_q(\lambda_i + \theta(N - i) - z)}{\Gamma_q(\lambda_i + \theta(N - i + 1) - z)}\right| \leq c_1\cdot q^{c_2 \Re z}, \qquad i = 1, 2, \dots, N,
\end{gather*}
for some $c_1, c_2 > 0$ and all $z\in\CC^+$ with $\Re z$ large enough. Since each $\lambda_i + \theta(N-i)$ is real and~$|\Im z|$ is constant for $z\in\CC^+$, when $\Re z$ is large enough, the statement reduces to showing
\begin{gather}\label{toproveineqgamma}
\left|\frac{\Gamma_q(- z)}{\Gamma_q(\theta - z)}\right| \leq c_1\cdot q^{c_2 \Re z},
\end{gather}
for some $c_1, c_2 > 0$ and all $z\in\CC^+$ with $\Re z$ large enough.

Let $\pi/2 > d>0$ be the value such that $|\Im z|$, for any point $z\in\CC^+$ with $\Re z$ large enough, equals $-d/\ln{q}$ (such $d$ exists by our assumptions on the contour $\CC^+$). There exists a real number $a > 0$ large enough so that the following inequality holds
\begin{gather}\label{adef}
q^a\big(1 + q^{\theta}\big) \leq 2 \cos d.
\end{gather}
Now we consider only points $z\in\CC^+$ with $\Re z$ large enough so that $|\Im z| = -d/\ln{q}$ and $\Re z > \theta + a + 1$; let $M = M(z) = \lfloor \Re z - \theta - a \rfloor \in\N$. From the def\/inition of~$a$ in~(\ref{adef}), the restriction on the values of $\Re z$, and $q\in (0, 1)$, one can easily obtain
\begin{gather}\label{consequenceadef}
\big|1 - q^{z - \theta - j}\big| \leq \big|1 - q^{z - j}\big|, \qquad \forall\, j = 0, 1, \dots, M.
\end{gather}
As a consequence of inequality~(\ref{consequenceadef}) and the def\/inition of the $q$-Gamma function, we have
\begin{gather*}
\left| \frac{\Gamma_q(-z)}{\Gamma_q(\theta - z)} \right| = \left| (1 - q)^{\theta}\frac{\big(q^{-z+\theta}; q\big)_{\infty}}{(q^{-z}; q)_{\infty}} \right|\\
 \qquad{} = \left| (1 - q)^{\theta}\frac{\big(1 - q^{\theta - z}\big)\cdots \big(1 - q^{\theta - z + M}\big)}{(1 - q^{-z})\cdots\big(1 - q^{-z+M}\big)}\frac{\big(q^{-z+\theta+M+1}; q\big)_{\infty}}{\big(q^{-z+M+1}; q\big)_{\infty}} \right|\\
\qquad{} = \left| (1 - q)^{\theta}\frac{\big(1 - q^{z - \theta}\big)\cdots\big(1 - q^{z - \theta - M}\big)}{(1 - q^z)\cdots \big(1 - q^{z-M}\big)}q^{\theta(M+1)}\frac{\big(q^{\theta - z + M + 1}; q\big)_{\infty}}{\big(q^{-z+M+1}; q\big)_{\infty}} \right|\\
\qquad{} \leq (1 - q)^{\theta}\cdot q^{\theta(M+1)} \left| \frac{\big(q^{\theta-z+M+1}; q\big)_{\infty}}{\big(q^{-z+M+1}; q\big)_{\infty}} \right|.
\end{gather*}
Since $M = \lfloor \Re z - \theta - a\rfloor$, the previous inequality almost shows~(\ref{toproveineqgamma}).
We are left to show that the absolute value of $(q^{\theta-z+M+1}; q)_{\infty}/(q^{-z+M+1}; q)_{\infty}$ is upper bounded by a constant independent of $z\in\CC^+$ as long as $\Re z$ is large enough.
In fact, we have $\Re z - \theta - a\geq M > \Re z - \theta - a - 1$, so~$|q^{\theta-z+M+1}|, |q^{-z+M+1}|\leq q^{-\theta-a}$ and $|\Im q^{\theta - z + M + 1}|, |\Im q^{-z+M+1}| \geq q^{\theta + M + 1 - \Re z}\sin d \geq q^{1-a}\sin d > 0$. Because the $q$-Pochhammer symbol $(x; q)_{\infty}$ is holomorphic on $x\in\C$, in particular continuous, the absolute value $|(x; q)_{\infty}|$ attains a maximum $c_{\max}\geq 0$ and a minimum value \smash{$c_{\min}\geq 0$} on the compact subset $\{x\in\C\colon |x| \leq q^{-\theta-a}, |\Im x| \geq q^{1-a}\sin d\}$, and moreover $c_{\min} > 0$ because all the roots of $(x; q)_{\infty} = 0$ are real. We have thus proved $|(q^{\theta-z+M+1}; q)_{\infty}|/|(q^{-z+M+1}; q)_{\infty}| \leq c_{\max}/c_{\min} < \infty$, as desired.

\textbf{Proof of Step 1.} Assume we proved identity~(\ref{macdonaldthm2eqn}) when $x\in\C \setminus \{0\}$, $|x| < 1$. Then an easy application of the dominated convergence theorem (we know the integral converges absolutely when $x = 1$ because of step 1) shows the equation would also hold for all $x\in\C \setminus \{0\}$, $|x|\leq 1$.

Also, observe that step 5 of the proof of Theorem~\ref{macdonaldthm1} can be repeated almost word-by-word to extend the theorem to all $\lambda\in\GT_N$, assuming that it was proved for all $\lambda\in\GTp_N$. Therefore we will assume for convenience in the next steps that $|x| < 1$, $x\neq 0$, $\lambda\in\GTp_N$, and prove the theorem only in that case, without loss of generality.

\textbf{Proof of Step 2.} In this step, we consider the case $\theta\in\N$. Since we are also assuming $|x| < 1$, $x\neq 0$, then clearly $xt^N\notin\big\{0, q, q^2, \dots, q^{\theta N - 1}\big\}$ and therefore equality~(\ref{macdonaldthm1eqn}) holds if $x$ was replaced with $xt^N$. After multiplying both sides of the resulting equation by $t^{-|\lambda|}$, we claim that we arrive at the desired~(\ref{macdonaldthm2eqn}) with $\CC_0$ in place of $\CC^+$. In fact, the left-hand side of our equation is
\begin{gather*}
t^{-|\lambda|}\frac{P_{\lambda}\big(xt^N, t, t^2, \dots, t^{N-1}; q, t\big)}{P_{\lambda}\big(1, t, t^2, \dots, t^{N-1}; q, t\big)} \\
\qquad{} = \frac{P_{\lambda}\big(xt^{N-1}, 1, t, \dots, t^{N-2}; q, t\big)}{P_{\lambda}\big(1, t, t^2, \dots, t^{N-1}; q, t\big)} = \frac{P_{\lambda}\big(xt^{N-1}, t^{N-2}, \dots, t, 1; q, t\big)}{P_{\lambda}\big(t^{N-1}, t^{N-2}, \dots, t, 1; q, t\big)}
\end{gather*}
thanks to the homogeneity and symmetry of Macdonald polynomials. On the other hand, the right-hand side of our equation is
\begin{gather*}
t^{-|\lambda|}\ln(1/q)\prod_{i=1}^{\theta N - 1}{\frac{1 - q^i}{xq^{\theta N} - q^i}}\frac{1}{2\pi\sqrt{-1}}\oint_{\CC_0}{\frac{x^zt^{Nz}}{\prod\limits_{i=1}^N\prod\limits_{j=0}^{\theta -1}{\big(1 - q^{z - (\lambda_i + \theta(N-i) + j)}\big)}}{\rm d}z},
\end{gather*}
which can be shown equal to the right-hand side of~(\ref{macdonaldthm2eqn}) (with $\CC^+$ replaced by $\CC_0$) after simple algebraic manipulations. Finally observe that contour~$\CC_0$ can be replaced by~$\CC^+$ by an application of Cauchy's theorem.

\textbf{Proof of Step 3.} Let us begin by proving holomorphicity of the left-hand side of~(\ref{macdonaldthm2eqn}) with respect to the variable $\theta$; observe that $\theta$ only appears inside the variable $t = q^{\theta}$. The Macdonald polynomials $P_{\lambda}(x_1, \dots, x_N; q, t)$ are holomorphic functions of $\theta$ on $\{\theta\in\C\colon \Re\theta > 0\}$ because all the branching coef\/f\/icients $\psi_{\mu/\nu}(q, t)$ are holomorphic on this domain. Then $P_{\lambda}\big(xt^{N-1}, t^{N-2}, \dots, t, 1; q, t\big)$ and $P_{\lambda}\big(t^{N-1}, t^{N-2}, \dots, t, 1; q, t\big)$ are also holomorphic. It follows that the ratio of these two quantities is holomorphic if we proved that the denominator $P_{\lambda}\big(t^{N-1}, t^{N-2}, \dots, t, 1; q, t\big)$ never vanishes for $\Re\theta > 0$, or equivalently for $|t| < 1$; this is evident from the evaluation identity for Macdonald polynomials, Theorem~\ref{evaluation}.

Next we prove holomorphicity of the right-hand side of~(\ref{macdonaldthm2eqn}) as a function of $\theta$ in the domain $\{\theta\in\C\colon \Re\theta > 0\}$. Clearly $\big(xq^{\theta N}; q\big)_{\infty}\Gamma_q(\theta N)$ is holomorphic in the given domain of $\theta$, but it is less clear that the integral $\int_{\CC^+}{F_{q}(z, \theta){\rm d}z}$ is holomorphic in the right half-plane, where
\begin{gather*}
F_q(z; \theta) = x^z\prod\limits_{i=1}^N{\frac{\Gamma_q(\lambda_i + \theta(N-i) - z)}{\Gamma_q(\lambda_i + \theta(N-i+1) - z)}}.
\end{gather*}

First we claim that $F_q(z; \theta)$ is holomorphic on $U\times \{\theta\in\C\colon \Re\theta > 0\}$, for some neighborhood~$U$ of~$\CC^+$. Indeed, the factor $x^z$ is clearly entire on $z$, and does not depend on $\theta$. We can write the product of ratios of $q$-Gamma functions in the def\/inition of $F_q(z; \theta)$, as we did in Remark~\ref{macdonaldthetaNinteger}, see~(\ref{ratioqgammas}). For $1\leq i \leq N-1$, we have
\begin{gather*}
\frac{\Gamma_q(\lambda_i + \theta(N - i) - z)}{\Gamma_q(\lambda_{i+1} + \theta(N - i) - z)} =
\prod_{n = \lambda_{i+1}}^{\lambda_i - 1}{[n + \theta(N - i) - z]_q} =
\prod_{n = \lambda_{i+1}}^{\lambda_i - 1}{\frac{1 - q^{n + \theta(N - i) - z}}{1 - q}},
\end{gather*}
which is clearly holomorphic on $(z, \theta)\in\C^2$. Finally the remaining factor can be written as
\begin{gather*}
\frac{\Gamma_q(\lambda_N - z)}{\Gamma_q(\lambda_1 + \theta N - z)} = (1 - q)^{\lambda_1 + \theta N - \lambda_N}\frac{\big(q^{\lambda_1 + \theta N - z}; q\big)_{\infty}}{\big(q^{\lambda_N - z}; q\big)_{\infty}}.
\end{gather*}
It is clear that $(1 - q)^{\lambda_1 + \theta N - \lambda_N}\big(q^{\lambda_1 + \theta N - z}; q\big)_{\infty}$ is holomorphic on $(z, \theta)\in\C^2$. And also there is a neighborhood $U$ of $\CC^+$ on which the function $\big(q^{\lambda_N - z}; q\big)_{\infty}^{-1}$ of $z$ is holomorphic on~$U$.

Secondly, we claim that $\int_{\CC^+}{F_q(z; \theta){\rm d}z}$ is absolutely convergent and moreover $\int_{\CC^+}{|F_q(z; \theta)|{\rm d}z}$ is uniformly bounded on compact subsets of $\{\theta\in\C\colon \Re\theta > 0\}$; this will be a consequence of the stronger statement

\begin{claim}\label{bound1}
Consider any compact subset $K\subset\{\theta\in\C\colon \Re\theta > 0\}$. There exists a constant $M_1> 0$, depending on $K$, such that
\begin{gather}\label{toproveineq1}
|F_q(z; \theta)| < M_1 |x^z|, \qquad \forall\, z\in\CC^+,\quad \theta\in K.
\end{gather}
\end{claim}

Let us f\/irst conclude the proof of step 3 from the claim above.

Since $|x| < 1$ we have that $|x^z|$ decreases exponentially as $|z|\rightarrow\infty$, $z\in\CC^+$. Then inequality~(\ref{toproveineq1}) shows that $\int_{\CC^+}{F_q(z; \theta){\rm d}z}$ is absolutely convergent and for all $\theta$ belonging to the compact subset $K\subset\{\theta\in\C\colon \Re\theta > 0\}$, we have the bound $\int_{\CC^+}{|F_q(z; \theta)|{\rm d}z} < C(K)$, for some constant $C(K) > 0$ that depends on~$K$.

Let $T$ be any triangular contour belonging to $\{\theta\in\C\colon \Re\theta > 0\}$. Then
\begin{gather*}
\int_T{\int_{\CC^+}{|F_q(z; \theta)|{\rm d}z}{\rm d}\theta} \leq C(T)\int_T{|{\rm d}\theta|} < \infty.
 \end{gather*}
 By Fubini's and Cauchy's theorems, we have \begin{gather*}
 \int_T{\int_{\CC^+}{F_q(z; \theta){\rm d}z}{\rm d}\theta} = \int_{\CC^+}\int_{T}{{F_q(z; \theta){\rm d}\theta}{\rm d}z} = 0.
 \end{gather*}
Morera's theorem implies that $\int_{\CC^+}{F_q(z; \theta){\rm d}z}$ is holomorphic on $\{\theta\in\C\colon \Re\theta > 0\}$, concluding step~3.

\begin{proof}[Proof of Claim~\ref{bound1}]
Since we have shown before that $F_q(z; \theta)$ is holomorphic on $U\times\{\theta\in\C\colon \Re\theta > 0\}$, then it suf\/f\/ices to prove inequality~(\ref{toproveineq1}) for all $z\in\CC^+$, $\Re z > M_2$ and all $\theta\in K$, where $M_2$ is an arbitrarily large positive constant.

We can express the product of $q$-Gamma function ratios in the def\/inition of $F_q(z; \theta)$ as we did in~(\ref{ratioqgammas}). The last ratio is
\begin{gather*}
\frac{\Gamma_q(\lambda_N - z)}{\Gamma_q(\lambda_1 + \theta N - z)} = \frac{[-z + \lambda_N - 1]_q\cdots [-z+1]_q[-z]_q}{[-z + \lambda_1 + \theta N - 1]_q\cdots [-z + \theta N + 1]_q[-z + \theta N]_q}\frac{\Gamma_q(-z)}{\Gamma_q(\theta N - z)},
\end{gather*}
because we are assuming $\lambda\in\GTp_N$ (and so $\lambda_1 \geq \lambda_N \geq 0$). Plugging the equality above into~(\ref{ratioqgammas}), we obtain
\begin{gather}
\prod_{i=1}^N{\frac{\Gamma_q(\lambda_i + \theta(N - i) - z)}{\Gamma_q(\lambda_i + \theta(N - i + 1) - z) }} \nonumber\\
\qquad{} = \prod_{i, j}{ [\theta (N - i) + j - z]_q }\frac{[\lambda_N - z - 1]_q\cdots [-z]_q}{[\lambda_1 + \theta N - z - 1]_q\cdots [\theta N - z]_q}\label{line1pfclaim}\\
\qquad\quad{} \times\frac{\Gamma_q(-z)}{\Gamma_q(\theta N - z)},\label{line2pfclaim}
\end{gather}
where the product in~(\ref{line1pfclaim}) is over $1 \leq i \leq N-1$ and $\lambda_{i+1} \leq j < \lambda_i$.
The term~(\ref{line1pfclaim}) is of the form $P(q^{-z})/Q(q^{-z})$, where $P, Q\in\C(q, q^{\theta})[x]$ are both polynomials of degree~$\lambda_1$. It follows that there exist constants $M_1', M_2' > 0$ such that $\Re z > M_2'$ implies $|P(q^{-z})|/|Q(q^{-z})| < M_1'$.

Thus we are left to deal with~(\ref{line2pfclaim}), which by def\/inition of the $q$-Gamma function equals $(1 - q)^{\theta N}\frac{(q^{\theta N - z}; q)_{\infty}}{(q^{-z}; q)_{\infty}}$. For any $\theta$ with $\Re\theta > 0$, we have $|(1 - q)^{\theta}| < 1$. Thus we only need to prove the existence of $M_1'', M_2'' > 0$ such that $z\in\CC^+$, $\Re z > M_2''$, implies
\begin{gather}\label{boundthetaN}
\left| \frac{\big(q^{\theta N - z}; q\big)_{\infty}}{(q^{-z}; q)_{\infty}} \right| < M_1'',\qquad \textrm{for all} \quad \theta\in K.
\end{gather}
There exists $a > 0$ large enough such that
\begin{gather}\label{conditiona}
q^{-a} \big( 1 - q^{N\inf_{\theta\in K}{\Re \theta}} \big) > 2.
\end{gather}
Note that such $a$ exists because $K \subset \{\theta\in\C\colon \Re\theta > 0\}$ is compact and so $\inf_{\theta\in K}{\Re \theta} > 0$. Now consider only $z\in\CC^+$ with $\Re z > a+1$, and let $M = M(z) \myeq \lfloor \Re z - a \rfloor\in\N$. From the triangle inequality, $|1 - q^{\theta N - z + j}| \leq 1 + q^{N\Re\theta - \Re z + j}$ and $|1 - q^{-z+j}| \geq q^{-\Re z + j} - 1$. Along with condition~(\ref{conditiona}), we deduce that $|1 - q^{-z+j}| \geq |1 - q^{\theta N - z + j}|$, for all $0\leq j\leq M$. Since we can write
\begin{gather*}
\frac{\big(q^{\theta N - z}; q\big)_{\infty}}{(q^{- z}; q)_{\infty}} = \frac{\big(1 - q^{\theta N - z}\big)\cdots \big(1 - q^{\theta N - z + M}\big)}{(1 - q^{- z})\cdots \big(1 - q^{- z + M}\big)}\cdot\frac{\big(q^{\theta N - z + M + 1}; q\big)_{\infty}}{\big(q^{- z + M + 1}; q\big)_{\infty}},
\end{gather*}
it follows that
\begin{gather}\label{boundpochhammer}
\left| \frac{\big(q^{\theta N - z}; q\big)_{\infty}}{(q^{- z}; q)_{\infty}} \right| \leq \left| \frac{\big(q^{\theta N - z + M + 1}; q\big)_{\infty}}{\big(q^{- z + M + 1}; q\big)_{\infty}} \right|.
\end{gather}
Thus we only need to show that the right-hand side of~(\ref{boundpochhammer}) is bounded by a constant, for all $z\in\CC^+$ with $\Re z$ large enough.

Since $M = \lfloor \Re z - a \rfloor > \Re z - a -1$, we have $|q^{-z + M + 1}| = q^{M + 1 - \Re z} \leq q^{-a}$ and $|q^{\theta N - z + M + 1}| \leq q^{N\Re\theta - a} < q^{-a}$. Moreover, for $z\in\CC^+$ with $\Re z$ large enough, $|\Im q^{-z+M+1}| \geq q^{-a+1}|\sin(\ln{q}\cdot \Im z)| \myeq m_1(z)$ and since $|\Im z|$ is a constant between $0$ and $-\frac{\pi}{2\ln{q}}$ for $z\in\CC^+$ with $\Re z$ large enough, then $m_1(z) = m_1 > 0$ is a strictly positive constant independent of $z\in\CC^+$ as long as $\Re z$ is large enough. Since the function $(x; q)_{\infty}$ is continuous on $x\in\C$, we have $c_{\max} \myeq \sup\limits_{|x| \leq q^{-a}}{|(x; q)_{\infty}|} < \infty$, and $c_{\min} \myeq \inf_{|x|\leq q^{-a}, |\Im x| \geq m_1}{|(x; q)_{\infty}|} \in (0, \infty)$. Thus the right-hand side of~(\ref{boundpochhammer}) is upper bounded by the constant $c_{\max}/c_{\min} < \infty$.
\end{proof}

\textbf{Proof of Step 4.} We prove a stronger statement than step 4. Let $M > 0$ be \textit{any} positive number. We show that both sides of~(\ref{macdonaldthm2eqn}) are uniformly bounded on $\{\theta\in\C\colon \Re\theta \geq M\}$. Let us begin with the left-hand side of~(\ref{macdonaldthm2eqn}). Observe that $\theta$ appears in the left side only within the variable $t = q^{\theta}$ and $|t| = q^{\Re\theta} \leq q^M$. Name $\epsilon = q^M\in (0, 1)$; we have to prove that there exists a constant $C > 0$ such that
\begin{gather*}
\sup_{|t| \leq \epsilon}{\left| \frac{P_{\lambda}\big(xt^{N-1}, t^{N-2}, \dots, t, 1; q, t\big)}{P_{\lambda}\big(t^{N-1}, t^{N-2}, \dots, t, 1; q, t\big)} \right|} \leq C.
\end{gather*}
Thanks to the branching rule for Macdonald polynomials, Theorem~\ref{branchingmacdonald}, and the assumptions $|x|\leq 1$, $x \neq 0$, $\lambda\in\GTp_N$, we have
\begin{gather*}
\left| \frac{P_{\lambda}\big(xt^{N-1}, t^{N-2}, \dots, t, 1; q, t\big)}{P_{\lambda}\big(t^{N-1}, t^{N-2}, \dots, t, 1; q, t\big)} \right| \leq \sum_{\mu\prec\lambda}{|\psi_{\lambda/\mu}(q, t)| |t|^{(N-1)(|\lambda| - |\mu|)} \left| \frac{P_{\mu}\big(t^{N-2}, \dots, t, 1; q, t\big)}{P_{\lambda}\big(t^{N-1}, \dots, t, 1; q, t\big)} \right| }.
\end{gather*}
Given $\lambda\in\GTp_N$, there are f\/initely many $\mu\in\GTp_{N-1}$ with $\mu\prec\lambda$. Thus it suf\/f\/ices to prove that there exist constants $C_1, C_2 > 0$ such that
\begin{gather*}
|\psi_{\lambda/\mu}(q, t)| \leq C_1, \qquad \left| t^{(N-1)(|\lambda| - |\mu|)}\frac{P_{\mu}\big(t^{N-2}, \dots, t, 1; q, t\big)}{P_{\mu}\big(t^{N-1}, \dots, t, 1; q, t\big)} \right| \leq C_2,
\end{gather*}
where $C_1$, $C_2$ do not depend on $t$, though they may depend on $\mu$.

The branching coef\/f\/icient $|\psi_{\lambda/\mu}(q, t)|$, due to the expression in Theorem~\ref{branchingmacdonald}, is a f\/inite product of terms of the form $\frac{1 - q^at^b}{1 - q^ct^d}$, with $a, b, c, d\in\Z_{\geq 0}$ and $(c, d)\neq (0, 0)$. We have{\samepage
\begin{gather*}
\left| \frac{1 - q^at^b}{1 - q^ct^d} \right| \leq 2\big(1 - q^c\epsilon^d\big)^{-1}
\end{gather*}
for all $|t|\leq\epsilon$ and $2(1 - q^c\epsilon^d)^{-1}$ does not depend on $t$, so the boundedness of $|\psi_{\lambda/\mu}(q, t)|$ follows.}

Due to the evaluation identity for Macdonald polynomials, Theorem~\ref{evaluation}, we have
\begin{gather*}
t^{(N-1)(|\lambda| - |\mu|)}\frac{P_{\mu}\big(t^{N-2}, \dots, t, 1; q, t\big)}{P_{\lambda}\big(t^{N-1}, \dots, t, 1; q, t\big)} = t^{(N-1)(|\lambda| - |\mu|) + n(\mu) - n(\lambda)} \\
\qquad{} \times\prod_{1\leq i < j\leq N-1}{\frac{(q^{\mu_i - \mu_j}t^{j-i}; q)_{\infty}(t^{j-i+1}; q)_{\infty}}{(t^{j-i}; q)_{\infty}(q^{\mu_i - \mu_j}t^{j-i+1}; q)_{\infty}}}\prod_{1\leq i < j\leq N}{\frac{(t^{j-i}; q)_{\infty}(q^{\lambda_i - \lambda_j}t^{j-i+1}; q)_{\infty}}{(q^{\lambda_i - \lambda_j}t^{j-i}; q)_{\infty}(t^{j-i+1}; q)_{\infty}}},
\end{gather*}
where $n(\lambda) \myeq (N-1)\lambda_N + (N-2)\lambda_{N-1} + \dots + 2\lambda_3 + \lambda_2$ and similarly for $n(\mu)$.

The last two terms above are products of a f\/inite number of fractions $\frac{1 - q^at^b}{1 - q^ct^d}$, with $a, b, c, d\in\Z_{\geq 0}$, $(c, d)\neq (0, 0)$, and as we saw above it is implied that the absolute value of the last two terms above are upper bounded by a constant independent of~$t$ (as long as $|t| \leq \epsilon$). Thus our only goal is to show there is an upper bound for $t^{(N-1)(|\lambda| - |\mu|) + n(\mu) - n(\lambda)}$; this fact follows if the exponent is nonnegative. In fact, we have
\begin{gather*}
(N - 1)(|\lambda| - |\mu|) - (n(\lambda) - n(\mu)) \\
\qquad{} =(N - 1)(|\lambda| - |\mu|) - ((N - 1)\lambda_N + (N - 2)(\lambda_{N-1} - \mu_{N-1}) + \dots + (\lambda_2 - \mu_2)) \\
\qquad {} \geq (N-1)(|\lambda| - |\mu|) - (N-1)(\lambda_N + \lambda_{N-1} - \mu_{N-1} + \dots + \lambda_2 - \mu_2)\\
\qquad{} = (N - 1)(\lambda_1 - \mu_1)\geq 0.
\end{gather*}

Let us proceed to prove uniform boundedness of the right-hand side of~(\ref{macdonaldthm2eqn}) on $\{\theta\in\C\colon \Re\theta \geq M\}$. First of all, the triangular inequality gives
\begin{gather*}
\big|\big(xt^N; q\big)_{\infty}\big| \leq \prod_{i=0}^{\infty}{\big(1 + |x| q^{N\Re\theta + i}\big)} \leq \prod_{i=0}^{\infty}{\big(1 + |x| q^i\big)}, \\ |(xq; q)_{\infty}| = \prod_{i=1}^{\infty}{\big|\big(1 - xq^i\big)\big|} \geq \prod_{i=1}^{\infty}{\big(1 - |x| q^i\big)},
\end{gather*}
so the factor $\big(xt^N; q\big)_{\infty}/(xq; q)_{\infty}$ in~(\ref{macdonaldthm2eqn}) has an upper-bounded absolute value. We are left to deal with
\begin{gather}
\Gamma_q(\theta N)\int_{\CC^+}{x^z \prod_{i=1}^N{\frac{\Gamma_q(\lambda_i + \theta(N - i) - z)}{\Gamma_q(\lambda_i + \theta(N - i + 1) - z)}} {\rm d}z}\nonumber\\
 \qquad {} = \int_{\CC^+}{x^z \Gamma_q(\theta N)\prod_{i=1}^N{\frac{\Gamma_q(\lambda_i + \theta(N - i) - z)}{\Gamma_q(\lambda_i + \theta(N - i + 1) - z)}}{\rm d}z}\label{integrand1}
\end{gather}
and prove its absolute value is uniformly bounded on $\{\theta\in\C \colon \Re\theta \geq M\}$.

For any $M_2 > 0$, the contribution of the portion $\CC^+ \cap \{z\in\C \colon \Re z \leq M_2\}$ of the contour is bounded by a constant. In fact, $\Re\theta \geq M$ implies $|t| = |q^{\theta}| = q^{\Re\theta} \leq q^M$ and $\theta$ appears in the integrand of~(\ref{macdonaldthm2eqn}) only as part of the exponent of some~$q$, thus the integrand can be written as a function of $z$ and $t$ (with $q\in (0, 1)$ f\/ixed). Thus for $(z, t)$ in the compact subset $(\CC^+ \cap \{z\in\C\colon \Re z \leq M_2\})\times \big[0, q^M\big]$, the integrand in~(\ref{integrand1}) attains a maximum value $L < \infty$, and the contribution of the integral in the portion $\CC^+ \cap \{z\in\C\colon \Re z \leq M_2\}$ of the contour is upper bounded by $L$ times the length of that f\/inite portion.

Since $|x| < 1$, the term $x^z$ decreases exponentially as $|z|\rightarrow\infty$, $z\in\CC^+$. Thus to deal with the inf\/inite portion of the integral $\CC^+\cap\{z\in\C\colon \Re z \geq M_2\}$, it is enough to show that
\begin{gather*}
\Gamma_q(\theta N)\prod_{i=1}^N{\frac{\Gamma_q(\lambda_i + \theta(N - i) - z)}{\Gamma_q(\lambda_i + \theta(N - i + 1) - z)}} = (1 - q)\frac{(q; q)_{\infty}}{(q^{\theta N}; q)_{\infty}}\prod_{i=1}^N{\frac{\big(q^{\lambda_i + \theta(N - i + 1) - z}; q\big)_{\infty}}{\big(q^{\lambda_i + \theta (N - i) - z}; q\big)_{\infty}}}
\end{gather*}
has bounded absolute value for all $z\in\CC^+$, $\Re z > M_2$ and $\theta\in\C$, $\Re\theta \geq M$, for a suitable constant $M_2 > 0$. Clearly $\big|\big(q^{\theta N}; q\big)\big|^{-1} \leq \big(q^{N\Re\theta}; q\big)^{-1} \leq \big(q^{NM}; q\big)^{-1}$, thus we only need a bound on the absolute value of
\begin{gather}\label{toproveboundedness}
\prod_{i=1}^N{\frac{\big(q^{\lambda_i + \theta(N - i + 1) - z}; q\big)_{\infty}}{\big(q^{\lambda_i + \theta (N - i) - z}; q\big)_{\infty}}},
\end{gather}
uniformly over all $z\in\CC^+$, $\Re z > M_2$, and $\theta\in\C$, $\Re\theta \geq M$. We can bound the absolute value of~(\ref{toproveboundedness}) by
\begin{gather*}
\prod_{i=1}^{N-1}{\left|\frac{\big(q^{\lambda_{i+1} + \theta(N - i) - z}; q\big)_{\infty}}{\big(q^{\lambda_i + \theta (N - i) - z}; q\big)_{\infty}}\right|}\times\left|\frac{\big(q^{\lambda_1 + \theta N - z}; q\big)_{\infty}}{\big(q^{\lambda_N - z}; q\big)_{\infty}}\right|\\
\qquad{} = \prod_{ \substack{1\leq i\leq N - 1 \\ \lambda_{i+1} \leq j < \lambda_i} }{\big|1 - q^{\theta(N - i) + j - z}\big|}\times\frac{1}{\prod\limits_{\lambda_N \leq j < \lambda_1}{|1 - q^{j - z}|}}
\times\left|\frac{\big(q^{\lambda_1 + \theta N - z}; q\big)_{\infty}}{\big(q^{\lambda_1 - z}; q\big)_{\infty}}\right|\\
\qquad{} \leq \prod_{ \substack{1\leq i\leq N - 1 \\ \lambda_{i+1} \leq j < \lambda_i} }{\big(1 + q^{(N - i)\Re\theta + j - \Re z}\big)}\times\frac{1}{\prod\limits_{\lambda_N \leq j < \lambda_1}{\big(q^{j-\Re z} - 1\big)}}\times\left|\frac{\big(q^{\lambda_1 + \theta N - z}; q\big)_{\infty}}{\big(q^{\lambda_1 - z}; q\big)_{\infty}}\right|\\
\qquad{} \leq \prod_{\lambda_N \leq j < \lambda_1}{\frac{1 + q^{j - \Re z}}{q^{j - \Re z} - 1}}\times\left|\frac{\big(q^{\lambda_1 + \theta N - z}; q\big)_{\infty}}{\big(q^{\lambda_1 - z}; q\big)_{\infty}}\right|.
\end{gather*}
If $\Re z$ is large enough, the product $\prod\limits_{\lambda_N \leq j < \lambda_1}{\frac{1 + q^{j - \Re z}}{q^{j - \Re z} - 1}}$ is clearly upper bounded by a constant. We still have to bound $\left|\frac{(q^{\lambda_1 + \theta N - z}; q)_{\infty}}{(q^{\lambda_1 - z}; q)_{\infty}}\right|$. Since $\lambda_1$ is real and $\Im z$ is constant for $z\in\CC^+$, $\Re z$ large enough, it suf\/f\/ices to prove the following statement: there exist constants $M_1, M_2 > 0$ such that $z\in\CC^+$, $\Re z > M_2$, and $\Re\theta \geq M$ imply
\begin{gather*}
\left| \frac{\big(q^{\theta N - z}; q\big)_{\infty}}{(q^{-z}; q)_{\infty}} \right| < M_1.
\end{gather*}
This statement was proved above in step 3, see~(\ref{boundthetaN}). In that case, $\theta$ varied over a compact subset $K\subset\{\theta\in\C\colon \Re\theta > 0\}$, but in this case $\theta$ varies over a closed inf\/inite domain of the form $\{\theta\in\C\colon \Re\theta \geq M\}$. However, the expression $\frac{(q^{\theta N - z}; q)_{\infty}}{(q^{-z}; q)_{\infty}}$ depends on $\theta$ only by means of $q^{\theta}$, so the proof of~(\ref{boundthetaN}) above can be repeated word-by-word, since we only used $|t|\leq q^{\inf_{\theta\in K}{\Re\theta}}< 1$ in that proof.

\section{Multiplicative formulas for Macdonald characters}\label{macdonaldmultiplicativesec}

We now come to the multiplicative formulas. All of our results require parameter~$\theta$ to be a~positive integer. In this section, $q$ is typically a~variable (but of course, we can specialize~$q$ to a~complex number later).

\subsection{Statement of the multiplicative theorem and some consequences}\label{macdonaldmultiplicativesubsec1}

We need some non-standard terminology on $q$-dif\/ference operators. The \textit{$q$-shift operators $\{T_{q, x_i}\colon i = 1, \dots, m\}$} are linear operators on $\C(q)[x_1, \dots, x_m]$ that act as
\begin{gather*}
(T_{q, x_i}f)(x_1, \dots, x_m) \myeq f(x_1, \dots, x_{i-1}, qx_i, x_{i+1}, \dots, x_m).
\end{gather*}
The \textit{$q$-degree operators $\{D_{q, x_i}\}_{i = 1, \dots, m}$} are linear operators on $\C(q)[x_1, \dots, x_m]$ def\/ined by
\begin{gather*}
D_{q, x_i} \myeq \frac{T_{q, x_i} - 1}{q - 1} \qquad \text{or}\\
(D_{q, x_i}f)(x_1, \dots, x_m) = \frac{f(x_1, \dots, x_{i-1}, qx_i, x_{i+1}, \dots, x_m) - f(x_1, \dots, x_m)}{q - 1}.
\end{gather*}
The $q$-dif\/ference operators that appear in the multiplicative formulas for Macdonald polynomials are f\/inite sums of terms
\begin{gather*}
c_{i_1, \dots, i_m}(x_1, \dots, x_m; q)T_{q, x_1}^{i_1}\cdots T_{q, x_m}^{i_m} \qquad \text{or} \qquad c_{i_1, \dots, i_m}(x_1, \dots, x_m; q)D_{q, x_1}^{i_1}\cdots D_{q, x_m}^{i_m},
 \end{gather*} where $(i_1, \dots, i_m)$ vary over a f\/inite subset of $\Z_{\geq 0}^m$, and $c_{i_1, \dots, i_m}(x_1, \dots, x_m; q)$ are rational functions in the variables $q, x_1, \dots, x_m$. Thus the operators that we consider are linear operators $\C(q)[x_1, \dots, x_m] \longrightarrow \C(q, x_1, \dots, x_m)$ that act on polynomials and yield rational functions.

Recall the setting of Jacobi Trudi's formula for Macdonald polynomials, Theorem~\ref{jacobitrudi}. The expressions $C_{\tau_1, \dots, \tau_n}^{(q, q^{\theta})}(u_1, \dots, u_n)$ were def\/ined in~(\ref{Ccoeff}); they are rational functions in $q, u_1, \dots, u_n$ whose denominators are products of linear factors. We def\/ine $M_{\theta}^{(m)}$ as the set of strictly upper-triangular $m\times m$ matrices whose entries belong to $\{0, 1, \dots, \theta\}$ (the cardinality of $M_{\theta}^{(m)}$ is $(\theta + 1){m \choose 2}$). For any strictly upper-triangular $m\times m$ matrix $\tau$ and $1\leq i\leq m$, let $\tau_i^+ \myeq \sum\limits_{j=i+1}^m{\tau_{i, j}}$ (resp.\ $\tau_i^- \myeq \sum\limits_{j=1}^{i-1}{\tau_{j, i}}$) be the sum of the entries of~$\tau$ to the right of $(i, i)$ (resp.\ sum of entries of $\tau$ above~$(i, i)$), cf.~(\ref{tauplusminus}). The main theorem of this section is the following.

\begin{thm}\label{macdonaldthm3}
Let $\theta\in\N$, $t = q^{\theta}$, $N\in\N$, $\lambda\in\GT_N$. Then
\begin{gather}
P_{\lambda}\big(x_1, \dots, x_m; N, q, q^{\theta}\big) = \frac{q^{\theta^2 {m+ 1 \choose 3} - \left\{ N\theta^2 - {\theta + 1 \choose 2} \right\}{m \choose 2}}\prod\limits_{i=1}^m{[\theta(N - i +1) - 1]_q!}}{\prod\limits_{i=1}^m{\prod\limits_{j=1}^{\theta(N - m + 1) - 1}{(x_i - q^{j-\theta})}}}\nonumber\\
\qquad{}
\times\frac{1}{\prod_{\substack{1\leq i<j\leq m \\ 0\leq k < \theta}}{(x_i - q^k x_j)}}\times\DD^{(m)}_{q, \theta} \left\{ \prod_{i=1}^m{\frac{P_{\lambda}\big(x_i; N, q, q^{\theta}\big)\prod\limits_{j=1}^{\theta N-1}{\big(x_i - q^{j-\theta}\big)}}{[\theta N-1]_q!}}\right\},\label{macdonaldthm3eqn}
\end{gather}
where $\DD^{(m)}_{q, \theta}$ is the $q$-difference operator $\C(q)[x_1, \dots, x_m] \longrightarrow \C(q, x_1, \dots, x_m)$ given by
\begin{gather}\label{macdonaldthm3Dq}
\DD^{(m)}_{q, \theta} \myeq \frac{1}{(q - 1)^{\theta {m \choose 2}}}\sum_{\tau\in M^{(m)}_{\theta}}\left\{ C_{\tau}^{(q, q^{\theta})}(x_1, \dots, x_m)\prod_{i = 1}^m{T_{q, x_i}^{(i-1)\theta + \tau_i^+ - \tau_i^-}}\right\},
\end{gather}
where for any $\tau\in M_{\theta}^{(m)}$, we denoted
\begin{gather}
C_{\tau}^{(q, q^{\theta})}(x_1, \dots, x_m) \nonumber\\
\qquad{} \myeq \prod_{s = 1}^{m-1}{C_{\tau_{1, s+1}, \dots, \tau_{s, s+1}}^{(q, q^{\theta})}\left(\{u_i = x_{s+1}^{-1}x_iq^{-\theta + \sum\limits_{j=s+2}^m{(\tau_{i, j} - \tau_{s+1, j})}}\colon 1\leq i \leq s\} \right)}.\label{eq:Ccoeff}
\end{gather}
\end{thm}

The proof of Theorem~\ref{macdonaldthm3} is given in the next subsection. We derive here some conclusions, namely two special cases when the operator $\DD_{q, \theta}^{(m)}$ has a simple form. The f\/irst simple case is $m = 2$.

\begin{cor}\label{macdonaldcor3}
In the same setting as Theorem~{\rm \ref{macdonaldthm3}} $($for $m = 2)$, we have
\begin{gather*}
P_{\lambda}\big(x_1, x_2; N, q, q^{\theta}\big) = \frac{q^{-(N-1)\theta^2 + {\theta + 1 \choose 2} } \cdot [\theta(N - 1) - 1]_q!}{\prod\limits_{j=1}^{\theta(N - 1) - 1}{\big(x_1 - q^{j-\theta}\big)\big(x_2 - q^{j-\theta}\big)}}\\
\hphantom{P_{\lambda}\big(x_1, x_2; N, q, q^{\theta}\big) = }{} \times\widetilde{\DD}_{q, \theta}^{(2)} \left\{ \frac{\prod\limits_{i=1}^2{ P_{\lambda}\big(x_i; N, q, q^{\theta}\big)\prod\limits_{j=1}^{\theta N-1}{\big(x_i - q^{j-\theta}\big)} } }{[\theta N - 1]_q!} \right\}
\end{gather*}
where
\begin{gather*}
\widetilde{\DD}_{q, \theta}^{(2)} \myeq \left( \frac{1}{x_1 - x_2} \circ (D_{q, x_2} - D_{q, x_1}) \right)^{\theta}\\
\hphantom{\widetilde{\DD}_{q, \theta}^{(2)} }{} = \underbrace{\left(\frac{1}{x_1 - x_2} \circ (D_{q, x_2} - D_{q, x_1})\right) \circ\cdots \circ \left(\frac{1}{x_1 - x_2} \circ (D_{q, x_2} - D_{q, x_1})\right)}_{\textrm{composition of $\theta$ operators}}.
\end{gather*}
\end{cor}
\begin{proof}
For $m = 2$, we have $M_{\theta}^{(2)} =\left\{ \left[
 \begin{smallmatrix}
 0 & n \\
 0 & 0
 \end{smallmatrix} \right], \ n\in\{0, 1, \dots, \theta\} \right\}$,
and thus the operator $\DD_{q, \theta}^{(2)}$ becomes
\begin{gather*}
\DD^{(2)}_{q, \theta} = \frac{1}{(q - 1)^{\theta}}\sum_{n = 0}^{\theta}\big\{{C_n^{(q, q^{\theta})}\big(x_{2}^{-1}x_1q^{-\theta} \big)} T_{q, x_1}^nT_{q, x_2}^{\theta - n}\big\}.
\end{gather*}
We let
\begin{gather*}
a_n^{(\theta)} = \frac{C_n^{(q, q^{\theta})}\big(x_2^{-1}x_1q^{-\theta}\big)}{\prod\limits_{i=0}^{\theta - 1}{(x_1 - q^ix_2)}}, \qquad 0\leq n\leq \theta.
\end{gather*}
The statement of the theorem can be easily reduced to prove that the coef\/f\/icient of $T_{q, x_1}^nT_{q, x_2}^{\theta - n}$ in the product
\begin{gather*}
D_{q, \theta} = \left(\frac{1}{x_1 - x_2} \circ (T_{q, x_2} - T_{q, x_1})\right)^{\theta} = \sum_{n=0}^{\theta}{b_n^{(\theta)}T_{q, x_1}^nT_{q, x_2}^{\theta-n}}
\end{gather*}
is equal to $a_n^{(\theta)}$, i.e., we prove $a_n^{(\theta)} = b_n^{(\theta)}$ for all $\theta, n\in\N$, $0 \leq n \leq \theta$, and we do it by induction on $\theta$. The case $\theta = 1$ can be easily dealt with, using Lemma~\ref{Cprop2}. Now assume $a_n^{(\theta - 1)} = b_n^{(\theta - 1)}$ for all $0\leq n \leq \theta - 1$, and some $\theta \geq 2$. We prove $a_n^{(\theta)} = b_n^{(\theta)}$ for all $0\leq n\leq \theta$. Evidently, $D_{q, \theta} = \big(\frac{1}{x_1 - x_2}\circ (T_{q, x_2} - T_{q, x_1})\big)\circ D_{q, \theta - 1}$ implies
\begin{gather*}
b_n^{(\theta)} = \frac{1}{x_1 - x_2}\big( T_{q, x_2} b_n^{(\theta - 1)} - T_{q, x_1} b_{n-1}^{(\theta - 1)} \big), \qquad n = 1, 2, \dots, \theta-1,\\
b_0^{(\theta)} = \frac{1}{x_1 - x_2} T_{q, x_2}b_0^{(\theta-1)}, \qquad b_{\theta}^{(\theta)} = \frac{1}{x_2 - x_1} T_{q, x_1}b_{\theta-1}^{(\theta-1)}.
\end{gather*}
From Lemma~\ref{Cprop3} in Appendix~\ref{app:cfunctions}, the terms $a_n^{(\theta)}$ satisfy
\begin{gather*}
a_n^{(\theta)} = \frac{1}{x_1 - x_2} \big( T_{q, x_2} a_n^{(\theta - 1)} - T_{q, x_1} a_{n-1}^{(\theta - 1)} \big), \qquad n = 1, 2, \dots, \theta-1,\\
a_0^{(\theta)} = \frac{1}{\prod\limits_{i=0}^{\theta - 1}{(x_1 - q^i x_2)}},\qquad a_{\theta}^{(\theta)} = \frac{1}{\prod\limits_{i=0}^{\theta - 1}{(x_2 - q^i x_1)}}.
\end{gather*}
It is not dif\/f\/icult to conclude from these relations, and the inductive hypothesis $a_m^{(\theta - 1)} = b_m^{(\theta - 1)}$ $\forall \, m = 0, 1, \dots, \theta-1$, that $a_n^{(\theta)} = b_n^{(\theta)}$ for all $0\leq n\leq\theta$, as desired.
\end{proof}

When $\theta = 1$ (equivalently $t = q$), the result has a compact form as well. Let us recall that when $t = q$, the Macdonald polynomials become the well known \textit{Schur polynomials} $s_{\lambda}(x_1, \dots, x_N) = P_{\lambda}(x_1, \dots, x_N; q, q)$. The Schur (Laurent) polynomials $s_{\lambda}(x_1, \dots, x_N)$, $\lambda\in\GT_N$, can also be def\/ined by the simple determinantal formula
\begin{gather*}
s_{\lambda}(x_1, \dots, x_N) = \frac{\det\big[x_i^{N+1-j}\big]_{i, j = 1}^N}{\prod\limits_{1\leq i < j \leq N}{(x_i - x_j)}}.
\end{gather*}
For any $m\in\N$ with $1\leq m\leq N$, we consider
\begin{gather*}
s_{\lambda}(x_1, \dots, x_m; N, q) \myeq \frac{s_{\lambda}\big(x_1, \dots, x_m, 1, q, \dots, q^{N-1-m}\big)}{s_{\lambda}(1, q, \dots, q^{N-1})}
\end{gather*}
and call it a \textit{$q$-Schur character} of rank~$N$, number of variables $m$ and parametrized by~$\lambda$; it was def\/ined before in~\cite{G}. We recover the following theorem.

\begin{cor}[{\cite[Theorem~3.5]{GP}}]\label{macdonaldcor4}
Let $N\in\N$, $\lambda\in\GT_N$. Then
\begin{gather*}
s_{\lambda}(x_1, \dots, x_m; N, q) = \frac{q^{{m+1 \choose 3} - (N-1){m \choose 2} }\prod\limits_{i=1}^m{[N - i]_q!}}{\prod\limits_{i=1}^m{\prod\limits_{j=1}^{N - m}{\big(x_i - q^{j-1}\big)}}}\frac{1}{\Delta(x_1, \dots, x_m)}\\
\hphantom{s_{\lambda}(x_1, \dots, x_m; N, q) =}{} \times \DD_q^{(m)}\left\{ \prod\limits_{i=1}^m{\frac{s_{\lambda}(x_i; N, q)\prod\limits_{j=1}^{N-1}{\big(x_i - q^{j-1}\big)}}{[N-1]_q!}}\right\},
\end{gather*}
where
\begin{gather*}
\DD_{q}^{(m)} \myeq \det \big[ D_{q, x_i}^{j-1} \big]_{i, j = 1}^m = \prod_{1\leq i < j\leq m}{(D_{q, x_j} - D_{q, x_i})}.
\end{gather*}
\end{cor}
\begin{proof}
Letting $\theta = 1$ in Theorem~\ref{macdonaldthm3}, we see that the equation above holds if $\DD_{q, 1}^{(m)} = \DD_q^{(m)}$. In fact, thanks to Lemma~\ref{Cprop2}, we have
\begin{gather*}
\DD^{(m)}_{q, 1} = (q - 1)^{- {m \choose 2}}\sum_{\tau\in M^{(m)}_1}
\prod_{s = 1}^{m-1}{(-1)^{\tau_{1, s+1} + \dots + \tau_{s, s+1}}}\prod_{i = 1}^m{T_{q, x_i}^{i-1 + \tau_i^+ - \tau_i^-}}.
\end{gather*}
Evidently $\DD_q^{(m)} = \prod\limits_{i<j}{(D_{q, x_j} - D_{q, x_i})} = (q - 1)^{- {m \choose 2}}\prod\limits_{i < j}{(T_{q, x_j} - T_{q, x_i})}$, thus we need to show
\begin{gather}\label{eqn:vander}
\prod_{1\leq i < j\leq m}{(T_{q, x_j} - T_{q, x_i})} = \sum_{\tau\in M^{(m)}_1}{(-1)^{|\tau|}\prod_{k = 1}^m{T_{q, x_k}^{k - 1 + \tau_k^+ - \tau_k^- }}},
\end{gather}
where we denoted by $|\tau|$ to the sum of all entries of $\tau\in M_1^{(m)}$. The operators $T_{q, x_1}, \dots, T_{q, x_m}$ pairwise commute. When expanding the left-hand side of~(\ref{eqn:vander}), it is clear that the resulting terms can be parametrized by matrices in $M_1^{(m)}$: the term corresponding to $\tau\in M_1^{(m)}$ is the product of $(-1)^k T_{q, x_i}^k T_{q, x_j}^{1 - k}$, where $k \in \{0, 1\}$ ranges over the elements of $\tau$ strictly above the main diagonal. Thus the term corresponding to $\tau$ is
\begin{gather*}
\prod_{1\leq i < j\leq m}{(-1)^{\tau_{i, j}} T_{q, x_i}^{\tau_{i, j}} T_{q, x_j}^{1 - \tau_{i, j}}} = (-1)^{\sum\limits_{1\leq i < j\leq m}{\tau_{i, j}}}\prod_{1\leq i < j\leq m}{T_{q, x_i}^{\tau_{i, j}}T_{q, x_j}^{1 - \tau_{i, j}}}.
\end{gather*}
We are left to show
\begin{gather}\label{eqn:product}
\prod_{1\leq i < j\leq m}{T_{q, x_i}^{\tau_{i, j}}T_{q, x_j}^{1 - \tau_{i, j}}} = \prod_{k=1}^m{T_{q, x_k}^{k -1 + \tau_k^+ - \tau_k^-}}.
\end{gather}
Both sides of~(\ref{eqn:product}) are of the form $T_{q, x_1}^{p_1}\cdots T_{q, x_m}^{p_m}$, for some $p_1, \dots, p_m\in\Z_{\geq 0}$, so we simply need to check the equality between exponents $p_k$ of $T_{q, x_k}$ in both sides, for an arbitrary $1\leq k\leq m$. In the left side, there are $k - 1$ factors of the form $T_{q, x_i}^{\tau_{i, k}}T_{q, x_k}^{1 - \tau_{i, k}}$, $i\leq k - 1$, which overall contribute $k - 1 - \sum\limits_{i=1}^{k-1}{\tau_{i, k}} = k - 1 - \tau_k^-$ to the exponent of $T_{q, x_k}$. Moreover there are $m - k$ factors of the form $T_{q, x_k}^{\tau_{k, j}}T_{q, x_j}^{1 - \tau_{k, j}}$, $k+1\leq i$, which contribute $\sum\limits_{i=k+1}^m{\tau_{k, i}} = \tau_k^+$ to the exponent of~$T_{q, x_k}$. Therefore the power of $T_{q, x_k}$ in the left-hand side of~(\ref{eqn:product}) is $T_{q, x_k}^{k-1+\tau_k^+ - \tau_k^-}$. Evidently, the power of~$T_{q, x_k}$ in the right-hand side of~(\ref{eqn:product}) if also $T_{q, x_k}^{k-1+\tau_k^+ - \tau_k^-}$, which f\/inishes the proof.
\end{proof}

\begin{exam}\label{macdonaldmultiplicative}
We discuss the f\/irst nontrivial example of the multiplicative formula for Macdonald polynomials (an example that is not dealt with in the Corollaries above): $\theta = 2$, $m = 3$. The formula in this case is
\begin{gather*}
P_{\lambda}\big(x_1, x_2, x_3; N, q, q^2\big) = \frac{q^{25 - 12N}}{([2N - 1]_q)^2([2N - 2]_q)^2[2N - 3]_q[2N - 4]_q}\\
\hphantom{P_{\lambda}\big(x_1, x_2, x_3; N, q, q^2\big) =}{}
\times\frac{1}{\prod\limits_{j=-1}^{2N - 7}{(x_1 - q^j)(x_2 - q^j)(x_3 - q^j)}}
\frac{1}{\prod\limits_{1\leq i<j\leq 3}{(x_i - x_j)(x_i - qx_j)}}\\
\hphantom{P_{\lambda}\big(x_1, x_2, x_3; N, q, q^2\big) =}{}
\times\widehat{\DD}_{q, 2}^{(3)}\left\{ \prod_{i=1}^3{P_{\lambda}\big(x_i; N, q, q^2\big)\prod_{j=-1}^{2N - 3}{(x_i - q^j)}} \right\},
\end{gather*}
the $q$-dif\/ference operator is
\begin{gather*}
\widehat{\DD}_{q, 2}^{(3)} = \frac{1}{(q - 1)^{6}}\sum_{\substack{a, b, c\in\Z \\ 0\leq a, b, c\leq 2}}\big\{ C^{(q, q^2)}_a\big(x_2^{-1}x_1q^{-2 + b - c}\big)\\
\hphantom{\widehat{\DD}_{q, 2}^{(3)} = \frac{1}{(q - 1)^{6}}\sum_{\substack{a, b, c\in\Z \\ 0\leq a, b, c\leq 2}}}{} \times C^{(q, q^2)}_{b, c}\big(x_3^{-1}x_1q^{-2}, x_3^{-1}x_2q^{-2}\big) T_{q, x_1}^{a + b}T_{q, x_2}^{2 + c - a}T_{q, x_3}^{4 - b - c}\big\}.
\end{gather*}
We can also write the $q$-dif\/ference operator in terms of the $q$-degree operators $\{D_{q, x_i} \colon i = 1, 2, 3\}$ by using $\frac{1}{q - 1}T_{q, x_i} = \frac{1}{q - 1} + D_{q, x_i}$. Then we can replace the operator $\widehat{\DD}_{q, 2}^{(3)}$ with the sum of operators $\DD_{q, 2}^{(3, {\rm top})} + \AAA_{q, 2}^{(3)}$, where
\begin{gather*}
\DD_{q, 2}^{(3, {\rm top})} = \sum_{\substack{a, b, c\in\Z \\ 0\leq a, b, c\leq 2}}\big\{ C^{(q, q^2)}_a\big(x_2^{-1}x_1q^{-2 + b - c}\big)\\
\hphantom{\DD_{q, 2}^{(3, {\rm top})} = \sum_{\substack{a, b, c\in\Z \\ 0\leq a, b, c\leq 2}}}{}
 \times C^{(q, q^2)}_{b, c}\big(x_3^{-1}x_1q^{-2}, x_3^{-1}x_2q^{-2}\big) D_{q, x_1}^{a + b} D_{q, x_2}^{2 + c - a} D_{q, x_3}^{4 - b - c}\big\},\\
\AAA_{q, 2}^{(3)} = \sum_{\substack{ i_1, i_2, i_3\geq 0 \\ i_1 + i_2 + i_3 \leq 5}}{\frac{f_{i_1, i_2, i_3}(x_1, x_2, x_3; q, 2)}{(q - 1)^{6 - i_1 - i_2 - i_3}} D_{q, x_1}^{i_1}D_{q, x_2}^{i_2}D_{q, x_3}^{i_3}},
\end{gather*}
and $f_{i_1, i_2, i_3}(x_1, x_2, x_3; q, 2)$ are certain rational functions. With the help of Sage, we found
\begin{gather*}
f_{i_1, i_2, i_3}(x_1, x_2, x_3; q, 2) = 0, \qquad \textrm{for all} \quad i_1 + i_2 + i_3 \leq 2.
\end{gather*}
There are nontrivial rational functions $f_{i_1, i_2, i_3}(x_1, x_2, x_3; q, 2)$ as well, for some $i_1 + i_2 + i_3 \geq 3$, e.g.,
\begin{gather}
f_{4, 1, 0}(x_1, x_2, x_3; q, 2) = -(q - 1)\frac{(x_2 + x_3)(x_1 - qx_3)(x_1 - qx_2)}{(qx_2 - x_3)(qx_1 - x_3)(qx_1 - x_2)},\label{bmacdonald1}\\
f_{2, 1, 1}(x_1, x_2, x_3; q, 2) = -(q-1)^2\frac{(q+1)(x_2 - x_3)\big(x_1^2 + x_2x_3 + 2x_1x_2 + 2x_1x_3\big)}{(qx_1 - x_2)(qx_1 - x_3)(qx_2 - x_3)}.\label{bmacdonald2}
\end{gather}
Two important observations are in order. First, from Corollaries~\ref{macdonaldcor3} and~\ref{macdonaldcor4}, one might believe that $\DD_{q, \theta}^{(m)}$ is, in general, homogeneous of degree $\theta {m \choose 2}$ as a functions of the operators $\{D_{q, x_i} \colon i = 1, 2, 3\}$. However, this example disproves it. Second, the terms $f_{i_1, i_2, i_3}(x_1, x_2, x_3; q, 2)$ above make us suspect that $f_{i_1, i_2, i_3}(x_1, x_2, x_3; q, 2)$ is divisible by $(q- 1)^{6 - i_1 - i_2 - i_3}$ $\forall \, 0\leq i_1 + i_2 + i_3\leq 5$. We have checked this fact in the computer. In fact, we believe that the analogous statement for general $m$, $\theta \in \N$ holds true, but the author could not prove it.
\end{exam}

\subsection{Proof of Theorem~\ref{macdonaldthm3}}

Fix a positive signature $\lambda\in\GTp_N$ and let us prove equation~(\ref{macdonaldthm3eqn}); we extend the result for all signatures $\lambda\in\GT_N$ at the end.

Let us consider $m$ positive integers $n_1 > n_2 > \cdots > n_m > \theta (N + m)$. By the index-argument symmetry, Theorem~\ref{symmetry}, applied to $\lambda\in\GTp_N$ and $\mu = (n_1 \geq \cdots \geq n_m \geq 0 \geq \cdots \geq 0)\in\GTp_N$, as well as the def\/inition of the dual Macdonald polynomials $Q_{\mu}(\cdot; q, t)$, we obtain
\begin{gather}
\frac{P_{\lambda}\big(q^{n_1}t^{N-1}, \dots, q^{n_m}t^{N-m}, t^{N-m-1}, \dots, t, 1; q, t\big)}{P_{\lambda}\big(t^{N-1}, t^{N-2}, \dots, 1; q, t\big)} \nonumber\\
\qquad {}= \frac{Q_{(n_1, \dots, n_m)}\big(q^{\lambda_1}t^{N-1}, q^{\lambda_2}t^{N-2}, \dots, q^{\lambda_N}; q, t\big)}{Q_{(n_1, \dots, n_m)}\big(t^{N-1}, t^{N-2}, \dots, 1; q, t\big)}.\label{eq1}
\end{gather}
Apply the Jacobi--Trudi formula for Macdonald polynomials, Theorem~\ref{jacobitrudi}, to the numerator of~\eqref{eq1}, then multiply and divide the term parametrized by $\tau$ by the product
\begin{gather*}
\prod_{s=1}^m{{g_{n_s+\tau_s^+ - \tau_s^-}\big(t^{N-1}, \dots, t, 1; q, t\big)}},
\end{gather*} so (\ref{eq1}) equals
\begin{gather}
\sum_{\tau\in M^{(m)}}\left\{\prod_{s=1}^{m-1}{C^{(q, t)}_{\tau_{1, s+1}, \dots, \tau_{s, s+1}}\Big(\Big\{ u_i = q^{n_i - n_{s+1} + \sum\limits_{j=s+2}^m{(\tau_{i, j} - \tau_{s+1, j})}}t^{s-i}\colon 1\leq i\leq s\Big\}\Big)} \right.\nonumber\\
\left.{} \qquad{} \times\prod_{s=1}^m{\frac{g_{n_s + \tau_s^+ - \tau_s^-}\big(q^{\lambda_1} t^{N-1}, \dots, q^{\lambda_N}; q, t\big)}{g_{n_s + \tau_s^+ - \tau_s^-}\big(t^{N-1}, \dots, t, 1; q, t\big)}} \times \frac{\prod\limits_{s=1}^m{g_{n_s + \tau_s^+ - \tau_s^-}\big(t^{N-1}, \dots, t, 1; q, t\big)}}{Q_{(n_1, \dots, n_m)}\big(t^{N-1}, \dots, t, 1; q, t\big)} \right\}.\!\!\!\!\label{eq2}
\end{gather}

(In equation (\ref{eq2}), we are setting $g_n(q, t) = 0$ if $n$ is a nonpositive integer.) Recall that $t = q^{\theta}$, $\theta\in\N$. In view of Lemma~\ref{Cprop1}, the only terms in the sum~(\ref{eq2}) with nonzero contributions are those parametrized by $m\times m$ matrices whose entries belong to the set $\{0, 1, \dots, \theta\}$; def\/ine $M^{(m)}_{\theta}$ to be the set of such matrices. Notice that $n_s + \tau_s^+ - \tau_s^- > 0$ for all $\tau\in M_{\theta}^{(m)}$, because of our initial assumption on the values of $n_1, \dots, n_m$. By another application of the index-argument symmetry (and of the identity~(\ref{gtoQ}) above), we have
\begin{gather}\label{eq33}
\frac{g_{n_s + \tau_s^+ - \tau_s^+}\big(q^{\lambda_1}t^{N-1}, \dots, q^{\lambda_N}; q, t\big)}{g_{n_s + \tau_s^+ - \tau_s^-}(t^{N-1}, \dots, t, 1; q, t)}\nonumber\\
\qquad {} = \frac{P_{\lambda}\big(q^{n_s + \tau_s^+ - \tau_s^-} t^{N-1}, t^{N-2}, \dots, t, 1; q, t\big)}{P_{\lambda}(t^{N-1}, \dots, t, 1; q, t)} \qquad \forall\, 1\leq s\leq m.
\end{gather}

Plugging (\ref{eq33}) into (\ref{eq2}), we obtain
\begin{gather}
\frac{P_{\lambda}\big(q^{n_1 + \theta(N-1)}, \dots, q^{n_m + \theta(N-m)}, t^{N-m-1}, \dots, t, 1; q, t\big)}{P_{\lambda}\big(t^{N-1}, t^{N-2}, \dots, 1; q, t\big)}\label{eq4}\\
\qquad {} =
\sum_{\tau\in M^{(m)}_{\theta}}\left\{\prod_{s=1}^{m-1}{C^{(q, q^{\theta})}_{\tau_{1, s+1}, \dots, \tau_{s, s+1}}\Big(\Big\{ u_i = q^{n_i - n_{s+1} + \sum\limits_{j=s+2}^m{(\tau_{i, j} - \tau_{s+1, j})} + \theta(s-i)}\colon 1\leq i\leq s\Big\}\Big) } \right.\nonumber\\
\left. \qquad{} \times\!\frac{\prod\limits_{s=1}^m{g_{n_s + \tau_s^+ - \tau_s^-}\big(t^{N-1}, \dots, t, 1; q, t\big)}}{Q_{(n_1, \dots, n_m)}\big(t^{N-1}, \dots, t, 1; q, t\big)}\!\times \!\prod_{s=1}^m\! { \frac{P_{\lambda}\big(q^{n_s + \tau_s^+ - \tau_s^- + \theta(N-1)}, t^{N-2}, \dots, t, 1; q, t\big)}{P_{\lambda}\big(t^{N-1}, \dots, t, 1; q, t\big)} } \right\}\!.\!\nonumber
\end{gather}
Let us make the change of variables
\begin{gather}\label{changevars}
z_s \myeq q^{n_s + \theta(N - s)}, \qquad 1\leq s\leq m,
\end{gather}
and rewrite some terms from~(\ref{eq4}) in these new variables. Clearly the left-hand side of equality~(\ref{eq4}) is the Macdonald character $P_{\lambda}(z_1, \dots, z_m; N, q, q^{\theta})$. It is also evident that the variable $u_i$ in the term $C_{\tau_{1, s+1}, \dots, \tau_{s, s+1}}^{(q, q^{\theta})}$ can be rewritten as $u_i = z_{s+1}^{-1}z_i q^{-\theta + \sum\limits_{j=s+2}^m{(\tau_{i, j} - \tau_{s+1, j})}}$, for $1\leq i\leq s$. Additionally we have $q^{n_s + \tau_s^+ - \tau_s^- + \theta(N-1)} = T_{q, z_s}^{\theta s - \theta + \tau_s^+ - \tau_s^-}( z_s )$, for all $s$, which implies that the last product in (\ref{eq4}) can be written as
\begin{gather*}
\prod\limits_{s=1}^m{T_{q, x_s}^{\theta(s-1) + \tau_s^+ - \tau_s^-}}\left( \prod_{s=1}^m{P_{\lambda}\big(z_s; N, q, q^{\theta}\big)} \right).
\end{gather*}

Next, we need to rewrite
\begin{gather*}
\frac{\prod\limits_{s=1}^m{g_{n_s + \tau_s^+ - \tau_s^-}\big(t^{N-1}, \dots, t, 1; q, t\big)}}{Q_{(n_1, \dots, n_m)}\big(t^{N-1}, \dots, t, 1; q, t\big)}
 \end{gather*} in terms of $z_1, \dots, z_m$. From Corollary~\ref{evaluationcor} of the evaluation identity for Macdonald polynomials, we obtain
\begin{gather}
g_p\big(t^{N-1}, t^{N-2}, \dots, t, 1; q, t\big) = \prod_{s = (1, j)\in (p)}{\frac{1 - q^{a'(s) + \theta(N - l'(s))}}{1 - q^{a(s)+1 + \theta l(s)}}}\nonumber\\
\qquad {} = \prod_{i=1}^p{ \frac{1 - q^{i - 1 + \theta N}}{1 - q^{p - i + 1}} } = \prod_{j=1}^{\theta N - 1}{\frac{1 - q^{p+j}}{1 - q^j}} = \frac{1}{(1 - q)^{\theta N - 1}}\frac{\prod\limits_{j=1}^{\theta N - 1}{(1 - q^{p+ j})}}{[\theta N - 1]_q!}\label{eq6}
\end{gather}
for any $p\in\N$, $p > \theta N$. By similar, but more complicated, computations we f\/ind
\begin{gather}
Q_{(p_1, \dots, p_m)}\big(t^{N-1}, \dots, t, 1; q, t\big) \nonumber\\
\qquad {}= \frac{1}{(1 - q)^{m(\theta N - 1) - \theta m(m-1)/2}}\times\prod_{s=1}^{\theta}{ \prod_{1\leq i < j\leq m}{ \frac{q^{p_j} - q^{p_i + s + \theta(j - i - 1)}}{1 - q^{p_i + s + \theta (j - i - 1)}} } }\nonumber\\
\qquad\quad {} \times\frac{\prod\limits_{j=1}^{\theta(N-m+1)-1}{\big(1 - q^{p_m + j}\big)}\cdot \prod\limits_{j=1}^{\theta(N-m+2)-1}{\big(1 - q^{p_{m-1} + j}\big)}\cdots\prod\limits_{j=1}^{\theta N-1}{\big(1 - q^{p_1 + j}\big)}}{[\theta(N-m+1)-1]_q![\theta(N-m+2)-1]_q!\cdots [\theta N - 1]_q!},\label{eq7}
\end{gather}
for any partition $(p_1 > p_2 > \dots > p_m > 0)$ with $p_i > \theta (N - i + 1)$ for all $1\leq i\leq m$. From~(\ref{eq6}) and~(\ref{eq7}), we obtain
\begin{gather}
\frac{\prod\limits_{s=1}^m{g_{n_s + \tau_s^+ - \tau_s^-}\big(t^{N-1}, \dots, t, 1\big)}}{Q_{(n_1, \dots, n_m)}\big(t^{N-1}, \dots, t, 1; q, t\big)} = \frac{1}{(1 - q)^{\theta m(m-1)/2}}\prod_{i=1}^m{\frac{[\theta(N-i+1)-1]_q!}{[\theta N - 1]_q!}} \label{eq77}\\
{} {}\times \frac{\prod\limits_{j=1}^{\theta N - 1}{\big(1 - q^{n_1 + \tau_1^+ - \tau_1^- + j}\big)}}{\prod\limits_{j=1}^{\theta N - 1}{\big(1 - q^{n_1 + j}\big)}} \cdot \frac{\prod\limits_{j=1}^{\theta N - 1}{\big(1 - q^{n_2 + \tau_2^+ - \tau_2^- + j}\big)}}{\prod\limits_{j=1}^{\theta (N-1) - 1}{\big(1 - q^{n_2 + j}\big)}}\cdots \frac{\prod\limits_{j=1}^{\theta N - 1}{\big(1 - q^{n_m + \tau_m^+ - \tau_m^- + j}\big)}}{\prod\limits_{j=1}^{\theta (N-m+1) - 1}{\big(1 - q^{n_m + j}\big)}} \label{eq8}\\
{} \times \prod_{s=1}^{\theta}{ \prod_{1\leq i < j\leq m}{ \frac{1 - q^{n_i + s + \theta(j - i - 1)}}{q^{n_j} - q^{n_i + s + \theta(j - i - 1)}} } }.\label{eq9}
\end{gather}
Observe that we used our assumption $n_1 > n_2 > \dots > n_m > \theta (N + m)$ to guarantee that~(\ref{eq6}) and~(\ref{eq7}) are applicable. We have to rewrite both~(\ref{eq8}) and~(\ref{eq9}) in terms of $z_s$. Let us begin with~(\ref{eq8}), which is a product of~$m$ terms; for $1\leq r\leq m$, the $r^{\rm th}$ term is
\begin{gather}
\frac{\prod\limits_{j=1}^{\theta N - 1}{\big(1 - q^{n_r + \tau_r^+ - \tau_r^- + j}\big)}}{\prod\limits_{j=1}^{\theta (N - r + 1) - 1}{\big(1 - q^{n_r + j}\big)}} = \frac{\prod\limits_{j = 1}^{\theta N - 1}{\big(1 - q^{-\theta(N-r)}z_rq^{\tau_r^+ - \tau_r^- + j}\big)}}{\prod\limits_{j=1}^{\theta (N - r + 1)- 1}{\big(1 - q^{-\theta(N - r)}z_r q^j\big)}}\nonumber\\
\qquad{} = (-1)^{\theta(r-1)}q^{-\sum\limits_{j=\theta(N-r+1)}^{\theta N - 1}{(\theta(N - r)-j)}}\cdot\frac{\prod\limits_{j=1}^{\theta N - 1}{\big(z_rq^{\tau_r^+ - \tau_r^+} - q^{\theta(N - r)-j}\big)}}{\prod\limits_{j=1}^{\theta(N-r+1)-1}{\big(z_r - q^{\theta(N-r)-j}\big)}}\label{eq10}\\
\qquad{}= (-1)^{\theta(r-1)} q^{-\sum\limits_{j=\theta(N-r+1)}^{\theta N - 1}{(\theta(N - r)-j)}} q^{-(\theta r - \theta)(\theta N - 1)}\frac{\prod\limits_{j=1}^{\theta N - 1}{\big(z_rq^{\theta r - \theta + \tau_r^+ - \tau_r^-} - q^{\theta(N-1) - j}\big)}}{\prod\limits_{j=1}^{\theta(N-r+1)-1}{\big(z_r - q^{\theta(N - r)-j}\big)}}.\nonumber
\end{gather}
Under the change of indexing $j \mapsto \theta(N-r+1)-j$, the product in the denominator of~(\ref{eq10}) becomes
\begin{gather*}
\prod_{j=1}^{\theta(N-r+1)-1}{\big(z_r - q^{\theta(N - r)-j}\big)} = \prod_{j=1}^{\theta(N-r+1)-1}{\big(z_r - q^{j - \theta}\big)},
 \end{gather*} whereas the numerator of~(\ref{eq10}) can be expressed as
\begin{gather*}
\prod_{j=1}^{\theta N - 1}{\big(z_rq^{\theta r - \theta + \tau_r^+ - \tau_r^-} - q^{\theta(N-1) - j}\big)} = T_{q, z_r}^{\theta r - \theta + \tau_r^+ - \tau_r^-}\prod_{j=1}^{\theta N - 1}{\big(z_r - q^{\theta(N - 1) - j}\big)}.
\end{gather*} Def\/ine also
\begin{gather}\label{eq11}
c(N, m, \theta) = -\sum_{r=1}^m{ \sum_{j=\theta(N-r+1)}^{\theta N - 1}{(\theta(N - r)-j )} } - \sum_{r=1}^{m}{(\theta r - \theta)(\theta N - 1)},
\end{gather}
which is the power of $q$ coming from the terms~(\ref{eq10}), for $1\leq r\leq m$. Thus~(\ref{eq8}) equals
\begin{gather}\label{eq12}
\frac{(-1)^{\theta m(m-1)/2}q^{c(N, m, \theta)}}{\prod\limits_{r=1}^m{\prod\limits_{j=1}^{\theta(N - r+1)-1}{\big(z_r - q^{j-\theta}\big)}}}\cdot\prod_{r=1}^m{T_{q, z_r}^{\theta r - \theta + \tau_r^+ - \tau_r^-}}\left( \prod_{j=1}^{\theta N - 1}{\big(z_r - q^{\theta(N - 1) - j}\big)} \right).
\end{gather}

On the other hand, (\ref{eq9}) can be expressed in terms of the variables $z_i$ as follows:
\begin{gather}
\prod_{r = 1}^{\theta}{ \prod_{i < j}{ \frac{1 - q^{n_i + r + \theta(j-i-1)}}{q^{n_j} - q^{n_i + r + \theta(j-i-1)}} } } = \prod_{r = 1}^{\theta}{ \prod_{i < j}{ \frac{q^{n_i + r + \theta(j-i-1)} - 1}{ q^{n_i + r + \theta(j-i-1)} - q^{n_j}} } }\nonumber\\
\qquad{} = \prod_{r = 1}^{\theta}{ \prod_{i < j}{ \frac{q^{n_i + r + \theta(N- i -1)} - q^{\theta(N-j)}}{ q^{n_i + r + \theta(N- i -1)} - q^{n_j + \theta(N - j)}} } } = \prod_{r = 1}^{\theta}{ \prod_{i < j}{ \frac{z_i q^{r - \theta} - q^{\theta(N-j)}}{ z_i q^{ r - \theta} - z_j} } }\nonumber\\
 \qquad{} = \prod_{r = 1}^{\theta}{ \prod_{i < j}{ \frac{z_i - q^{\theta(N - j + 1) - r}}{ z_i - z_j q^{\theta - r}} } }.\label{eq13}
\end{gather}
The denominator of~(\ref{eq13}) is $\prod\limits_{\substack{ 1\leq i < j\leq m \\ 0\leq k < \theta}}{(z_i - q^k z_j)}$, and the numerator is $\prod\limits_{r=1}^m{ \prod\limits_{j=\theta(N - m + 1)}^{\theta(N-i+1)-1}{ (z_r - q^{j-\theta}) } }$. Then~(\ref{eq9}) equals
\begin{gather}\label{eq14}
\frac{\prod\limits_{r=1}^m{ \prod\limits_{j=\theta(N - m+1)}^{\theta(N-r+1)-1}{ \big(z_r - q^{j-\theta}\big) } }}{\prod\limits_{\substack{ 1\leq i < j\leq m \\ 0\leq k < \theta}}{(z_i - q^k z_j)}}.
\end{gather}
Therefore after the change of variables~(\ref{changevars}), and using (\ref{eq77}), (\ref{eq8}), (\ref{eq9}), (\ref{eq12}), (\ref{eq14}) and the identities after~(\ref{changevars}), equation (\ref{eq4}) yields the desired~(\ref{macdonaldthm3eqn}) for any $x_s = q^{n_s + \theta(N - s)}$, $1\leq s\leq m$, such that $n_1 > \dots > n_m > \theta (N + m)$, provided the equality{\samepage
\begin{gather*}
c(N, m, \theta) = \theta^2 \dbinom{m+1}{3} - \left\{ N\theta^2 - \dbinom{\theta + 1}{2} \right\}\dbinom{m}{2}
\end{gather*}
holds, where $c(N, m, \theta)$ is def\/ined in~(\ref{eq11}). This is an easy exercise.}

We have just proved that the statement of Theorem~\ref{macdonaldthm3} holds for all $x_s = q^{n_s + \theta(N - s)}$, for all $n_1 > n_2 > \dots > n_m > \theta (N + m)$. Since both sides of the equality~(\ref{macdonaldthm3eqn}) are evidently rational functions in $x_1, \dots, x_m$, an easy algebro-geometric argument shows the equality holds for all $x_1, \dots, x_m\in\C$, as desired.

We still have to extend the theorem to all signatures. Let us prove~(\ref{macdonaldthm3eqn}) for an arbitrary $\lambda\in\GT_N$. If $\lambda\in\GTp_N$, then we are done. Otherwise, choose any $p\in\N$ large enough so that $\widetilde{\lambda} \myeq (\lambda_1 + p, \lambda_2 + p, \dots, \lambda_N + p)\in\GTp_N$. By homogeneity of Macdonald polynomials, we have
\begin{gather*}
P_{\widetilde{\lambda}}\big(x_1, \dots, x_m; N, q, q^{\theta}\big) = P_{\lambda}\big(x_1, \dots, x_m; N, q, q^{\theta}\big) (x_1\cdots x_m)^p q^{-\theta p N m}q^{\theta pm(m+1)/2}
\end{gather*}
and
\begin{gather*}
P_{\widetilde{\lambda}}\big(x_i; N, q, q^{\theta}\big) = P_{\lambda}\big(x_i; N, q, q^{\theta}\big) x_i^p q^{-\theta p N}q^{\theta p} \qquad \forall\, 1\leq i\leq m\\
\Rightarrow \ T_{q, x_i}^{(i-1)\theta + \tau_i^+ - \tau_i^-}P_{\widetilde{\lambda}}\big(x_i; N, q, q^{\theta}\big) = q^{-\theta p N}q^{p(i\theta + \tau_i^+ - \tau_i^-)}x_i^p \\
\hphantom{\Rightarrow \ T_{q, x_i}^{(i-1)\theta + \tau_i^+ - \tau_i^-}P_{\widetilde{\lambda}}\big(x_i; N, q, q^{\theta}\big) =}{} \times T_{q, x_i}^{(i-1)\theta + \tau_i^+ - \tau_i^-}P_{\lambda}\big(x_i; N, q, q^{\theta}\big) \qquad \forall\, 1\leq i\leq m.
\end{gather*}
We know that~(\ref{macdonaldthm3eqn}) holds for $\widetilde{\lambda}$; from the expressions above, it holds also for $\lambda$ provided that (the powers of~$q$ match)
\begin{gather*}
-\theta p N m + \frac{\theta p m(m+1)}{2} = \sum_{i=1}^m{(-\theta p N + p (i\theta + \tau_i^+ - \tau_i^-))}.
\end{gather*}
The latter equation is easy to check, and the proof of Theorem~\ref{macdonaldthm3} is therefore f\/inished.

\section{Asymptotics of Macdonald characters}\label{sec:asymptotics}

In the remaining of the paper, starting here, we denote the Macdonald polynomial $P_{\lambda}(x_1, \dots, x_N;$ $q, t)$ simply by $P_{\lambda}(x_1, \dots, x_N)$.

In this section, assume that $\theta\in\N$ and let $t = q^{\theta}$. We study the asymptotics of (certain normalization of) Macdonald characters of a f\/ixed number $m$ of variables, as the rank $N$ tends to inf\/inity and the signatures~$\lambda(N)$ stabilize in certain way that we def\/ine next.

\begin{df}\label{def:stabilizing}
Let
\begin{gather*}
\Nu \myeq \{\nu = (\nu_1, \nu_2, \dots)\in\Z^{\infty}\colon \nu_1 \leq \nu_2 \leq \dots\}
\end{gather*}
be the set of weakly increasing integer sequences. We say that \textit{the sequence $\{\lambda(N)\}_{N\geq 1}$ of signatures, $\lambda(N)\in\GT_N$, stabilizes to $\nu\in\Nu$} if we have the following limits
\begin{gather*}
\lim_{N\rightarrow\infty}{\lambda_{N - i + 1}(N)} = \nu_i, \qquad \forall\, i = 1, 2, \dots.
\end{gather*}
In other words $\{\lambda(N)\}_{N\geq 1}$ stabilizes to $\nu$ if there exists a sequence of integers $0 < N_1 < N_2 < \cdots$ such that $\lambda_{N - i + 1}(N) = \nu_i$, $\forall\, N > N_i$, $i = 1, 2, \dots$.
\end{df}

\subsection{Asymptotics of Macdonald characters of one variable}

\begin{thm}\label{thm:asymptotics1}
Let $t = q^{\theta}$, $\theta\in\N$, and $\{\lambda(N)\}_{N\geq 1}$, $\lambda(N)\in\GT_N$, be a sequence of signatures that stabilizes to $\nu\in\Nu$. Then
\begin{gather}\label{eqn:asymptotics1}
\lim_{N\rightarrow\infty}{\frac{P_{\lambda(N)}\big(x, t^{-1}, \dots, t^{1 - N}\big)}{P_{\lambda(N)}\big(1, t^{-1}, t^{-2}, \dots, t^{1 - N}\big)}} = \Phi^{\nu}(x; q, t),
\end{gather}
where
\begin{gather}\label{eqn:Phi1}
\Phi^{\nu}(x; q, t) \myeq \frac{(q; q)_{\infty}}{(xq; q)_{\infty}}\frac{\ln{q}}{2\pi\sqrt{-1}}\int_{\CC^+}{x^z\prod_{i=1}^{\infty}{\frac{\big(q^{-z+\nu_i}t^i; q\big)_{\infty}}{\big(q^{-z+\nu_i}t^{i-1}; q\big)_{\infty}}}{\rm d}z},
\end{gather}
and the contour $\CC^+$ is the infinite positive contour consisting of $\big[{-}M + \frac{\pi\sqrt{-1}}{\ln{q}}, -M - \frac{\pi\sqrt{-1}}{\ln{q}}\big]$ and the lines $\big[ {-}M+\frac{\pi\sqrt{-1}}{\ln{q}}, +\infty + \frac{\pi\sqrt{-1}}{\ln{q}}\big)$, $\big[ {-}M-\frac{\pi\sqrt{-1}}{\ln{q}}, +\infty-\frac{\pi\sqrt{-1}}{\ln{q}} \big)$ for an arbitrary $M > \max\{0, -\nu_1\}$, see Fig.~{\rm \ref{fig:C}}. The function $\Phi^{\nu}(x; q, t)$ is defined by the formula above in the domain
\begin{gather*}
\mathcal{U} \myeq \bigcap_{k\geq 1}{\big\{x\neq q^{-k}\big\}}\cap{\{x \neq 0\}},
\end{gather*}
and admits an analytic continuation to the domain $\{x \neq 0\} = \C \setminus \{0\}$. Moreover if $\nu_1 \geq 0$, the function $\Phi^{\nu}(x; q, t)$ can be analytically continued to~$\C$.

The convergence~\eqref{eqn:asymptotics1} is uniform on compact subsets of $\C \setminus \{0\}$, and if $\nu_1 \geq 0$, then it is uniform on compact subsets of $\C$.
\end{thm}

\begin{rem}The theorem is a direct application of the integral representation in Theorem~\ref{macdonaldthm1}. If we use instead Theorems~\ref{macdonaldthm2} and~\ref{macdonaldthm25}, we could possibly extend the theorem for general $q, t\in(0, 1)$. Since our main application does require $\theta\in\N$ and $t = q^{\theta}$, we do not bother to pursue the more general case $q, t\in(0, 1)$.
\end{rem}

\begin{proof}Let us f\/irst prove that the limit~(\ref{eqn:asymptotics1}) holds uniformly for $x$ in compact subsets of $\U$. We use the integral representation for Macdonald characters of one variable, Theorem~\ref{macdonaldthm1}, for the signature $\lambda(N)\in\GT_N$, and with~$x$ replaced by $xt^N$, $x \in \U$. Since $P_{\lambda(N)}$ is a homogeneous polynomial of degree $|\lambda(N)|$, then the left-hand side is
\begin{gather*}
 t^{|\lambda(N)|}\frac{P_{\lambda(N)}\big(x, t^{-1}, t^{-2}, \dots, t^{1-N}\big)}{P_{\lambda(N)}\big(1, t^{-1}, t^{-2}, \dots, t^{1-N}\big)}.
\end{gather*}
After simple algebraic manipulations, the right-hand side becomes
\begin{gather*}
q^{-{\theta N \choose 2}}\frac{(q; q)_{\theta N - 1}}{(xq; q)_{\theta N - 1}}\frac{\ln{q}}{2\pi\sqrt{-1}}\oint_{\CC_0}{\frac{x^z}{\prod\limits_{i=1}^N\prod\limits_{j=0}^{\theta - 1}{\big(q^{-(\lambda_i(N) + \theta(N-i) + j)} - q^{-z}\big)}}}{\rm d}z.
\end{gather*}
Therefore we conclude (make change of variables $i\mapsto N - i + 1$ in the inner product)
\begin{gather}
\frac{P_{\lambda(N)}\big(x, t^{-1}, t^{-2}, \dots, t^{1-N}\big)}{P_{\lambda(N)}\big(1, t^{-1}, t^{-2}, \dots, t^{1-N}\big)} \nonumber\\
\qquad{} = \frac{(q; q)_{\theta N - 1}}{(xq; q)_{\theta N - 1}}\frac{\ln{q}}{2\pi\sqrt{-1}}\oint_{\CC_0}{\frac{x^z}{\prod\limits_{i=1}^N\prod\limits_{j=0}^{\theta - 1}{\big(1 - q^{-z}q^{\lambda_{N-i+1}(N) + \theta(i - 1) + j}\big)}}{\rm d}z},\label{eq:beforelimit}
\end{gather}
where $\CC_0$ is a f\/inite contour encloses all real poles of the integrand and no other poles. We have the limit
\begin{gather*}
\lim_{N\rightarrow\infty}{\frac{(q; q)_{\theta N - 1}}{(xq; q)_{\theta N - 1}}} = \frac{(q; q)_{\infty}}{(xq; q)_{\infty}},
\end{gather*}
uniformly for $x$ belonging to compact subsets of $\mathcal{U}$.

Next modify the contour $\CC_0$ into an inf\/inite contour $\CC^+$ as described in the statement of the theorem. This is possible to do because all the poles of the integrands of~(\ref{eq:beforelimit}) belong to the interior of $\CC^+$, as~$N$ grows. Moreover the resulting integral is well-posed for large enough $N$ since the integrand is of order $\big(|x|q^{\theta N}\big)^{\Re z}$ and for any compact set~$K\subset\C$, there exists $N_0\in\N$ such that $\sup\limits_{N>N_0}\sup\limits_{x\in K}{\big\{|x|q^{\theta N}\big\}} < 1$.

We now look at the asymptotics of the integral in~(\ref{eq:beforelimit}), with $\CC_0$ replaced by $\CC^+$. The denominator in the integrand has the following limit
\begin{gather*}
\lim_{N\rightarrow\infty}{\prod_{i=1}^N\prod_{j=0}^{\theta-1}{\big(1 - q^{-z + \lambda_{N-i+1}(N) + \theta(i-1) + j} \big) }}\\
 \qquad{} = \prod_{i=1}^{\infty}\prod_{j=0}^{\theta - 1}{\big( 1 - q^{-z + \nu_i + \theta(i-1) + j} \big)} = \prod_{i=1}^{\infty}{\frac{\big(q^{-z + \nu_i}t^{i-1}; q\big)_{\infty}}{\big(q^{-z+\nu_i}t^i; q\big)_{\infty}}}.
\end{gather*}
The limit above can be justif\/ied properly by using the dominated convergence theorem and the estimate
\begin{gather*}
\sum_{i=1}^N\sum_{j=0}^{\theta-1}{\big| q^{-z + \lambda_{N-i+1}(N) + \theta(i-1)+j} \big|} \\
\qquad{} \leq \sum_{i=1}^N\sum_{j=0}^{\theta-1}{q^{-\Re z + \lambda_N(N) + \theta(i-1)+j}} \leq c\sum_{i=1}^N\sum_{j=0}^{\theta - 1}{q^{\theta(i-1)+j}} < \frac{c}{1 - q},
\end{gather*}
valid for some constant $c>0$ such that $q^{-\Re z + \lambda_N(N)} < c$ for all $N\geq 1$ (it exists because $\lim\limits_{N\rightarrow\infty}{\lambda_N(N)} = \nu_1$).

To prove the limit in the statement of the theorem, we are left to show
\begin{gather*}
\lim_{N\rightarrow\infty}{\int_{\CC^+}\frac{x^z}{\prod\limits_{i=1}^N\prod\limits_{j=0}^{\theta-1}{\big(1 - q^{-z}q^{\lambda_{N-i+1}(N) + \theta(i-1)+j}\big)}}{\rm d}z} = \int_{\CC^+}{x^z \prod_{i=1}^{\infty}{\frac{\big(q^{-z+\nu_i}t^i; q\big)_{\infty}}{\big(q^{-z+\nu_i}t^{i-1}; q\big)_{\infty}}} {\rm d}z}.
\end{gather*}
We already proved the pointwise convergence; to make use of the dominated convergence theorem, we simply need estimates on the contribution of the tails of~$\CC^+$ that are uniform in $N$ and uniform for $x$ belonging to compact subsets of $\C\setminus\{0\}$. Parametrize the tails of $\CC^+$ as $z = r + \frac{\pi\sqrt{-1}}{\ln{q}}$ or $z = r - \frac{\pi\sqrt{-1}}{\ln{q}}$; for $r$ ranging from some large $R > 1$ to $+\infty$, we want to show that the contribution of each of these lines is small. We have
\begin{gather*}
\left| \frac{x^z}{\prod\limits_{i=1}^N\prod\limits_{j=0}^{\theta-1}{\big(1 - q^{-z}q^{\lambda_{N-i+1}(N) + \theta(i-1)+j}\big)}} \right| = \frac{\big|x^{r \pm \frac{\pi\sqrt{-1}}{\ln{q}}} \big|}{\prod\limits_{i=1}^N\prod\limits_{j=0}^{\theta-1}{\big(1 + q^{-r}q^{\lambda_{N-i+1}(N) + \theta(i-1)+j}\big)}}\\
\qquad{} \leq C_1\times \frac{|x|^r}{\prod\limits_{i=1}^k\prod\limits_{j=0}^{\theta-1}{\big(1 + q^{-r}q^{\lambda_{N-i+1}(N) + \theta(i-1)+j}\big)}}\\
\qquad{} \leq C_2 \times \frac{|x|^r}{\prod\limits_{i=1}^k\prod\limits_{j=0}^{\theta-1}{(1 + q^{-r})}} \leq C_2 \times \frac{|x|^r}{(q^{-r})^{\theta k}} = C_2 \times \big( |x| t^k \big)^r,
\end{gather*}
for all $1\leq k \leq N$, where the constant in the second line is uniform for $x$ over compact subsets of $\C\setminus\{0\}$, and the constant in the third line depends on $k$, but not on $N$ (note that to go from the second to the third line, we needed to use that $\lim\limits_{N\rightarrow\infty}{\lambda_{N-i+1}(N)}$ exists for $i = 1, \dots, k$ and so the sequence $\{\lambda_{N-i+1}(N)\}_{N\geq 1}$ is uniformly bounded for $i = 1, \dots, k$). Now for any compact set $K\subset\C\setminus\{0\}$ and any $M>0$, there exists $k\in\N$ such that $\sup\limits_{x\in K}{|x|}\cdot t^k \leq e^{-M}$. Since $\int_R^{\infty}{e^{-Mr}{\rm d}r} = e^{-MR}/M \xrightarrow{R\rightarrow\infty} 0$, we have shown that the contribution of the tails of $\CC^+$ is uniformly small over~$x$ in compact subsets of $\C\setminus\{0\}$ and for large enough~$N$.

We have proved so far that the limit in the theorem holds uniformly for $x$ in compact subsets of $\U$. Since the set $\{q^{-1}, q^{-2}, \dots\}$ is discrete and has no accumulation points, Cauchy's integral formula allows us to deduce the uniform convergence in compact subsets of $\C\setminus\{0\}$ as soon as we show that $\Phi^{\nu}(x; q, t)$ admits an analytic continuation to $\C\setminus\{0\}$.

By virtue of Riemann's theorem of removable singularities, it will suf\/f\/ice to show that $\Phi^{\nu}(x; q, t)$ is uniformly bounded in an open neighborhood of each pole $q^{-k}$. Let $R> 0$, $R\notin\{q^n\colon n\in\Z\}$, be arbitrary. From what we have shown so far, it is clear that
\begin{gather*}
\sup_{\substack{\frac{1}{R} \leq |x| \leq R \\ x\in \U}}\big| \Phi^{\nu}(x; q, t) \big|
\leq \sup_{N\geq 1}\sup_{\substack{\frac{1}{R} \leq |x| \leq R \\ x\in \U}}{\left|\frac{P_{\lambda(N)}\big(x, t^{-1}, \dots, t^{1-N}\big)}{P_{\lambda(N)}\big(1, t^{-1}, \dots, t^{1-N}\big)}\right|}\\
\hphantom{\sup_{\substack{\frac{1}{R} \leq |x| \leq R \\ x\in \U}}\big| \Phi^{\nu}(x; q, t) \big|}{}
\leq \sup_{N\geq 1}\sup_{\frac{1}{R} \leq |x| \leq R}{\left|\frac{P_{\lambda(N)}\big(x, t^{-1}, \dots, t^{1-N}\big)}{P_{\lambda(N)}\big(1, t^{-1}, \dots, t^{1-N}\big)}\right|}.
\end{gather*}
Thanks to the branching rule for Macdonald polynomials, Theorem~\ref{branchingmacdonald}, the fact that all the branching coef\/f\/icients $\psi_{\mu/\nu}(q, t)$ are nonnegative when $q, t\in (0, 1)$, and the fact that each $P_{\lambda(N)}\big(x, t^{-1}, \dots, t^{1-N}\big)$ is a Laurent polynomial in~$x$, we have
\begin{gather*}
\sup_{\frac{1}{R} \leq |x| \leq R}{\left|\frac{P_{\lambda(N)}\big(x, t^{-1}, \dots, t^{1-N}\big)}{P_{\lambda(N)}\big(1, t^{-1}, \dots, t^{1-N}\big)}\right|} \leq \sup_{\frac{1}{R} \leq |x| \leq R}{\frac{P_{\lambda(N)}\big(|x|, t^{-1}, \dots, t^{1-N}\big)}{P_{\lambda(N)}\big(1, t^{-1}, \dots, t^{1-N}\big)}}\\
\hphantom{\sup_{\frac{1}{R} \leq |x| \leq R}{\left|\frac{P_{\lambda(N)}\big(x, t^{-1}, \dots, t^{1-N}\big)}{P_{\lambda(N)}\big(1, t^{-1}, \dots, t^{1-N}\big)}\right|} }{}
\leq \frac{P_{\lambda(N)}\big(R, t^{-1}, \dots, t^{1-N}\big)}{P_{\lambda(N)}\big(1, t^{-1}, \dots, t^{1-N}\big)}\!
+\! \frac{P_{\lambda(N)}\big(R^{-1}, t^{-1}, \dots, t^{1-N}\big)}{P_{\lambda(N)}\big(1, t^{-1}, \dots, t^{1-N}\big)}.
\end{gather*}
From the pointwise limits
\begin{gather*}
\lim_{N\rightarrow\infty}{\frac{P_{\lambda(N)}\big(R^{\pm 1}, t^{-1}, \dots, t^{1-N}\big)}{P_{\lambda(N)}\big(1, t^{-1}, \dots, t^{1-N}\big)}} = \Phi^{\nu}(R^{\pm 1}; q, t),
\end{gather*}
we deduce that the sequences $\big\{ P_{\lambda(N)}\big(R^{\pm 1}, t^{-1}, \dots, t^{1-N}\big)/P_{\lambda(N)}\big(1, t^{-1}, \dots, t^{1-N}\big) \big\}_{N\geq 1}$ are uniformly bounded. As a result of the estimates above,
\begin{gather*} \sup\limits_{\substack{\frac{1}{R} \leq |x| \leq R \\ x\in\U}}{|\Phi^{\nu}(x; q, t)|} < \infty.
 \end{gather*} Thus $\Phi^{\nu}(x; q, t)$ admits an analytic continuation to all the poles in $\big\{z\in\C\colon \frac{1}{R} \leq |z|\leq R\big\}\cap\big\{q^{-1}, q^{-2}, \dots\big\}$. Since $R>0$ was arbitrary (outside of a lattice), we conclude that $\Phi^{\nu}(x; q, t)$ admits an analytic continuation to $\C\setminus\{0\}$.

If $\nu_1 \geq 0$, then $\lim\limits_{N\rightarrow\infty}{\lambda_N(N)} = \nu_1$ shows that $\lambda_N(N)\geq 0$ for all $N > N_0$ and $N_0\in\N$ large enough. For all $N>N_0$, the functions ${P_{\lambda(N)}\big(x, t^{-1}, \dots, t^{1-N}\big)}/{P_{\lambda(N)}\big(1, t^{-1}, \dots, t^{1-N}\big)}$ are polynomials in~$x$ and therefore holomorphic on $\C$. Similar considerations as above allow us to analytically continue $\Phi^{\nu}(x; q, t)$ to $\C$ in this case, and also show that the limit~(\ref{eqn:asymptotics1}) holds uniformly for~$x$ belonging to compact subsets of~$\C$.
\end{proof}

We end this subsection with a Lemma that will be used in Section~\ref{sec:qtboundary}.

\begin{lem}\label{lem:expansionintegrals}
If $\nu, \widetilde{\nu} \in\Nu$ are such that
\begin{gather}\label{eqn:phiinjective}
\Phi^{\nu}(x; q, t) = \Phi^{\widetilde{\nu}}(x; q, t) \qquad \forall \, x \in \TT,
\end{gather}
then $\nu = \widetilde{\nu}$.
\end{lem}
\begin{proof}
In the integral representation of $\Phi^{\nu}(x; q, t)$, the set of poles of the integrand is
\[
\bigcup_{r=1}^{\infty}\bigcup_{s=0}^{\theta-1}{\{\nu_r + \theta(r-1) + s\}}
\]
and they are all enclosed by contour $\CC^+$. We can therefore expand
\begin{gather*}
\frac{\ln{q}}{2\pi\sqrt{-1}}\int_{\CC^+}{x^z\prod_{i=1}^{\infty}{\frac{(q^{-z+\nu_i}t^i; q)_{\infty}}{(q^{-z+\nu_i}t^{i-1}; q)_{\infty}}}{\rm d}z},
\end{gather*}
at least formally, as the sum of residues
\begin{gather}\label{eq:sumresidues}
\sum_{r=1}^{\infty}\sum_{s=0}^{\theta-1}
\left\{ x^{\nu_r + \theta(r-1) + s} \prod_{\substack{i\geq 1 \\ i \neq r}}{\frac{\big(q^{-\nu_r+\nu_i-s}t^{i-r+1}; q\big)_{\infty}}{\big(q^{-\nu_r + \nu_i - s}t^{i-r}; q\big)_{\infty}}} \times \prod_{\substack{\theta > j \geq 0 \\ j \neq s}}{\frac{1}{(1 - q^{j-s})}} \right\}.
\end{gather}
The series above converges absolutely for all $x\in\C\setminus\{0\}$, and uniformly for $x$ in compact subsets of $\C\setminus\{0\}$; in fact, we can argue as in the special case $\theta = 1$ of \cite[Section~6.2]{GP}. We have the bounds
\begin{gather*}
\prod_{\substack{\theta > j \geq 0 \\ j \neq s}}{ \left| \frac{1}{(1 - q^{j-s})} \right| }
\leq \frac{1}{(1 - q)\cdots \big(1 - q^{\theta}\big)\big(q^{-1} - 1\big)\cdots \big(q^{-\theta} - 1\big)};\\
\prod_{i > r}{\left| \frac{\big(q^{-\nu_r+\nu_i-s}t^{i-r+1}; q\big)_{\infty}}{\big(q^{-\nu_r + \nu_i - s}t^{i-r}; q\big)_{\infty}} \right|} =
\prod_{i > r} \prod_{j = 0}^{\theta - 1} {\left| \frac{1}{\big(1 - q^{-\nu_r + \nu_i - s + \theta(i - r)+j}\big)} \right|}\\
\qquad {} = \prod_{i > r} \prod_{j = 0}^{\theta - 1} { \frac{1}{\big(1 - q^{-\nu_r + \nu_i - s + \theta(i - r)+j}\big)} }
\qquad{} \leq \prod_{i > r} \prod_{j = 0}^{\theta - 1} { \frac{1}{\big(1 - q^{- (\theta-1) + \theta(i - r)+j}\big)} } = \frac{1}{(q; q)_{\infty}};\\
\prod_{i < r}{\left| \frac{\big(q^{-\nu_r+\nu_i-s}t^{i-r+1}; q\big)_{\infty}}{\big(q^{-\nu_r + \nu_i - s}t^{i-r}; q\big)_{\infty}} \right|} =
\prod_{i < r} \prod_{j = 0}^{\theta - 1} { \frac{q^{\nu_r + \theta (r - 1) + s - \nu_i - \theta(i - 1) - j}}
{1 - q^{\nu_r - \nu_i + \theta(r - i) - j + s}} }\\
\qquad{}
\leq \frac{1}{(q; q)_{\infty}} \prod_{i < r} \prod_{j = 0}^{\theta - 1} {
q^{\nu_r + \theta (r - 1) + s - \nu_i - \theta(i - 1) - j} } \\
\qquad{} \leq \frac{1}{(q; q)_{\infty}} \prod_{i = 1}^m \prod_{j = 0}^{\theta - 1} {
q^{\nu_r + \theta (r - 1) + s - \nu_i - \theta(i - 1) - j} }, \qquad \textrm{ for any} \quad 1 \leq m < r,\\
\qquad {} = q^{m\theta(\nu_r + \theta(r - 1) + s)} \times \frac{q^{-\theta (\nu_1 + \dots + \nu_m)}q^{-{\theta m \choose 2}} }{(q; q)_{\infty}}.
\end{gather*}
Choose an arbitrary $m\in\N$ and f\/ix it.
For $r > m$, the general term in brackets at~(\ref{eq:sumresidues}) has modulus upper bounded by
\[
\big( |x|q^{m\theta} \big)^{\nu_r + \theta(r - 1) + s} \times c(m, \theta; q),
\]
where $c(m, \theta; q) > 0$ is a constant depending on $m, \theta$, as well as $\nu_1, \dots, \nu_m$, but not on $r$.
It follows that the sum~(\ref{eq:sumresidues}) converges absolutely for any $x\in\C\setminus \{0\}$ with $|x| < q^{-m\theta}$.
Since $m\in\N$ was arbitrary, it follows that~(\ref{eq:sumresidues}) converges absolutely for any $x\in\C\setminus\{0\}$.
The uniform convergence also follows from the bound above.
In particular, (\ref{eq:sumresidues}) is the Fourier expansion of $\frac{(xq; q)_{\infty}}{(q; q)_{\infty}}\Phi^{\nu}(x; q, t)$.
A similar Fourier expansion can be given for $\frac{(xq; q)_{\infty}}{(q; q)_{\infty}}\Phi^{\widetilde{\nu}}(x; q, t)$.
After multiplying equality~(\ref{eqn:phiinjective}) by $(xq; q)_{\infty}/(q; q)_{\infty}$, we have
\begin{gather}\label{eqn:equalityintegrals}
\frac{\ln{q}}{2\pi\sqrt{-1}}\int_{\CC^+}{x^z\prod_{i=1}^{\infty}{\frac{\big(q^{-z+\nu_i}t^i; q\big)_{\infty}}{\big(q^{-z+\nu_i}t^{i-1}; q\big)_{\infty}}}{\rm d}z}
= \frac{\ln{q}}{2\pi\sqrt{-1}}\int_{\CC^+}{x^z\prod_{i=1}^{\infty}{\frac{\big(q^{-z+\widetilde{\nu}_i}t^i; q\big)_{\infty}}{\big(q^{-z+\widetilde{\nu}_i}t^{i-1}; q\big)_{\infty}}}{\rm d}z}.
\end{gather}
We can expand both sides of~(\ref{eqn:equalityintegrals}) as above, to get an equality of the form $\sum\limits_{k\in\Z}{c_k(\nu) x^k} = \sum\limits_{k\in\Z}{c_k(\widetilde{\nu}) x^k}$.
More precisely, (\ref{eq:sumresidues}) gives that the set $\{k\in\Z\colon c_k(\nu) \neq 0\}$ of indices which appear in the expansion of the left-hand side of~(\ref{eqn:equalityintegrals}) is $\{\nu_r + \theta(r-1) + s\colon 1\leq r, \, 0\leq s\leq \theta-1\}$. Similarly, $\{k\in\Z\colon c_k(\widetilde{\nu}) \neq 0\} = \{\widetilde{\nu}_r + \theta(r-1) + s\colon 1\leq r, \, 0\leq s\leq \theta-1\}$. The equality of these sets, and the inequalities $\nu_1 \leq \nu_2 \leq \cdots$, $\widetilde{\nu}_1 \leq \widetilde{\nu}_2 \leq \cdots$, imply that $\nu_r = \widetilde{\nu}_r$ for all $r\geq 1$, i.e., $\nu = \widetilde{\nu}$.
\end{proof}

\subsection[Asymptotics of Macdonald characters of a f\/ixed number $m$ of variables]{Asymptotics of Macdonald characters of a f\/ixed number $\boldsymbol{m}$ of variables}

The following theorem is the main result for asymptotics of Macdonald characters of any given rank $m\in\N$ as $N$ tends to inf\/inity.

\begin{thm}\label{thm:asymptotics2}
Let $\theta\in\N$, $t = q^{\theta}$, and $\{\lambda(N)\}_{N\geq 1}$, $\lambda(N)\in\GT_N$, be a sequence of signatures that stabilizes to $\nu\in\Nu$. Also let $m\in\N$ be arbitrary. Then
\begin{gather}\label{eqn:asymptotics2}
\lim_{N\rightarrow\infty}{\frac{P_{\lambda(N)}\big(x_1, \dots, x_m, t^{-m}, \dots, t^{1 - N}\big)}{P_{\lambda(N)}\big(1, t^{-1}, t^{-2}, \dots, t^{1 - N}\big)}} = \Phi^{\nu}(x_1, \dots, x_m; q, t),
\end{gather}
where
\begin{gather}
\Phi^{\nu}(x_1, \dots, x_m; q, t) \myeq \frac{q^{-2\theta^2 {m \choose 3} - {\theta+1 \choose 2}{m \choose 2}}}{\prod\limits_{i=1}^m{\big(x_iqt^{m-1}; q\big)_{\infty}}}\frac{1}{\prod\limits_{\substack{1\leq i < j\leq m \\ 0\leq k < \theta}}{(q^kx_j - x_i)}}\nonumber\\
\hphantom{\Phi^{\nu}(x_1, \dots, x_m; q, t) \myeq}{}\times
\widetilde{D}_{q, \theta}^{(m)}\left\{ \prod_{i=1}^m{\Phi^{\nu}(x_i; q, t)(x_iq; q)_{\infty}} \right\},\nonumber\\
\widetilde{D}_{q, \theta}^{(m)} \myeq \sum_{\tau\in M_{\theta}^{(m)}}{C_{\tau}^{(q, q^{\theta})}(x_1, \dots, x_m)\prod_{i=1}^m{T_{q, x_i}^{(i-1)\theta + \tau_i^+ - \tau_i^-}}}.\label{def:Phinu}
\end{gather}
Above we used the notation of Section~{\rm \ref{macdonaldmultiplicativesec}}: $M_{\theta}^{(m)}$ is the set of all strictly upper triangular matrices with entries in $\{0, 1, \dots, \theta\}$, and for $\tau\in M_{\theta}^{(m)}$, we let $\tau_i^+ \myeq \sum\limits_{j > i}{\tau_{i, j}}$, $\tau_i^- \myeq \sum\limits_{k < i}{\tau_{k, i}}$. Also for $\tau\in M_{\theta}^{(m)}$, let $C_{\tau}^{(q, q^{\theta})}(x_1, \dots, x_m)$ be the rational functions defined in~\eqref{eq:Ccoeff}. The function $\Phi^{\nu}(x_1, \dots, x_m; q, t)$ is defined by the formula above in the domain
\begin{gather*}
\mathcal{U}_m \myeq \bigcap_{k\in\Z}{\bigcap_{i<j}\big\{x_i \neq q^kx_j\big\}}\cap\bigcap_{k\geq 1}{\bigcap_{i=1}^m{\big\{x_i \neq q^{-k}t^{1-m}\big\}}}\cap\bigcap_{i=1}^m{\{x_i \neq 0\}},
\end{gather*}
and admits an analytic continuation to the domain $\bigcap\limits_{i=1}^m{\{x_i \neq 0\}}=(\C \setminus \{0\})^m$. Moreover if $\nu_1 \geq 0$, the function $\Phi^{\nu}(x_1, \dots, x_m; q, t)$ can be analytically continued to $\C^m$.

The convergence~\eqref{eqn:asymptotics2} is uniform on compact subsets of $(\C \setminus \{0\})^m$ and if $\nu_1\geq 0$, then it is uniform on compact subsets of $\C^m$.
\end{thm}

\begin{rem}
For $\theta = 1$, our theorem has a dif\/ferent form than \cite[Theorem~6.5]{GP}. It is not immediately clear that the two answers are the same.
\end{rem}

\begin{rem}
From their def\/inition, it is clear that the rational functions $\big\{C_{\tau}^{(q, q^{\theta})}(x_1, \dots, x_m)\colon$ $\tau\in M_{\theta}^{(m)}\big\}$ are holomorphic in a domain of the form $\bigcap\limits_{-N_1 \leq k \leq N_1}{\bigcap\limits_{i<j}\{x_i \neq q^kx_j\}}$, for large enough $N_1\in\N$. In particular, all functions $\big\{C_{\tau}^{(q, q^{\theta})}(x_1, \dots, x_m)\colon \tau\in M_{\theta}^{(m)}\big\}$ are holomorphic on~$\U_m$.
\end{rem}

\begin{proof}
This result is a consequence of Theorem~\ref{thm:asymptotics1} and the multiplicative formula for Macdonald polynomials, Theorem ~\ref{macdonaldthm3}. Let us give more details. As before we prove f\/irst the uniform limit on compact subsets of $\U_m$.
Begin by applying Theorem~\ref{macdonaldthm3} for the signature $\lambda(N)\in\GT_N$ and $t^{N-1}x_i$ instead of $x_i$, for $i = 1, \dots, m$. Since $P_{\lambda(N)}$ is a homogeneous Laurent polynomial of degree~$|\lambda(N)|$, the resulting left-hand side is
\begin{gather}\label{eq:lhs}
\frac{P_{\lambda(N)}\big(x_1, \dots, x_m, t^{-m}, \dots, t^{1-N}\big)}{P_{\lambda(N)}\big(1, t^{-1}, t^{-2}, \dots, t^{1-N}\big)}.
\end{gather}
As for the right side, the factor
\begin{gather*}
\prod_{\substack{1\leq i<j\leq m \\ 0\leq k<\theta}}{\big(t^{N-1}x_i - t^{N-1}x_jq^k\big)^{-1}} \times (q-1)^{-\theta{m \choose 2}}\times\prod_{i=1}^m{\frac{[\theta(N-i+1)-1]_q!}{[\theta N - 1]_q!}}
\end{gather*} equals
\begin{gather}\label{eq:rhs1}
(-1)^{\theta {m \choose 2}}\frac{q^{-\theta^2(N-1){m \choose 2}}}{\prod_{\substack{1\leq i<j\leq m \\ 0\leq k<\theta}}{\big(x_i - q^kx_j\big)}}\prod_{i=1}^m{\frac{1}{\big(t^{N-i+1}; q\big)_{\theta(i-1)}}}.
\end{gather}
We can also obtain easily the polynomial equality
\begin{gather*}
\prod_{j=1}^k{\big(t^{N-1}x - q^{j-\theta}\big)} = (-1)^kq^{{k+1 \choose 2} - \theta k}\prod_{i=1}^k{\big(1 - xq^{\theta N - i}\big)}, \qquad \text{for any}\quad k\in\N.
\end{gather*} Therefore for any $k\in\N$, we have
\begin{gather}\label{eq:rhs2}
\prod_{i=1}^m\prod_{j=1}^{k - 1}{\big(x_it^{N-1} - q^{j - \theta}\big)} = (-1)^{m(k - 1)}q^{m\left( {k\choose 2} - \theta(k - 1)\right)}\prod_{i=1}^m{\big(x_iq^{\theta N - k+1}; q\big)_{k-1}}.
\end{gather}
Observe also that the rational functions are invariant under the simultaneous transformations $x_i \mapsto t^{N-1}x_i$ $\forall\, i = 1, 2, \dots, m$, that is
\begin{gather}\label{eq:rhs3}
C_{\tau}^{(q, q^{\theta})}\big(t^{N-1}x_1, \dots, t^{N-1}x_m\big) = C_{\tau}^{(q, q^{\theta})}(x_1, \dots, x_m).
\end{gather}
By combining (\ref{eq:lhs}), (\ref{eq:rhs1}), (\ref{eq:rhs2}) for $k = \theta(N-m+1), \theta N$, and (\ref{eq:rhs3}), we obtain
\begin{gather*}
\frac{P_{\lambda(N)}\big(x_1, \dots, x_m, t^{-m}, \dots, t^{1-N}\big)}{P_{\lambda(N)}\big(1, t^{-1}, t^{-2}, \dots, t^{1-N}\big)}
 = (-1)^{\theta {m \choose 2}}q^{-2\theta^2 {m \choose 3} - {\theta + 1\choose 2}{m \choose 2}}\times\frac{1}{\prod\limits_{i=1}^m{\big(t^{N-i+1}; q\big)_{\theta(i-1)}}}\\
 \qquad{}
\times\frac{1}{\prod\limits_{\substack{ 1\leq i < j\leq m \\ 0\leq k<\theta }}{\big(x_i - q^kx_j\big)}}\frac{1}{\prod\limits_{i=1}^m{\big(x_i t^{m-1}q; q\big)_{\theta(N-m+1)-1}}}\\
\qquad{}\times \widetilde{D}_{q, \theta}^{(m)}\left\{ \prod\limits_{i=1}^m{\frac{P_{\lambda(N)}\big(x_i, t^{-1}, \dots, t^{1-N}\big)}{P_{\lambda(N)}\big(1, t^{-1}, \dots, t^{1-N}\big)} (x_i q; q)_{\theta N - 1} } \right\}.
\end{gather*}
The following limits hold uniformly for $(x_1, \dots, x_m)$ belonging to compact subsets of $\U_m$
\begin{gather*}
\lim_{N\rightarrow\infty}{\prod_{i=1}^m{\big(t^{N-1+1}; q\big)}_{\theta(i-1)}} = 1, \qquad \lim_{N\rightarrow\infty}{\frac{1}{\prod\limits_{i=1}^m{\big(x_i t^{m-1}q; q\big)_{\theta(N-m+1)-1}}}} = \frac{1}{\prod\limits_{i=1}^m{\big(x_it^{m-1}q; q\big)_{\infty}}}.
\end{gather*}
We have moreover the following limit holds uniformly for $(x_1, \dots, x_m)$ belonging to compact subsets of $ (\C\setminus\{0\} )^m$, because of Theorem~\ref{thm:asymptotics1},
\begin{gather*}
\lim_{N\rightarrow\infty}{\prod_{i=1}^m{\frac{P_{\lambda(N)}\big(x_i, t^{-1}, \dots, t^{1-N}\big)}{P_{\lambda(N)}\big(1, t^{-1}, \dots, t^{1-N}\big)} (x_i q; q)_{\theta N - 1}}} = \prod_{i=1}^m{\Phi^{\nu}(x_i; q, t)(x_iq; q)_{\infty}}.
\end{gather*}
It is not dif\/f\/icult to observe that if $U\subset\C^m$ is a domain preserved by the map of multiplication by~$q$, and $\{f_n\}_{n\geq 1}$, $f$ are sequences of holomorphic functions on $U$ for which $\lim\limits_{n\rightarrow\infty}{f_n(x)} = f(x)$ uniformly on compact subsets of $U$, then
\begin{gather*}
\lim_{n\rightarrow\infty}{T_{q, x}^sf_n(x)} = \lim_{n\rightarrow\infty}{f_n(q^s x)} = f(q^sx) = T_{q, x}^sf(x)
\end{gather*}
uniformly for $x$ belonging to compact subsets of $U$. As an implication, the order of the limit as $N\rightarrow\infty$ and the $q$-dif\/ference operator $\widetilde{D}_{q, \theta}^{(m)}$ can be interchanged. All the considerations above immediately imply the desired uniform limit for $(x_1, \dots, x_m)$ belonging to compact subsets of~$\U_m$.

As in the proof of Theorem~\ref{thm:asymptotics1}, the limit in the statement will hold also uniformly for compact subsets of $(\C\setminus\{0\})^m$ if we show that $\Phi^{\nu}(x_1, \dots, x_m; q, t)$ admits an analytic continuation to this larger domain. The extension of Riemann's theorem for removable singularities for several complex variables, \cite[Theorem~8]{R}, shows that $\Phi^{\nu}(x_1, \dots, x_m; q, t)$ admits an analytic continuation to all $\left(\{z\in\C \setminus\{0\}\colon |z| \leq R\}\right)^m$ if we showed only that $\Phi^{\nu}(x_1, \dots, x_m; q, t)$ is bounded on $(\{z\in\C\setminus\{0\}\colon |z|\leq R\})^m\cap\U_m$. The latter can be proved by repeating the argument in the proof of Theorem~\ref{thm:asymptotics1} above.

Finally, if $\nu_1 \geq 0$ then $\lim\limits_{N\rightarrow\infty}{\lambda_N(N)} = \nu_1$ implies $\lambda(N)_N \geq 0$ for large enough $N$. It follows that $P_{\lambda(N)}\big(x_1, \dots, x_m, t^{-m}, \dots, t^{1-N}\big)/P_{\lambda(N)}\big(1, t^{-1}, \dots, t^{1-N}\big)$ is a polynomial in $x_1, \dots, x_m$ for large enough $N$, and therefore holomorphic on~$\C^m$. The same argument as in the proof of Theorem~\ref{thm:asymptotics1} again shows $\Phi^{\nu}(x_1, \dots, x_m; q, t)$ admits an analytic continuation to~$\C^m$ and the limit holds uniformly for compact subsets of~$\C^m$.
\end{proof}

\section[Preliminaries on the $(q, t)$-Gelfand--Tsetlin graph]{Preliminaries on the $\boldsymbol{(q, t)}$-Gelfand--Tsetlin graph}\label{sec:qtpreliminaries}

In this section, assume $q, t\in (0, 1)$. We use the notation $\mathbb{P}-\lim\limits_{k\rightarrow\infty}M_k = M$ to indicate that a~sequence of probability measures $\{M_k\}_{k \geq 1}$ converges weakly to~$M$.

\subsection[The $(q, t)$-Gelfand--Tsetlin graph]{The $\boldsymbol{(q, t)}$-Gelfand--Tsetlin graph}\label{sec:qtGT}

The $(q, t)$-Gelfand--Tsetlin graph is an undirected, $\Z_{\geq 0}$-graded graph with countable vertices, together with a sequence of cotransition probabilities between the levels of the graph (considered as discrete spaces).

Begin by def\/ining the set of vertices of the graph as the set of all signatures
\begin{gather*}
\GT \myeq \bigsqcup_{N\geq 0}{\GT_N},
\end{gather*}
where, for convenience, we have also included the singleton $\GT_0 \myeq \{\varnothing\}$.

The edges are determined by the interlacing constraints: edges only join vertices associated to signatures whose lengths dif\/fer by $1$ and $\mu\in\GT_N$ is joined to $\lambda\in\GT_{N+1}$ if and only if $\mu \prec \lambda$, i.e., $\lambda_{N+1} \leq \mu_N \leq \cdots \leq \lambda_2 \leq \mu_1 \leq \lambda_1$. We also assume $\varnothing\prec (k)$, $\forall\, (k)\in\GT_1$. We call the graph with vertices and edges just described the \textit{Gelfand--Tsetlin graph}\footnote{GT graph, for short.}. Next we introduce a~$(q, t)$-deformation of the GT graph by considering certain cotransition probabilities.

\begin{df}
Consider the numbers $\Lambda^{N+1}_N(\lambda, \mu)$, $\mu\prec\lambda$, def\/ined by the expression
\begin{gather}\label{eq:linkdef}
\Lambda^{N+1}_N(\lambda, \mu) = t^{|\mu|}\psi_{\lambda/\mu}(q, t)\frac{P_{\mu}\big(1, t, \dots, t^{N-1}\big)}{P_{\lambda}\big(1, t, \dots, t^N\big)} = \psi_{\lambda/\mu}(q, t)\frac{P_{\mu}\big(t, t^2, \dots, t^N\big)}{P_{\lambda}(1, t, \dots, t^N\big)},
\end{gather}
for all $N\in\Z_{\geq 0}$, $\lambda\in\GT_{N+1}$, $\mu\in\GT_N$, $\mu\prec\lambda$, and where the branching coef\/f\/icients $\psi_{\lambda/\mu}(q, t)$ are def\/ined in Theorem~\ref{branchingmacdonald}. For convenience, also let $\Lambda^{N+1}_N(\lambda, \mu) = 0$ if $\mu\not\prec\lambda$, i.e., if $\mu$ is not adjacent to~$\lambda$.

The \textit{$(q, t)$-Gelfand--Tsetlin graph}\footnote{$(q, t)$-GT graph, for short.} is the sequence $\big\{\GT_N, \Lambda^{N+1}_N\colon N = 0, 1, 2, \dots\big\}$ of data consisting of the GT graph and the $\GT_{N+1}\times\GT_N$ matrices $\big[\Lambda^{N+1}_N(\lambda, \mu)\big]$, $N \geq 0$, def\/ined above.
\end{df}

In general, the numbers $\Lambda^{N+1}_N(\lambda, \mu)$ depend on the values $q$,~$t$, but for simplicity we omit that dependence from the notation. By virtue of the evaluation identity, Theorem~\ref{evaluation}, and the assumption $q, t\in (0, 1)$, we have
\begin{gather}\label{eq:nonnegative}
\Lambda^{N+1}_N(\lambda, \mu) \geq 0, \qquad \forall\, \lambda\in\GT_{N+1},\quad \mu\in\GT_N.
\end{gather}
Moreover, the branching rule for Macdonald polynomials, Theorem~\ref{branchingmacdonald}, shows
\begin{gather*}
P_{\lambda}\big(1, t, t^2, \dots, t^N\big) = \sum_{\mu\colon \mu\prec\lambda}{\psi_{\lambda/\mu}(q, t)P_{\mu}\big(t, t^2, \dots, t^N\big)}
\end{gather*}
and then
\begin{gather}\label{eq:sumone}
1 = \sum_{\mu\colon \mu\prec\lambda}{\Lambda^{N+1}_N(\lambda, \mu)}.
\end{gather}
Equations~(\ref{eq:nonnegative}) and~(\ref{eq:sumone}) show that $\big[\Lambda^{N+1}_N(\lambda, \mu)\big]$ is a stochastic matrix of format $\GT_{N+1}\times\GT_N$, for each $N\geq 0$. Thus $\Lambda^{N+1}_N$ determines a Markov kernel $\GT_{N+1} \dashrightarrow \GT_N$. For this reason, we call $\Lambda^{N+1}_N(\lambda, \mu)$ the \textit{cotransition probabilities}. Let us also def\/ine the more general Markov kernels $\Lambda^M_N\colon \GT_M \dashrightarrow \GT_N$, $M > N$, by
\begin{gather*}
\Lambda^M_N \myeq \Lambda^M_{M-1}\Lambda^{M-1}_{M-2}\cdots\Lambda^{N+1}_N,
\end{gather*}
or more explicitly
\begin{gather*}
\Lambda^M_N(\lambda, \mu) \myeq \sum_{\lambda \succ \lambda^{(M-1)} \succ \cdots \succ \lambda^{(N+1)} \succ \mu}{\Lambda^M_{M-1}\big(\lambda, \lambda^{(M-1)}\big)\cdots \Lambda^{N+1}_N\big(\lambda^{(N+1)}, \mu\big)}.
\end{gather*}

By duality, the kernel $\Lambda^M_N$ also determines a map $\mathcal{M}_{\rm prob}(\GT_M) \rightarrow \mathcal{M}_{\rm prob}(\GT_N)$ between the spaces of probability measures on $\GT_M$ and $\GT_N$, that we denote by the same symbol~$\Lambda^M_N$. For example, if $\lambda\in\GT_M$ and $\delta_{\lambda}$ is the delta mass at~$\lambda$, then $\Lambda^M_N\delta_{\lambda}$ is the probability measure on~$\GT_N$ given by
\begin{gather}\label{kernelsmeasures}
\Lambda^M_N\delta_{\lambda} ( \mu ) = \Lambda^M_N(\lambda, \mu).
\end{gather}

\begin{df}
A sequence $\{M_N\}_{N\geq 0}$, such that each $M_N$ is a probability measure on $\GT_N$, is called a \textit{$(q, t)$-coherent sequence} if the following relations are satisf\/ied
\begin{gather}\label{df:coherence}
M_N(\mu) = \sum_{\lambda\in\GT_{N+1}}{M_{N+1}(\lambda) \Lambda^{N+1}_N(\lambda, \mu)}, \qquad \forall\, N\geq 0, \quad \forall\, \mu\in\GT_N.
\end{gather}
Similarly, a f\/inite sequence $\{M_N\}_{N = 0, 1, \dots, k}$ is said to be a \textit{$(q, t)$-coherent sequence} if the relations~(\ref{df:coherence}) hold for $N = 0, 1, \dots, k-1$ and all $\mu\in\GT_N$.
\end{df}

It is clear that for any probability measure $M_N$ on $\GT_N$, there exist probability measures $M_0, M_1, \dots, M_{N-1}$ on $\GT_0, \GT_1, \dots, \GT_{N-1}$ such that $\{M_m\}_{m = 0, 1, \dots, N}$ is a $(q, t)$-coherent sequence, and moreover $M_0, \dots, M_{N-1}$ are uniquely determined by this condition: in fact, $M_m = \Lambda^N_m M_N$ $\forall \, 0\leq m\leq N-1$.

The set of (inf\/inite) $(q, t)$-coherent sequences $\{M_N\}_{N\geq 0}$ is a convex set. Theorem~\ref{thm:mainapplication} is, in dif\/ferent terms, a characterization of the extreme points of the set of $(q, t)$-coherent sequences.

\subsection[The path-space $\Tau$ and $(q, t)$-central measures]{The path-space $\boldsymbol{\Tau}$ and $\boldsymbol{(q, t)}$-central measures}\label{sec:central}

The set of $(q, t)$-coherent sequences def\/ined before is in bijection with a class of probability measures in the path-space of the GT graph~$\Tau$ that we def\/ine next.

\begin{df}The \textit{path-space $\Tau$} is the set of (inf\/inite) paths in the GT graph that begin at $\varnothing\in\GT_0$:
\begin{gather*}
\Tau \myeq \big\{\tau = \big(\varnothing = \tau^{(0)} \prec \tau^{(1)} \prec \tau^{(2)} \prec \cdots\big)\colon \tau^{(n)}\in\GT_n \ \forall n \geq 0\big\}.
\end{gather*}
For any f\/inite path $\phi = (\phi^{(0)} \prec \phi^{(1)} \prec \dots \prec \phi^{(n)})$, def\/ine the \textit{cylinder set $S_{\phi}$} (or simply $S(\phi)$) by
\begin{gather*}
S_{\phi} \myeq \big\{\tau\in\Tau\colon \tau^{(i)} = \phi^{(i)} \ \forall \, i = 0, 1, \dots, n\big\} \subset \Tau.
\end{gather*}
The set $\Tau$ is equipped with the $\sigma$-algebra generated by the cylinder sets~$S_{\phi}$, where~$\phi$ varies over all f\/inite paths in the GT graph. We always consider~$\Tau$ as a measurable space.
\end{df}

An interesting class of probability measures on $\Tau$ consists of the ones that are coherent with the sequence of stochastic matrices $\big\{\Lambda^{N+1}_N\big\}_{N\geq 0}$. To clarify what such coherence is, def\/ine the natural projection maps
\begin{gather*}
\operatorname{Proj}_N\colon \ \Tau \subset \prod_{n\geq 0}{\GT_n} \longrightarrow \GT_N,\\
\hphantom{\operatorname{Proj}_N\colon}{} \ \tau = \big(\tau^{(0)} \prec \tau^{(1)} \prec \tau^{(2)} \prec \cdots\big) \mapsto \tau^{(N)}.
\end{gather*}
It is a standard exercise to show that the $\sigma$-algebra of $\Tau$ is the smallest one for which all the maps $\operatorname{Proj}_N$ are measurable. Consequently, for any probability measure~$M$ on $\Tau$, we can associate to it the sequence of its pushforwards under the maps $\operatorname{Proj}_N$, namely the sequence $\{M_N\}_{N\geq 0}$, $M_N = (\operatorname{Proj}_N )_*M$.

\begin{df}\label{def:qtcentral}
A probability measure $M$ on $\Tau$ is said to be a \textit{$(q, t)$-central measure} if the following relations hold
\begin{gather*}
M\big( S\big(\phi^{(0)} \prec \phi^{(1)}\dots \prec \phi^{(N-1)} \prec \phi^{(N)}\big) \big) =
\Lambda^{N}_{N-1}\big(\phi^{(N)}, \phi^{(N-1)}\big) \cdots \Lambda^{1}_0\big(\phi^{(1)}, \phi^{(0)}\big) M_N\big(\phi^{(N)}\big) \\
\qquad{} = t^{|\phi^{(N-1)}| + \cdots + |\phi^{(0)}|} \frac{\psi_{\phi^{(N)}/\phi^{(N-1)}}(q, t) \cdots \psi_{\phi^{(1)}/\phi^{(0)}}(q, t) }{P_{\phi^{(N)}}\big(1, t, \dots, t^{N-1}\big)} M_N\big(\phi^{(N)}\big),
\end{gather*}
for all $N \geq 0$, all f\/inite paths $\phi = \big(\phi^{(0)} \prec \dots \prec \phi^{(N)}\big)$, and for some probability measures~$M_N$ on~$\GT_N$.
The branching coef\/f\/icients $\psi_{\mu/\nu}(q, t)$ are explicit in the statement of Theorem~\ref{branchingmacdonald}.

One can verify easily that, if the relations above hold, then the measure $M_N$ is the pushforward $( \operatorname{Proj}_N)_* M$, for all $N \geq 0$. Moreover, $\{M_N\}_{N \geq 0}$ is automatically a $(q, t)$-coherent sequence.

We denote by $M_{\rm prob}(\Tau)$ the set of $(q, t)$-central (probability) measures. The set of $(q, t)$-central measures is a convex set. The set of extreme points of $M_{\rm prob}(\Tau)$, equipped with its inherited topology, is called the \textit{boundary of the $(q, t)$-GT graph}\footnote{We may also call it the \textit{minimal boundary} of the $(q, t)$-GT graph, to dif\/ferentiate it from the \textit{Martin boundary} of the $(q, t)$-GT graph, see Section~\ref{sec:martindef} below.} and is denoted by $\Omega_{q, t}$.
\end{df}

The following proposition implies that the correspondence between the set of $(q, t)$-central measures and the set of $(q, t)$-coherent sequences is a~bijection.

\begin{prop}\label{prop:bijection}
Any probability measure $M$ on $\Tau$ has an associated sequence $\{M_N\}_{N\geq 0}$ of probability measures on $\{\GT_N\}_{N\geq 0}$, as shown above. The map $M\mapsto\{M_N\}_{N\geq 0}$ is a bijection between the set $M_{\rm prob}(\Tau)$ of $(q, t)$-central measures on $\Tau$, and the set of $(q, t)$-coherent sequences. The bijection is an isomorphism of convex sets.

Let $\big\{M^{(i)}\big\}_{i\geq 1}, M$ be elements on $M_{\rm prob}(\Tau)$ and $\big\{M^{(i)}_m\big\}_{\substack{i \geq 1, m\geq 0}}$, $\{M_m\}_{m\geq 0}$ be their correspon\-ding $(q, t)$-coherent sequences. Then the weak limit $\mathbb{P}-\lim\limits_{i\rightarrow\infty}{M^{(i)}} = M$ holds if and only if the weak limits $\mathbb{P}-\lim\limits_{i\rightarrow\infty}{M^{(i)}_m} = M_m$ hold for all $m\in\N$.
\end{prop}
\begin{proof}
Similar statements are known for other branching graphs, e.g., the case $t = q$ of our proposition is given in \cite[Propositions~4.4 and~4.9]{G}, and the case $t = q \rightarrow 1$ is in \cite[Proposition~10.3]{Ol1}.
In the case $t = q$, this proposition is given in \cite[Propositions~4.4 and~4.9]{G}.
The proof at our level of generality can be easily adapted from the proofs in~\cite{G}; details are left to the reader.
\end{proof}

\subsection{Macdonald generating functions}\label{sec:generating}

We introduce Macdonald generating functions, which will be very helpful in our study of $(q, t)$-coherent sequences.

\begin{df}
Let $M_N$ be a probability measure on $\GT_N$, then its \textit{Macdonald generating function} is the formal sum
\begin{gather*}
\PPP_{M_N}(x_1, \dots, x_N) \myeq \sum_{\lambda\in\GT_N}{M_N(\lambda)\frac{P_{\lambda}\big(x_1, x_2t, \dots, x_Nt^{N-1}\big)}{P_{\lambda}\big(1, t, \dots, t^{N-1}\big)}}.
\end{gather*}
\end{df}

Note that $\mathcal{P}_{M_m}(x_1, \dots, x_m)$ depends on the values $q$,~$t$, but we omit such dependence for simplicity.

The sum above is absolutely convergent on the torus $(x_1, \dots, x_m)\in\TT^m$. In fact, Theo\-rem~\ref{branchingmacdonald} and the fact that all the branching coef\/f\/icients $\psi_{\mu/\nu}(q, t)$ are nonnegative imply $\big| P_{\lambda}\big(x_1, x_2t, \dots, x_Nt^{N-1}\big) \big| \leq P_{\lambda}\big(|x_1|, |x_2t|, \dots, \big|x_N t^{N-1}\big|\big) = P_{\lambda}\big(1, t, \dots, t^{N-1}\big)$. Thus not only is $\PPP_{M_N}(x_1, \dots, x_N)$ well-def\/ined as a function on $\TT^N$, but also $\sup\limits_{x_1, \dots, x_N\in\TT}{|\PPP_{M_N}(x_1, \dots, x_N)|} \leq 1$. Therefore $\PPP_{M_N}\in L^{\infty}(\TT^m)$.

If $M_N$ is supported on the set of positive signatures $\GTp_N$, then each $P_{\lambda}\big(x_1, x_2t, \dots, x_Nt^{N-1}\big)$ is a polynomial in $x_1, \dots, x_N$ and therefore the sum def\/ining~$\PPP_{M_N}$ is absolutely convergent on the closed unit disk $(x_1, \dots, x_m)\in\DDD^m$. Moreover $\PPP_{M_N}\in L^{\infty}(\DDD^m)$ if $M_N$ is supported on~$\GTp_N$.

In general, $\PPP_{M_N}\in L^{\infty}(\TT^m) \subset L^2(\TT^m)$. The Fourier series expansion of $\PPP_{M_N}$ can be obtained by using Corollary~\ref{cor:branchingrule}. In fact,
\begin{gather*}
\PPP_{M_N}(x_1, \dots, x_N) = \sum_{\lambda\in\GT_N}{M_N(\lambda)\frac{P_{\lambda}\big(x_1, x_2t, \dots, x_Nt^{N-1}\big)}{P_{\lambda}\big(1, t, \dots, t^{N-1}\big)}}\\
\hphantom{\PPP_{M_N}(x_1, \dots, x_N)}{} = \sum_{\lambda\in\GT_N}{M_N(\lambda) \sum_{\mu\in\GT_N}{\frac{c_{\lambda, \mu} m_{\mu}\big(x_1, \dots, x_N t^{N-1}\big)}{P_{\lambda}\big(1, t, \dots, t^{N-1}\big)}} }\\
\hphantom{\PPP_{M_N}(x_1, \dots, x_N)}{}= \sum_{\mu\in\GT_N}{ m_{\mu}\big(x_1, \dots, x_N t^{N-1}\big) \sum_{\lambda\in\GT_N}{\frac{c_{\lambda, \mu}M_N(\lambda)}{P_{\lambda}\big(1, t, \dots, t^{N-1}\big)}} },
\end{gather*}
where the interchange in the order of summation follows from the absolute convergence of all the sums involved. (For the absolute convergence, the nonnegativity of all coef\/f\/icients $c_{\lambda, \mu}$ is needed.)
From the expansion above, we can extract the coef\/f\/icient of $x_1^{\kappa_1}\cdots x_N^{\kappa_N}$ in the Fourier series, for any $\kappa = (\kappa_1 \geq \dots \geq \kappa_N)\in\GT_N$. In fact, such term appears only in $m_{\kappa}\big(x_1, \dots, x_Nt^{N-1}\big)$ with coef\/f\/icient $t^{n(\kappa)}$, where $n(\kappa) = \kappa_2 + 2\kappa_3 + \cdots + (N-1)\kappa_N$. Thus the Fourier coef\/f\/icient of $x_1^{\kappa_1}\cdots x_N^{\kappa_N}$ in the expansion of $\PPP_{M_N}(x_1, \dots, x_N)$ is
\begin{gather}\label{rem:fouriercoeffs}
f_{\kappa_1, \dots, \kappa_N} = t^{n(\kappa)} \sum_{\lambda\in\GT_N}{\frac{c_{\lambda, \kappa}M_N(\lambda)}{P_{\lambda}\big(1, t, \dots, t^{N-1}\big)}}.
\end{gather}
If $M_N$ is supported on $\GTp_N$, then the sum def\/ining $f_{\kappa_1, \dots, \kappa_N}$ above is f\/inite. Indeed the only signatures with a nonzero contribution are $\lambda\in\GTp_N$ with $|\lambda| = |\kappa|$. But then $|\kappa| = |\lambda| \geq \lambda_1$, and there are f\/initely many signatures $\lambda\in\GT_N$ with $|\kappa| \geq \lambda_1 \geq \dots \geq \lambda_N \geq 0$. This observation will be put to use several times.

\begin{lem}\label{cor:generating}
Let $M_N$, $M_N'$ be probability measures on $\GT_N$ that are supported on $\GTp_N$. If
\begin{gather*}
\PPP_{M_N}(x_1, \dots, x_N) = \PPP_{M_N'}(x_1, \dots, x_N) \qquad \forall\, (x_1, \dots, x_N)\in\TT^N
\end{gather*}
then $M_N = M_N'$.
\end{lem}
\begin{proof}
Both $\PPP_{M_N}(x_1, \dots, x_N)$ and $\PPP_{M_N'}(x_1, \dots, x_N)$ belong to $L^2\big(\TT^N\big)$.
The equality of these functions implies that their Fourier coef\/f\/icients agree.
From~\eqref{rem:fouriercoeffs}, this means
\begin{gather}\label{eqn:eqfouriers1}
\sum_{\mu\in\GTp_N}{\frac{c_{\mu, \kappa}M_N(\mu)}{P_{\mu}\big(1, t, \dots, t^{N-1}\big)}}
= \sum_{\mu\in\GTp_N}{\frac{c_{\mu, \kappa}M_N'(\mu)}{P_{\mu}\big(1, t, \dots, t^{N-1}\big)}} \qquad \forall\, \kappa\in\GT_N.
\end{gather}

Observe that we have restricted the sum above to $\mu\in\GTp_N$, because $M_N$, $M_N'$ are supported on positive signatures.
Let $n\in\Z_{\geq 0}$ be arbitrary. We show that $M_N(\kappa) = M_N'(\kappa)$ for all $\kappa\in\GTp_N$ with $|\kappa| = n$.

Let $C$ be the matrix whose rows and columns are parametrized by $\lambda\in\GTp_N$ with $|\lambda| = n$, and such that its entry $C(\kappa, \mu)$ is $c_{\mu, \kappa}/P_{\mu}\big(1, t, \dots, t^{N-1}\big)$. Observe that $C$ is a f\/inite and square matrix.
Also let $M$ (resp.~$M'$) be the column vector whose entries are parametrized by $\lambda\in\GTp_N$ with $|\lambda| = n$ and whose entry $\mu$ is $M_N(\mu)$ (resp.~$M_N'(\mu)$). Then~(\ref{eqn:eqfouriers1}) yields $CM = CM'$. The matrix $C$ is upper-triangular with respect to the order $\geq$ on signatures because $c_{\mu, \kappa} = 0$ unless $\mu \geq \kappa$. Moreover the diagonal entries are $c_{\mu, \mu}/P_{\mu}\big(1, t, \dots, t^{N-1}\big) = 1/P_{\mu}\big(1, t, \dots, t^{N-1}\big) \neq 0$. It follows that $C$ has an inverse and $M = C^{-1}CM = C^{-1}CM' = M'$. Therefore $M_N(\kappa) = M_N'(\kappa)$ $\forall\, \kappa\in\GTp_N$ with $|\kappa| = n$. Since $n\in\Z_{\geq 0}$ was arbitrary and both $M_N$, $M_N'$ are supported on~$\GTp_N$, we conclude $M_N = M_N'$.
\end{proof}

\begin{prop}\label{prop:coherentsequences}
If the sequence $\{M_N\}_{N\geq 0}$ $($resp.\ finite sequence $\{M_N\}_{N=0, 1, \dots, k})$ is a $(q, t)$-coherent sequence, then
\begin{gather*}
\PPP_{M_N}(x_1, \dots, x_N) = \PPP_{M_{N+1}}(1, x_1, \dots, x_N), \qquad \forall\, (x_1, \dots, x_N)\in\TT^N
\end{gather*}
for all $N\geq 0$ $($resp.\ for all $N = 0, 1, \dots, k-1)$. The converse statement holds if, for each $N\geq 0$ $($resp.\ $N = 0, 1, \dots, k-1)$, $M_N$ is supported on~$\GTp_N$.
\end{prop}

\begin{proof}
Let us prove the f\/irst part. Let $\{M_N\}_{N\geq 0}$ be a $(q, t)$-coherent sequence. By making use of the branching rule, Theorem~\ref{branchingmacdonald}, the fact that $P_{\mu}$ is homogeneous of degree $|\mu|$, and making a~change in the order of summation, we obtain
\begin{gather*}
\PPP_{M_{N+1}}(1, x_1, \dots, x_N) = \sum_{\lambda\in\GT_{N+1}}{M_{N+1}(\lambda)\frac{P_{\lambda}\big(1, x_1t, \dots, x_Nt^{N}\big)}{P_{\lambda}\big(1, t, \dots, t^N\big)}}\\
\qquad{} = \sum_{\lambda\in\GT_{N+1}}{\frac{M_{N+1}(\lambda)}{P_{\lambda}\big(1, t, \dots, t^N\big)} \sum_{\mu\in\GT_N}{\psi_{\lambda/\mu}(q, t) P_{\mu}\big(x_1t, \dots, x_Nt^{N}\big) } }\\
\qquad{} = \sum_{\lambda\in\GT_{N+1}}{\frac{M_{N+1}(\lambda)}{P_{\lambda}\big(1, t, \dots, t^N\big)} \sum_{\mu\in\GT_N}{t^{|\mu|} \psi_{\lambda/\mu}(q, t) P_{\mu}\big(x_1, \dots, x_Nt^{N-1}\big) } }\\
\qquad{} = \sum_{\mu\in\GT_N}{ \frac{P_{\mu}(x_1, \dots, x_Nt^{N-1})}{P_{\mu}\big(1, t, \dots, t^{N-1}\big)} \sum_{\lambda\in\GT_{N+1}}{t^{|\mu|} \psi_{\lambda/\mu}(q, t) \frac{P_{\mu}\big(1, t, \dots, t^{N-1}\big)}{P_{\lambda}\big(1, t, \dots, t^N\big)}M_{N+1}(\lambda) } }\\
\qquad{} = \sum_{\mu\in\GT_N}{ \frac{P_{\mu}\big(x_1, \dots, x_Nt^{N-1}\big)}{P_{\mu}\big(1, t, \dots, t^{N-1}\big)} M_N(\mu) } = \mathcal{P}_{M_N}(x_1, \dots, x_N).
\end{gather*}
We can easily show that all sums above are absolutely convergent, so the change in the order of summations can be justif\/ied.

Next we prove the converse statement. Assume that $M_{N}$, $M_{N+1}$ are probability measures on $\GT_{N}$, $\GT_{N+1}$. Assume that they are supported on $\GTp_N$, $\GTp_{N+1}$, respectively, and also that $\PPP_{M_N}(x_1, \dots, x_N) = \PPP_{M_{N+1}}(1, x_1, \dots, x_N)$ on~$\TT^N$. Let~$M_N'$ be the measure on $\GT_N$ def\/ined by
\begin{gather*}
M_N'(\mu) = \sum_{\lambda\in\GT_{N+1}}{M_{N+1}(\lambda)\Lambda^{N+1}_N(\lambda, \mu)} \qquad \forall\, \mu\in\GT_N.
\end{gather*}
Since $\big[\Lambda^{N+1}_N(\lambda, \mu)\big]$ is a stochastic matrix and $M_{N+1}$ is a probability measure on $\GT_{N+1}$, then~$M_N'$ is a probability measure on~$\GT_N$. Moreover since $M_{N+1}$ is supported on $\GTp_{N+1}$, it follows that $M_N'$ is supported on $\GTp_N$. In fact, if $\mu\notin\GT_N \setminus \GTp_N$ (or equivalently $\mu_N < 0$), then for any $\lambda\in\GT_{N+1}$, either $\lambda_{N+1} < 0$ in which case $M_{N+1}(\lambda) = 0$, or $\lambda_{N+1} \geq 0$ in which case $\Lambda^{N+1}_N(\lambda, \mu) = 0$.

We will be done if we showed $M_N = M_N'$. From what we have proved in the f\/irst part of the argument, we have $\PPP_{M_N'}(x_1, \dots, x_N) = \PPP_{M_{N+1}}(1, x_1, \dots, x_N)$ on~$\TT^N$.
It follows that $\PPP_{M_N} = \PPP_{M_N'}$ on $\TT^N$.
An application of Lemma~\ref{cor:generating} concludes the proof.
\end{proof}

\begin{prop}\label{prop:convergencemacdonalds}
Let $N\in\N$ be arbitrary. If $\{M^m\}_{m\geq 1}$, $M$ are all probability measures on~$\GT_N$ such that the weak convergence holds $\mathbb{P}-\lim\limits_{m\rightarrow\infty}{M^m} = M$, then{\samepage
\begin{gather}\label{eqn:convergencemacdonalds}
\lim_{m\rightarrow\infty}{\PPP_{M^m}(x_1, \dots, x_N)} = \PPP_M(x_1, \dots, x_N) \qquad \forall\, (x_1, \dots, x_N)\in\TT^N.
\end{gather}
The convergence above is uniform on~$\TT^N$.}
\end{prop}

\begin{proof}
Let $\epsilon > 0$ be a very small real number. Since $M$ is a probability measure, there exists $c>0$ large enough so that
\begin{gather}\label{eqn:eqnlemma1}
M\left( \{ \lambda \in \GT_N\colon c \geq \lambda_1 \geq \lambda_2 \geq \dots \geq \lambda_N \geq -c \} \right) > 1 - \epsilon.
\end{gather}
From the weak convergence $\mathbb{P}-\lim\limits_{m\rightarrow\infty}{M^m} = M$, there exists $N_1 \in \N$ so that $m > N_1$ implies
\begin{gather}\label{eqn:eqnlemma2}
M^m\left( \{ \lambda \in \GT_N\colon c \geq \lambda_1 \geq \lambda_2 \geq \dots \geq \lambda_N \geq -c \} \right) > 1 - 2\epsilon.
\end{gather}
Consider the set $\GT_N^{[-c, c]} \myeq \{\lambda\in\GT_N\colon c \geq \lambda_1 \geq \dots \geq \lambda_N \geq -c\}\subset\GT_N$. It is clear that $\GT_N^{[-c, c]}$ is f\/inite and has cardinality no greater than $(2c + 1)^N$. Also, since $\GT_N^{[-c,, c]}$ is a f\/inite set, there exists $N_2\in\N$ such that $m > N_2$ implies
\begin{gather}\label{eqn:eqnlemma3}
\left|M(\lambda) - M^m(\lambda)\right| < \frac{\epsilon}{(2c+1)^N} \qquad \forall\, \lambda\in\GT_N^{[-c, c]}.
\end{gather}

We are ready to make the desired estimate. Use (\ref{eqn:eqnlemma1}), (\ref{eqn:eqnlemma2}), (\ref{eqn:eqnlemma3}) and the triangle inequality to argue that for any $m > \max\{N_1, N_2\}$, we have
\begin{gather}
\sup_{(x_1, \dots, x_N)\in\TT^N}{ \left| \PPP_{M^m}(x_1, \dots, x_N) - \PPP_M(x_1, \dots, x_N) \right| }\nonumber\\
\qquad{} \leq \sum_{\lambda\in\GT_N^{[-c, c]}}{|M^m(\lambda) - M(\lambda)|\sup_{(x_1, \dots, x_N)\in\TT^N}{ \left|\frac{P_{\lambda}\big(x_1, x_2t, \dots, x_Nt^{N-1}\big)}{P_{\lambda}\big(1, t, \dots, t^{N-1}\big)}\right| }}\nonumber\\
\qquad\quad{} + \sum_{\lambda\in\GT_N\setminus\GT_N^{[-c, c]}}{ \left( M^m(\lambda) + M(\lambda) \right)\sup_{(x_1, \dots, x_N)\in\TT^N}{ \left|\frac{P_{\lambda}\big(x_1, x_2t, \dots, x_Nt^{N-1}\big)}{P_{\lambda}\big(1, t, \dots, t^{N-1}\big)}\right| }}\nonumber\\
\qquad{} \leq \sum_{\lambda\in\GT_N^{[-c, c]}}{|M^m(\lambda) - M(\lambda)|} + \sum_{\lambda\in\GT_N\setminus\GT_N^{[-c, c]}}{\left(M^m(\lambda) + M(\lambda)\right)}\nonumber\\
\qquad{} \leq \frac{\epsilon}{(2c+1)^N}\big| \GT_N^{[-c, c]} \big| + M^m\big(\GT_N \setminus \GT_N^{[-c, c]}\big) + M\big(\GT_N \setminus \GT_N^{[-c, c]}\big)\label{eqns:estimates1}\\
\qquad{} \leq \epsilon + 2\epsilon + \epsilon = 4\epsilon.\tag*{\qed}
\end{gather}\renewcommand{\qed}{}
\end{proof}

A partial converse to the previous proposition is Proposition~\ref{prop:lemmaconv} below.

\begin{prop}\label{prop:lemmaconv}
Let $\{M^m\}_{m \geq 1}$, $M$, be probability measures on $\GT_N$, supported on $\GTp_N$. Assume the following convergence
\begin{gather*}
\lim_{m\rightarrow\infty}{\PPP_{M^m}(x_1, \dots, x_N)} = \PPP_{M}(x_1, \dots, x_N)
\end{gather*}
holds uniformly on $\TT^N$, then we have the weak convergence $\mathbb{P}-\lim\limits_{m\rightarrow\infty}{M^m} = M$.
\end{prop}
\begin{proof}
All functions $\{\PPP_{M^m}\}_{m\geq 1}$, $\PPP_M$ belong to $L^2(\TT^m)$. Therefore the limit $\lim\limits_{m\rightarrow\infty}\PPP_{M^m}(x_1, \dots,$ $ x_N) = \PPP_{M}(x_1, \dots, x_N)$ implies the convergence of the Fourier coef\/f\/icients. Due to~\eqref{rem:fouriercoeffs}, this implies that, for any $\kappa\in\GT_N$, we have
\begin{gather*}
\lim_{m\rightarrow\infty}{ \sum_{\lambda\in\GTp_N}{ \frac{c_{\lambda, \kappa}M^m(\lambda)}{P_{\lambda}\big(1, t, \dots, t^{N-1}\big)} } } = \sum_{\lambda\in\GTp_N}{ \frac{c_{\lambda, \kappa}M(\lambda)}{P_{\lambda}\big(1, t, \dots, t^{N-1}\big)} }.
\end{gather*}
We show that $\lim\limits_{m\rightarrow\infty}{M^m(\lambda)} = M(\lambda)$ for any $\lambda\in\GTp_N$. In fact, let $n\in\Z_{\geq 0}$ be arbitrary and let us prove $\lim\limits_{m\rightarrow\infty}{M^m(\lambda)} = M(\lambda)$ for any $\lambda\in\GTp_N$ with $|\lambda| = n$. Consider the f\/inite, square matrix~$C$ whose rows and columns are parametrized by $\lambda\in\GTp_N$, $|\lambda| = n$, and whose entries are $C(\kappa, \lambda) = c_{\lambda, \kappa}/P_{\lambda}\big(1, t, \dots, t^{N-1}\big)$.
Also let $\{M^m\}_{m\geq 0}, M$ be column vectors whose entries are parametrized by $\lambda\in\GTp_N$ with $|\lambda| = n$, and whose entries, at $\lambda\in\GTp_N$, are $\{M^m(\lambda)\}_{m\geq 0}, M(\lambda)$. From the limit relation above, we have the entrywise limit of column vectors $\lim\limits_{m\rightarrow\infty}{CM^{m}} = CM$. The matrix $C$ is upper triangular (with respect to the order on signatures given in Corollary~\ref{cor:branchingrule}) and has nonzero diagonal entries, thus it has an inverse~$C^{-1}$.
Each entry of the column vectors $M^m = \big(C^{-1}C\big)M^m = C^{-1}(CM^{m})$ is a f\/inite linear combination of entries of $CM^m$, and the same can be said about the entries of $C^{-1}CM = M$. Thus the entry-wise limit of column vectors $\lim\limits_{m\rightarrow\infty}{M^m} = M$ follows.

By assumption, $M^m(\lambda) = M(\lambda) = 0$ for any $\lambda\notin\GTp_N$, so also $\lim\limits_{m\rightarrow\infty}{M^m(\lambda)} = M(\lambda)$ in this case. Hence the weak convergence $\mathbb{P}-\lim\limits_{m\rightarrow\infty}{M^m} = M$ is proved.
\end{proof}

\subsection[Automorphisms $A_k$]{Automorphisms $\boldsymbol{A_k}$}\label{sec:automorphismsAk}

Recall the set $\Nu$ of nonincreasing integer sequences, given in Def\/inition~\ref{def:stabilizing}. Equip~$\Nu$ with the topology of pointwise convergence. We denote a generic element of $\Nu$ by $\nu = (\nu_1 \leq \nu_2 \leq \cdots)$. For each $k\in\Z$, we can def\/ine the continuous map $A_k\colon \Nu \rightarrow \Nu$ by $\nu \mapsto A_k\nu = (\nu_1 + k \leq \nu_2 + k \leq \cdots)$. It is clear that $A_k$ has inverse $A_{-k}$, so each $A_k$ is a homeomorphism.

Similar automorphisms can be constructed for $\GT$ and $\Tau$. In detail, we can def\/ine the map $A_k\colon \GT \rightarrow \GT$ by $\lambda\mapsto A_k\lambda = (\lambda_1 + k \geq \lambda_2 + k\geq \cdots)$, $A_k\varnothing = \varnothing$, whose inverse is $A_{-k}$, and moreover it restricts to automorphisms $\GT_N \rightarrow \GT_N$ for each $N\in\Z_{\geq 0}$. It is clear that $\mu\prec\lambda$ implies $A_k\mu\prec A_k\lambda$, so the automorphism $A_k$ of $\GT$ induces the automorphism of measurable spaces $A_k\colon \Tau \rightarrow \Tau$, $\tau = \big(\tau^{(0)} \prec \tau^{(1)}\prec \tau^{(2)} \prec \cdots\big)\mapsto \big(A_k\tau^{(0)}\prec A_k\tau^{(1)}\prec A_k\tau^{(2)}\prec\cdots\big)$.

The same notation $A_k$ is used to def\/ine automorphisms of the spaces $\Nu$, $\GT$ and $\Tau$, but there should be no confusion each time it is used in the future.

We have introduced the automorphisms $A_k$ because, in Lemma~\ref{lem:qtcoherency} below, we will relate the extreme central probability measures associated to $\nu$ and $A_k\nu$.
The starting point is the following simple statement, which has nothing to do with probability.

\begin{lem}\label{lem:Akfunctions}
Recall the functions $\Phi^{\nu}(x_1, \dots, x_m; q, t)$, defined in the statement of Theorem~{\rm \ref{thm:asymptotics2}}. Let $\nu\in\Nu$ and $k\in\Z$ be arbitrary. The following equality holds
\begin{gather}
\Phi^{A_k\nu}(x_1, \dots, x_m; q, t) = t^{k{m \choose 2}}(x_1\cdots x_m)^k \Phi^{\nu}(x_1, \dots, x_m; q, t), \nonumber\\ \forall\, (x_1, \dots, x_m)\in (\C\setminus\{0\})^m.\label{eqn:equalityshift}
\end{gather}
\end{lem}
\begin{proof}
As both sides of the identity~(\ref{eqn:equalityshift}) are analytic functions on $ (\C\setminus\{0\} )^m$, we only need to prove the equality for $(x_1, \dots, x_m)\in\U_m$, where the domain $\U_m$ was def\/ined in the statement of Theorem~\ref{thm:asymptotics2}. We can now make use of formula~(\ref{def:Phinu}) for $\Phi^{\nu}(x_1, \dots, x_m; q, t)$. Observe that the only place where~$\nu$ appears in the right-hand side is inside the univariate functions $\Phi^{\nu}(x_i; q, t)$. The operator $\widetilde{D}_{q, \theta}^{(m)}$ satisf\/ies that for any Laurent polynomial $f$ on variables $x_1, \dots, x_m$, the following identity holds
\[
\widetilde{D}_{q, \theta}^{(m)} \big\{ (x_1\cdots x_m)^k f \big\} = t^{k {m \choose 2}} (x_1 \cdots x_m)^k \widetilde{D}_{q, \theta}^{(m)} \{f\}.
\]
We deduce that the lemma will be proved for all $m\in\N$ once we prove it for $m = 1$, that is, we need $\Phi^{A_k\nu}(x; q, t) = x^k \Phi^{\nu}(x; q, t)$ $\forall \, x\in \U = \bigcap\limits_{k\geq 1}{\{x \neq q^{-k}\}}\cap\{x \neq 0\}$. The latter statement easily follows from the integral def\/inition of $\Phi^{\nu}(x; q, t)$ in Theorem~\ref{thm:asymptotics1} (to be precise, our desired statement follows after a change of variables $z\mapsto z+k$ in the integral).
\end{proof}

Next we introduce new maps $\{A_k\}_{k\in\Z}$ on spaces of probability measures.

For a probability measure $M_m$ on $\GT_m$, def\/ine $A_k M_m$ as the pushforward of $M_m$ under the automorphism $A_k$ of $\GT_m$, i.e., $A_k M_m (\mu) \myeq M_m(A_{-k}\mu)$ for all $\mu\in\GT_m$. Observe that if we let~$\delta_{\lambda}$ be the probability measure on $\GT_m$ given by the delta mass at $\lambda\in\GT_m$, then $A_k \delta_{\lambda} = \delta_{A_k\lambda}$ for any $k\in\Z$.

Similarly if $M$ is a probability measure on $\Tau$, def\/ine $A_k M$ as the pushforward of $M$ under the automorphism $A_k$ of~$\Tau$. This can be described concretely as follows. The automorphism $A_k$ of $\Tau$ induces automorphisms $A_k$ on the set of paths of length $m$ in the GT graph, for any $m\in\N$:
\begin{gather*}
\phi = \big(\phi^{(0)} \prec \phi^{(1)} \prec \cdots \prec \phi^{(m)}\big) \mapsto A_k\phi \myeq \big(A_k\phi^{(0)} \prec A_k\phi^{(1)} \prec \cdots \prec A_k\phi^{(m)}\big).
\end{gather*}
It is therefore natural to def\/ine also an automorphism on the family of cylinder sets by $A_k S_{\phi} \myeq S_{A_k \phi}$.
Then $A_k M$ is given by $A_k M(S_{\phi}) \myeq M(A_{-k}S_{\phi}) = M(S_{A_{-k}\phi})$ for all f\/inite paths $\phi = \big(\phi^{(0)} \prec \phi^{(1)} \prec \dots \prec \phi^{(m)}\big)$.

\begin{lem}\label{lem:qtcoherency}
Let $M\in M_{\rm prob}(\Tau)$ be a $(q, t)$-central probability measure and $\{M_m\}_{m\geq 0}$ be its associated $(q, t)$-coherent sequence. Then also $A_k M\in M_{\rm prob}(\Tau)$ and its associated $(q, t)$-coherent sequence is $\{A_k M_m\}_{m\geq 0}$.
\end{lem}
\begin{proof}
Clearly $A_k M\in M_{\rm prob}(\Tau)$ is a consequence of the def\/inition of $A_k M$ as a pushforward of the probability measure $M$.
The second claim can be restated as
\[
(\operatorname{Proj}_m)_* A_kM = A_k M_m.
\]
This equality follows from the def\/initions of $A_kM$ and $A_kM_m$ as pushforwards of $M$ and $M_m$, the fact that $(\operatorname{Proj}_m)_* M = M_m$, and the evident commutativity of the diagram
\begin{equation*}
\begin{tikzcd}
\Tau \arrow{r}{A_k} \arrow[swap]{d}{\operatorname{Proj}_m} & \Tau \arrow{d}{\operatorname{Proj}_m} \\
\GT_m \arrow{r}{A_k} & \GT_m.
\end{tikzcd}\tag*{\qed}
\end{equation*}\renewcommand{\qed}{}
\end{proof}

\section[The boundary of the $(q, t)$-Gelfand--Tsetlin graph]{The boundary of the $\boldsymbol{(q, t)}$-Gelfand--Tsetlin graph}\label{sec:qtboundary}

In this section we prove Theorem~\ref{thm:mainapplication}, which characterizes the (minimal) boundary of the $(q, t)$-GT graph. Along the way, we also def\/ine and characterize the Martin boundary of the $(q, t)$-GT graph.

Assume throughout this section that $q\in (0, 1)$, $\theta\in\N$ and set $t = q^{\theta}$.
Recall the notation $\mathbb{P}-\lim\limits_{k\rightarrow\infty}M_k = M$ indicates that a sequence of probability measures $\{M_k\}_{k \geq 1}$ converges weakly to~$M$.

\subsection{The Martin boundary: def\/inition and preliminaries}\label{sec:martindef}

For any $\lambda\in\GT_N$, let $\delta_{\lambda}$ be the delta mass at $\lambda$.
As remarked in Section~\ref{sec:qtGT}, there exists a unique $(q, t)$-coherent sequence $\{M_m^{\lambda}\}_{m = 0, 1, \dots, N}$ such that each $M_m^{\lambda}$ is a probability measure on $\GT_m$ and $M_N^{\lambda} = \delta_{\lambda}$. Such a sequence is given explicitly by $M^{\lambda}_N = \delta_{\lambda}$ and $M^{\lambda}_m = \Lambda^N_m\delta_{\lambda}$ $\forall\, 0\leq m\leq N-1$, where the probability measures $\Lambda^N_m\delta_{\lambda}$ on $\GT_m$ are given explicitly in~\eqref{kernelsmeasures}. Moreover recall that for a $(q, t)$-central probability measure $M$ on $\Tau$, we can associate a $(q, t)$-coherent sequence $\{M_m\}_{m\geq 0}$ as in Def\/inition~\ref{def:qtcentral}.

\begin{df}
The \textit{Martin boundary} of the $(q, t)$-Gelfand--Tsetlin graph is the subset of $(q, t)$-central probability measures $M \in M_{\rm prob}(\Tau)$ for which there exists a sequence $\{\lambda(N)\}_{N\geq 1}$, $\lambda(N)\in\GT_N$, such that the following weak limits hold
\begin{gather*}
\mathbb{P}-\lim_{N\rightarrow\infty}{\Lambda^N_m \delta_{\lambda(N)}} = M_m \qquad \forall\, m = 0, 1, 2, \dots.
\end{gather*}
Let us denote the Martin boundary by $\Mart \subset M_{\rm prob}(\Tau)$ and equip it with its subspace topology, namely with the topology of weak convergence.
\end{df}

To characterize the Martin boundary $\Mart$, we begin by giving necessary conditions on sequences of signatures $\{\lambda(N)\}_{N\geq 1}$ which yield a weak convergence as above. In this section, we sometimes denote $M^{\lambda}_N \myeq \delta_{\lambda}$ the delta mass at $\lambda\in\GT_N$, and $\{M_m^{\lambda} = \Lambda^N_m\delta_{\lambda}\}_{m = 0, 1, \dots, N}$ the corresponding $(q, t)$-coherent sequence.

\begin{lem}\label{lem:boundedparts0}
Assume that $\{\lambda(N)\}_{N\geq 1}$, $\lambda(N)\in\GT_N$, is a sequence of signatures such that $\{\Lambda^N_1\delta_{\lambda(N)}\}_{N\geq 1}$ converges weakly, as $N\rightarrow\infty$, to some probability measure $\mathfrak{m}$ on $\GT_1 = \Z$. Then the sequence $\{\lambda(N)_N\}_{N\geq 1}$ is bounded.
\end{lem}

\begin{proof}
We use the Macdonald generating functions of Section~\ref{sec:generating} above. By Proposition~\ref{prop:coherentsequences}, we have
\begin{gather*}
\PPP_{\delta_{\lambda(N)}}\big(1^{N-1}, z\big) = \PPP_{M^{\lambda(N)}_1}(z)
\end{gather*}
and then
\begin{gather}\label{eqn:taylors}
\frac{P_{\lambda(N)}\big(1, t, \dots, t^{N-2}, t^{N-1}z\big)}{P_{\lambda(N)}\big(1, t, \dots, t^{N-2}, t^{N-1}\big)}
= \frac{P_{\lambda(N)}\big(z, t^{-1}, t^{-2}, \dots, t^{1-N}\big)}{P_{\lambda(N)}\big(1, t^{-1}, t^{-2}, \dots, t^{1-N}\big)}
= \sum_{n\in\Z}{M^{\lambda(N)}_1(n)z^n}.
\end{gather}
In the sum of the right-hand side above, $n$ ranges from $\lambda(N)_N$ to $\lambda(N)_1$ because of the branching rule for Macdonald polynomials. We multiply the equality by $z^{-\lambda(N)_N}$ and then set $z = 0$, so in the right-hand side one clearly picks up the coef\/f\/icient of~$z^{\lambda(N)_N}$, namely $M^{\lambda(N)}_1 (\lambda(N)_N)$. By the index stability of Macdonald polynomials,~\ref{macdonaldsignatures11}), the left-hand side is
\begin{gather*}
\left.z^{-\lambda(N)_N}\frac{P_{\lambda(N)}\big(z, t^{-1}, \dots, t^{1-N}\big)}{P_{\lambda(N)}\big(1, t^{-1}, \dots, t^{1-N}\big)}\right|_{z=0} \\
\qquad{} = \left.\frac{\big(z\cdot t^{-1}\cdots t^{1-N}\big)^{-\lambda(N)_N}}{\big(1\cdot t^{-1}\cdots t^{1-N}\big)^{-\lambda(N)_N}}\frac{P_{\lambda(N)}\big(z, t^{-1}, \dots, t^{1-N}\big)}{P_{\lambda(N)}\big(1, t^{-1}, \dots, t^{1-N}\big)}\right|_{z=0}\\
\qquad{} = \left.\frac{P_{(\lambda(N)_1 - \lambda(N)_N, \dots, \lambda(N)_{N-1}-\lambda(N)_N, 0)}\big(z, t^{-1}, \dots, t^{1-N}\big)}{P_{(\lambda(N)_1 - \lambda(N)_N, \dots, \lambda(N)_{N-1}-\lambda(N)_N, 0)}\big(1, t^{-1}, \dots, t^{1-N}\big)}\right|_{z=0}\\
\qquad{} = \frac{P_{(\lambda(N)_1 - \lambda(N)_N, \dots, \lambda(N)_{N-1}-\lambda(N)_N)}\big(t^{-1}, \dots, t^{1-N}\big)}{P_{(\lambda(N)_1 - \lambda(N)_N, \dots, \lambda(N)_{N-1}-\lambda(N)_N, 0)}\big(1, t^{-1}, \dots, t^{1-N}\big)}.
\end{gather*}
Thanks to Theorem~\ref{evaluation}, we can then obtain a lower bound for $M^{\lambda(N)}_1 (\lambda(N)_N )$ as follows:
\begin{gather*}
M^{\lambda(N)}_1(\lambda(N)_N) = \frac{P_{(\lambda(N)_1 - \lambda(N)_N, \dots, \lambda(N)_{N-1}-\lambda(N)_N)}\big(t^{-1}, \dots, t^{1-N}\big)}{P_{(\lambda(N)_1 - \lambda(N)_N, \dots, \lambda(N)_{N-1}-\lambda(N)_N, 0)}\big(1, t^{-1}, \dots, t^{1-N}\big)}\\
\qquad{} = \frac{P_{(\lambda(N)_1 - \lambda(N)_N, \dots, \lambda(N)_{N-1}-\lambda(N)_N)}\big(t^{N-2}, \dots, t, 1\big)}{P_{(\lambda(N)_1 - \lambda(N)_N, \dots, \lambda(N)_{N-1}-\lambda(N)_N, 0)}\big(t^{N-1}, \dots, t, 1\big)}\\
\qquad{} = \prod_{1\leq i < j\leq N-1}{\frac{\big(q^{\lambda(N)_i - \lambda(N)_j}t^{j-i}; q\big)_{\infty}(t^{j-i+1}; q)_{\infty}}{\big(q^{\lambda(N)_i - \lambda(N)_j}t^{j-i+1}; q\big)_{\infty}(t^{j-i}; q)_{\infty}}}\\
\qquad\quad{}\times
\prod_{1\leq i < j\leq N}{\frac{\big(q^{\lambda(N)_i - \lambda(N)_j}t^{j-i+1}; q\big)_{\infty}(t^{j-i}; q)_{\infty}}{\big(q^{\lambda(N)_i - \lambda(N)_j}t^{j - i}; q\big)_{\infty}(t^{j-i+1}; q)_{\infty}}}\\
\qquad{} = \prod_{i=1}^{N-1}{\frac{\big(q^{\lambda(N)_i - \lambda(N)_N}t^{N-i+1}; q\big)_{\infty}(t^{N-i}; q)_{\infty}}{\big(q^{\lambda(N)_i - \lambda(N)_N}t^{N-i}; q\big)_{\infty}(t^{N-i+1}; q)_{\infty}}}\\
\qquad{} = \prod_{i=1}^{N-1}{\frac{\big(1 - t^{N-i}\big)\big(1 - qt^{N-i}\big)\cdots \big(1 - q^{\theta - 1}t^{N-i}\big)}{\big(1 - q^{\lambda(N)_i - \lambda(N)_N}t^{N-i}\big)\cdots \big(1 - q^{\lambda(N)_i - \lambda(N)_N + \theta - 1}t^{N-i}\big)}}\\
\qquad{} \geq \prod_{i=1}^{N-1}{\big(1 - t^{N-i}\big)\big(1 - qt^{N-i}\big)\cdots \big(1 - q^{\theta - 1}t^{N-i}\big)} \geq \prod_{i=1}^{N-1}{\big(1 - t^{N-i}\big)^{\theta}} \geq ( (t; t)_{\infty} )^{\theta}.
\end{gather*}

On the other hand, since $(t; t)_{\infty}\in (0, 1)$ and $\mathbb{P}-\lim\limits_{N\rightarrow\infty}{M^{\lambda(N)}_1} = \mathfrak{m}$, then there exist $N_1, N_2 \in \N$ large enough such that
\begin{gather*}
M^{\lambda(N)}_1\left(\{ (n)\in\GT_1\colon -N_1 \leq n \leq N_1\}\right) > 1 - ((t; t)_{\infty})^{\theta} \qquad \forall\, N > N_2.
\end{gather*}
We conclude that $-N_1 \leq \lambda(N)_N \leq N_1$, for all $N > N_2$.
Therefore the sequence $\{\lambda(N)_N\}_{N\geq 1}$ is bounded.
\end{proof}

One actually has the following more general statement.

\begin{lem}\label{lem:boundedparts}
Let $k\in\N$. Assume that $\{\lambda(N)\}_{N\geq 1}$, $\lambda(N)\in\GT_N$, is a sequence of signatures such that the sequence $\Lambda^N_k \delta_{\lambda(N)}$ converges weakly, as $N\rightarrow\infty$, to some probability measure~$\mathfrak{m}$ on~$\GT_k$. Then for any $i = 1, \dots, k$, the sequence $\{\lambda(N)_{N-i+1}\}_{N\geq 1}$, is bounded.
\end{lem}

\begin{proof}
The argument here is very similar to that of the previous proof for $k = 1$. As before, by making use of Macdonald generating functions, we can derive a more general equation than~(\ref{eqn:taylors}), which is
\begin{gather}
\frac{P_{\lambda(N)}\big(z_1 t^{1-k}, \dots, z_{k-1}t^{-1}, z_k, t^{-k}, \dots, t^{1-N}\big)}{P_{\lambda(N)}\big(1, t^{-1}, \dots, t^{1-N}\big)}\nonumber\\
\qquad{} = \sum_{\mu\in\GT_k}{M_k^{\lambda(N)}(\mu) \frac{P_{\mu}\big(z_1t^{1-k}, \dots, z_{k-1}t^{-1}, z_k\big)}{P_{\mu}\big(1, t^{-1}, \dots, t^{1-k}\big)}}.\label{eqn:taylors1}
\end{gather}
In the sum of the right-hand side above, note that $\mu$ ranges over signatures in $\GT_k$ such that
\begin{gather}\label{eqn:muineqs}
\mu_k \geq \lambda(N)_N,\quad \mu_{k-1} \geq \lambda(N)_{N-1}, \quad \dots,\quad \mu_1 \geq \lambda(N)_{N-k+1}.
\end{gather}
This is a consequence of the branching rule for Macdonald polynomials.
Another relevant observation is that for any $\mu\in\GT_k$ satisfying~(\ref{eqn:muineqs}), any monomial $c_{m_1, \dots, m_k} z_1^{m_1}\cdots z_k^{m_k}$ with $c_{m_1, \dots, m_k} \neq 0$ in the expansion of $P_{\mu}(z_1t^{1-k}, \dots, z_{k-1}t^{-1}, z_k)$ satisf\/ies: $m_k \geq \lambda(N)_N$; if $m_k = \lambda(N)_N$ then $m_{k-1} \geq \lambda(N)_{N-1}$; and so on until, if $m_2 = \lambda(N)_{N-k+2}, \dots, m_k \geq \lambda(N)_N$, then $m_1 \geq \lambda(N)_{N-k+1}$.
This is a consequence of the triangularity property of the Macdonald polynomials, see Def\/inition/Proposition~\ref{def:macdonaldpolys}.

In equation~(\ref{eqn:taylors1}) above, multiply both sides by $z_k^{-\lambda(N)_N}$ and then set $z_k = 0$.
From the branching rule for Macdonald polynomials and the fact that the branching coef\/f\/icients satisfy
\[
\psi_{\lambda(N)/(\lambda(N)_1, \dots, \lambda(N)_{N-1})}(q, t) = 1,
\]
the resulting left-hand side is
\[
\frac{P_{(\lambda(N)_1, \dots, \lambda(N)_{N-1})}\big(z_1 t^{1-k}, \dots, z_{k-1}t^{-1}, t^{-k}, \dots, t^{1-N}\big)}{P_{\lambda(N)}\big(1, t^{-1}, \dots, t^{1-N}\big)}.
\]
(Note that the argument $z_k$ is no longer present in the numerator.) Similarly, from the property
\[
\psi_{\mu/(\mu_1, \dots, \mu_{k-1})}(q, t) = 1, \qquad \textrm{for any} \quad \mu\in\GT_k,
\]
the resulting right-hand side is
\[
\sum_{\mu\in\GT_k\colon \mu_k = \lambda(N)_N}{M_k^{\lambda(N)}(\mu) \frac{P_{(\mu_1, \dots, \mu_{k-1})}\big(z_1t^{1-k}, \dots, z_{k-1}t^{-1}\big)}{P_{\mu}\big(1, t^{-1}, \dots, t^{1-k}\big)}}.
\]

After that, multiply both sides by $(z_{k-1}t^{-1})^{-\lambda(N)_{N-1}}$ and then set $z_{k-1}=0$; this gives
\begin{gather*}
\frac{P_{(\lambda(N)_1, \dots, \lambda(N)_{N-2})}\big(z_1 t^{1-k}, \dots, z_{k-2}t^{-2}, t^{-k}, \dots, t^{1-N}\big)}{P_{\lambda(N)}\big(1, t^{-1}, \dots, t^{1-N}\big)}\\
\qquad{} = \sum_{\substack{ \mu\in\GT_k\colon \\ \mu_{k-1} = \lambda(N)_{N-1},\ \mu_k = \lambda(N)_N } }{M_k^{\lambda(N)}(\mu) \frac{P_{(\mu_1, \dots, \mu_{k-2})}\big(z_1t^{1-k}, \dots, z_{k-2}t^{-2}\big)}{P_{\mu}\big(1, t^{-1}, \dots, t^{1-k}\big)}}.
\end{gather*}

Repeat the same procedure $k$ times, until we have multiplied both sides by $\big(z_1t^{1-k}\big)^{-\lambda(N)_{N-k+1}}$ and set $z_1=0$. The end result is
\begin{gather*}
M_k^{\lambda(N)}\left(\lambda(N)_{N-k+1}, \dots, \lambda(N)_N\right)\\
= \frac{P_{(\lambda(N)_1, \dots, \lambda(N)_{N-k})}\big( t^{-k}, \dots, t^{1-N} \big) P_{ (\lambda(N)_{N-k+1}, \dots, \lambda(N)_N) }\big( 1, t^{-1}, \dots, t^{1-k} \big) }{ P_{ \lambda(N) }\big( 1, t^{-1}, \dots, t^{1-N} \big)}\\
= t^{(N - k) (\lambda(N)_{N-k+1} + \lambda(N)_{N-k+2} + \dots + \lambda(N)_N )}\\
\times\frac{P_{(\lambda(N)_1, \dots, \lambda(N)_{N-k})}\big(1, t, \dots, t^{N-k-1}\big) P_{(\lambda(N)_{N-k+1}, \dots, \lambda(N)_N)}\big(1, t, \dots, t^{k-1}\big)}{P_{\lambda(N)}\big(1, t, t^2, \dots, t^{N-1}\big)}\\
= \prod_{i=1}^{N-k}\prod_{j=N-k+1}^N\! {\frac{\big(1 - t^{j-i}\big)\big(1 - qt^{j-i}\big)\cdots \big(1 - q^{\theta - 1}t^{j-i}\big)}{\big(1 \!-\! q^{\lambda(N)_i - \lambda(N)_j}t^{j-i}\big)\big(1 \!-\! q^{\lambda(N)_i - \lambda(N)_j+1}t^{j-i}\big)\cdots \big(1 \!-\! q^{\lambda(N)_i - \lambda(N)_j+\theta-1}t^{j-i}\big)}}\\
\geq \prod_{i=1}^{N-k}\prod_{j=N-k+1}^N{\big(1 - t^{j-i}\big)\big(1 - qt^{j-i}\big)\cdots \big(1 - q^{\theta - 1}t^{j-i}\big)} \geq ( (t; t)_{\infty} )^{\theta k},
\end{gather*}
where we used the homogeneity of Macdonald polynomials for the second equality, as well as three applications of Theorem~\ref{evaluation} for the third equality. Then $M_k^{\lambda(N)}(\lambda(N)_{N-k+1}, \dots, \lambda(N)_N) \geq c_k = ( (t; t)_{\infty} )^{\theta k} > 0$.
On the other hand, since $M_k^{\lambda(N)}$ converges weakly to $\mathfrak{m}$, there exist $N_1, N_2\in\N$ large enough such that
\begin{gather*}
M_k^{\lambda(N)}(\{ (n_1, \dots, n_k)\in\GT_k\colon -N_1 \leq n_1, \dots, n_k \leq N_1 \}) > 1 - c_k \qquad \forall\, N > N_2.
\end{gather*}
Therefore we conclude that $-N_1 \leq \lambda(N)_{N - k +1}, \dots, \lambda(N)_N \leq N_1$ for all $N > N_2$.
We conclude that each sequence $\{\lambda(N)_{N - i + 1}\}_{N\geq 1}$, $i = 1, \dots, k$, is uniformly bounded by a constant.
\end{proof}

Next we show that $\Mart$ is in bijection with the set $\Nu$. We do so by f\/irst constructing a~map $\Mart \rightarrow \Nu$ that will later be shown to be bijective.

\begin{prop}\label{prop:PhiProb}
Let $M$ be a $(q, t)$-central probability measure on $\Tau$ and $\{M_m\}_{m = 0, 1, 2, \dots}$ be its associated $(q, t)$-coherent system. If $M$ belongs to $\Omega_{q, t}^{{\rm Martin}}$, then there exists a unique $\nu\in\Nu$ such that
\begin{gather}
\Phi^{\nu}\big(x_1t^{1-m}, \dots, x_{m-1}t^{-1}, x_m; q, t\big) = \mathcal{P}_{M_m}(x_1, \dots, x_m), \nonumber\\
 \forall\, (x_1, \dots, x_m) \in \TT^m, \quad \forall\, m\geq 1.\label{eqn:macdonaldPhi}
\end{gather}
The function $\Phi^{\nu}(z_1, \dots, z_m; q, t)$ was defined in Theorem~{\rm \ref{thm:asymptotics1}} for $m = 1$, and in Theorem~{\rm \ref{thm:asymptotics2}} for general $m$.
\end{prop}
\begin{proof}
Let $M\in\Omega_{q, t}^{{\rm Martin}}$. By def\/inition, there exists a sequence $\{\lambda(N)\}_{N \geq 1}$ such that $\lambda(N)\in\GT_N$, $N \geq 1$, and $\mathbb{P}-\lim\limits_{m\rightarrow\infty}{\Lambda^N_m \delta_{\lambda(N)}} = M_m$ holds weakly, for any $m\in\N$.

By Lemma~\ref{lem:boundedparts}, it follows that for any $i = 1, 2, \dots$, the sequence $\{\lambda(N)_{N-i+1}\}_{N\geq 1}$ is uniformly bounded. In particular, since $\{\lambda(N)_N\}_{N\geq 1}$ is uniformly bounded, there exists a subsequence $\big\{N^1_1 < N^1_2 < N^1_3 < \cdots\big\}\subset\N$ such that $\lambda\big(N^1_1\big)_{N^1_1} = \lambda\big(N^1_2\big)_{N^1_2} = \cdots$. Analogously, using that $\{\lambda(N)_{N-1}\}_{N \geq 1}$ is bounded, there exists a subsequence $\big\{N^2_1 < N^2_2 < N^2_3 < \cdots\big\} \subset \big\{N^1_1 < N^1_2 < N^1_3 < \cdots\big\}$ such that $\lambda\big(N^2_1\big)_{N^2_1-1} = \lambda\big(N^2_2\big)_{N^2_2-1} = \cdots$. In a similar fashion, we can def\/ine subsequences $\big\{N^k_1 < N^k_2 < \cdots\big\}$ inductively.

Now consider the subsequence $\big\{N_1 = N^1_1 < N_2 = N^2_2 < N_3 = N_3^3 < \cdots\big\} \subset \N$. By construction, the sequence $\{\lambda(N_k)\}_{k\geq 1}$ is such that the limits $\lim\limits_{k\rightarrow\infty}{\lambda(N_k)_{N_k - i + 1}}$ exist for all $i = 1, 2, \dots$. Consequently, there exists $\nu\in\Nu$ such that $\{\lambda(N_k)\}_{k\geq 1}$ stabilizes to $\nu$ in the sense of Def\/inition~\ref{def:stabilizing}. From Theorem~\ref{thm:asymptotics2}, the following limit
\begin{gather}
\lim_{k\rightarrow\infty}{ \frac{P_{\lambda(N_k)}\big(x_1 t^{1-m}, x_2 t^{2-m}, \dots, x_m, t^{-m}, \dots, t^{1 - N_k}\big)}{P_{\lambda(N_k)}\big(1, t^{-1}, \dots, t^{1 - N_k}\big)} }\nonumber\\
\qquad{} = \Phi^{\nu}\big(x_1 t^{1-m}, x_2 t^{2-m}, \dots, x_m; q, t\big)\label{eqn:limitsPhi}
\end{gather}
holds uniformly for $(x_1, \dots, x_m)$ in compact subsets of the domain $(\C \setminus \{0\})^m$, in particular the convergence is uniform on $\TT^m$.

On the other hand, since we also have the weak convergence $\mathbb{P}-\lim\limits_{k\rightarrow\infty}{\Lambda^{N_k}_m\delta_{\lambda(N_k)}} = M_m$, then by Proposition~\ref{prop:convergencemacdonalds}:
\begin{gather}\label{eqn:limitsubsequence}
\lim_{k \rightarrow \infty}{\PPP_{\Lambda^{N_k}_m\delta_{\lambda(N_k)}}(x_1, \dots, x_m)} = \PPP_{M_m}(x_1, \dots, x_m)
\end{gather}
uniformly on $\TT^m$. By Proposition~\ref{prop:coherentsequences} and the homogeneity of the Macdonald polynomials, the Macdonald generating function of $\Lambda^{N_k}_m\delta_{\lambda(N_k)}$ equals
\begin{gather}
\PPP_{\Lambda^{N_k}_m\delta_{\lambda(N_k)}}(x_1, \dots, x_m) = \PPP_{\delta_{\lambda(N_k)}}\big(1^{N_k - m}, x_1, \dots, x_m\big)\nonumber\\
\qquad{} = \frac{P_{\lambda(N_k)}\big(1, t, \dots, t^{N_k - m - 1}, t^{N_k - m}x_1, \dots, t^{N_k - 1}x_m\big)}{P_{\lambda(N_k)}\big(1, t, t^2, \dots, t^{N_k - 1}\big)}\nonumber\\
\qquad{} = \frac{P_{\lambda(N_k)}\big(x_1 t^{1-m}, x_2 t^{2-m}, \dots, x_m, t^{-m}, \dots, t^{1 - N_k}\big)}{P_{\lambda(N_k)}\big(1, t^{-1}, t^{-2}, \dots, t^{1 - N_k}\big)}.\label{eqn:macdonaldsdelta}
\end{gather}

By combining~(\ref{eqn:limitsPhi}), (\ref{eqn:limitsubsequence}) and (\ref{eqn:macdonaldsdelta}), we conclude
\begin{gather*} \Phi^{\nu}\big(x_1 t^{1-m}, x_2 t^{2-m}, \dots, x_m; q, t\big) = \mathcal{P}_{M_m}(x_1, \dots, x_m) \qquad\text{for all}\quad (x_1, \dots, x_m) \in \TT^m
 \end{gather*} and all $m\geq 1$, as desired.
The statement about the uniqueness of $\nu\in\Nu$ follows from Lemma~\ref{lem:expansionintegrals}.
\end{proof}

Because of Proposition~\ref{prop:PhiProb}, for each $M\in\Mart$, there exists a unique $\nu\in\Nu$ such that equation~(\ref{eqn:macdonaldPhi}) is satisf\/ied for all $m\in\N$. Thus there is a well-def\/ined map of sets
\begin{gather}
\NN\colon \ \Omega_{q, t}^{{\rm Martin}} \rightarrow \Nu,\nonumber\\
\hphantom{\NN\colon}{} \ M \mapsto \nu,\label{def:Nmap1}
\end{gather}
which is determined by setting $\nu = \NN(M)$ be the unique element of~$\Nu$ such that~(\ref{eqn:macdonaldPhi}) is satisf\/ied. We prove below that $\NN$ is bijective, but f\/irst we show the convenient fact that~$\NN$ commutes with the automorphisms $A_k$, $k\in\Z$, of Section~\ref{sec:automorphismsAk}.

\begin{lem}\label{lem:commutingops}
Let $M\in\Mart$, $k\in\Z$. Then $A_k M \in\Mart$ and $\NN(A_k M) = A_k \NN(M)$.
\end{lem}
\begin{proof}
As $M\in\Mart\subset M_{\textrm{Prob}}(\Tau)$, Lemma~\ref{lem:qtcoherency} shows $A_k M\in M_{\textrm{Prob}}(\Tau)$ and $\{A_kM_m\}_{m \geq 0}$ is its corresponding $(q, t)$-coherent sequence. By def\/inition, there exists a sequence $\{\lambda(N)\}_{N \geq 1}$, $\lambda(N)\in\GT_N$, such that the weak convergence holds $\mathbb{P}-\lim\limits_{N\rightarrow\infty}{\Lambda^N_m\delta_{\lambda(N)}} = M_m$ $\forall\, m \geq 0$. We claim that $\mathbb{P}-\lim\limits_{N\rightarrow\infty}{\Lambda^N_m\delta_{A_k\lambda(N)}} = A_kM_m$ $\forall\, m \geq 0$, which would show $A_k M\in\Mart$ indeed.

Take any $m\in\Z_{\geq 0}$, $\mu\in\GT_m$, then
\begin{gather*}
\big(\Lambda^N_m\delta_{A_k\lambda(N)}\big)(\mu)
= \Lambda^N_m ( A_k\lambda(N), \mu )
= \Lambda^N_m ( \lambda(N), A_{-k}\mu )\\
\hphantom{\big(\Lambda^N_m\delta_{A_k\lambda(N)}\big)(\mu)}{}
= \big( \Lambda^N_m\delta_{\lambda(N)} \big) (A_{-k}\mu)
\xrightarrow{N\rightarrow\infty} M_m(A_{-k}\mu) = A_k M_m (\mu),
\end{gather*}
where we have used $\Lambda^N_m(A_k\nu, A_k\kappa) = \Lambda^N_m(\nu, \kappa)$, previously stated in the proof of Lemma~\ref{lem:qtcoherency}.

Let us move on to the second part of the lemma. Let $\NN(M) = \nu$; to show $\NN(A_k M) = A_k \nu$, we need
\begin{gather*}
\PPP_{A_k M_m}(x_1, \dots, x_m) = \Phi^{A_k\nu}\big(x_1 t^{1-m}, \dots, x_{m-1} t^{-1}, x_m; q, t\big)\\
 \forall\, (x_1, \dots, x_m)\in\TT^m, \quad \forall\, m\in\N.
\end{gather*}
For any $m\in\N$, $(x_1, \dots, x_m)\in\TT^m$, we have
\begin{gather*}
\PPP_{A_k M_m}(x_1, \dots, x_m)
= \sum_{\lambda\in\GT_m}{A_k M_m(\lambda)\frac{P_{\lambda}\big(x_1, x_2 t, \dots, x_m t^{m-1}\big)}{P_{\lambda}\big(1, t, t^2, \dots, t^{m-1}\big)}}\\
\hphantom{\PPP_{A_k M_m}(x_1, \dots, x_m)}{} = \sum_{\lambda\in\GT_m}{M_m(A_{-k}\lambda)\frac{P_{\lambda}\big(x_1, x_2 t, \dots, x_m t^{m-1}\big)}{P_{\lambda}\big(1, t, t^2, \dots, t^{m-1}\big)}}\\
\hphantom{\PPP_{A_k M_m}(x_1, \dots, x_m)}{}
= \sum_{\lambda\in\GT_m}{M_m(\lambda)\frac{P_{A_k \lambda}\big(x_1, x_2 t, \dots, x_m t^{m-1}\big)}{P_{A_k \lambda}\big(1, t, t^2, \dots, t^{m-1}\big)}}\\
\hphantom{\PPP_{A_k M_m}(x_1, \dots, x_m)}{} = (x_1x_2\cdots x_m)^k\cdot\sum_{\lambda\in\GT_m}{M_m(\lambda)\frac{P_{\lambda}\big(x_1, x_2 t, \dots, x_m t^{m-1}\big)}{P_{\lambda}\big(1, t, t^2, \dots, t^{m-1}\big)}}\\
\hphantom{\PPP_{A_k M_m}(x_1, \dots, x_m)}{}
= (x_1x_2\cdots x_m)^k\cdot\PPP_{M_m}(x_1, \dots, x_m)\\
\hphantom{\PPP_{A_k M_m}(x_1, \dots, x_m)}{} = (x_1\cdots x_m)^k\cdot\Phi^{\nu}\big(x_1 t^{1-m}, \dots, x_m; q, t\big)\\
\hphantom{\PPP_{A_k M_m}(x_1, \dots, x_m)}{} = t^{k {m \choose 2}}\big(\big(x_1 t^{1-m}\big)\cdots (x_m)\big)^k\cdot\Phi^{\nu}\big(x_1 t^{1-m}, \dots, x_m; q, t\big) \\
\hphantom{\PPP_{A_k M_m}(x_1, \dots, x_m)}{} = \Phi^{A_k \nu}\big( x_1 t^{1-m}, \dots, x_m; q, t\big),
\end{gather*}
where the last equality follows from Lemma~\ref{lem:Akfunctions}.
\end{proof}

\begin{prop}\label{thm:martinboundary0}
The Martin boundary $\Omega_{q, t}^{{\rm Martin}}$ of the $(q, t)$-GT graph is bijective to $\Nu$ under the map $\NN$ defined in~\eqref{def:Nmap1} above.
\end{prop}

\begin{proof}
\textbf{Step 1.} We prove $\NN$ is surjective.

Let $\nu\in\Nu$ be arbitrary; we want to show it belongs to the range of $\NN$. From Lemma~\ref{lem:commutingops}, we may assume $\nu_1 = 0$ without any loss of generality.
Consider the sequence $\{\lambda(N) = (\nu_N \geq \nu_{N-1} \geq \dots \geq \nu_1)\}_{N\geq 1}$ of signatures that stabilizes to $\nu$. Let $m\in\N$ be arbitrary. The f\/irst claim is that the limit $\lim\limits_{N\rightarrow\infty}{M^{\lambda(N)}_m(\lambda)}$ exists for any $\lambda\in\GT_m$.

By def\/inition~(\ref{kernelsmeasures}), $M^{\lambda(N)}_m(\lambda) = \Lambda^N_m(\lambda(N), \lambda) = 0$ unless $\lambda_m \geq \lambda(N)_N = \nu_1 = 0$. Thus the previous claim is clear if $\lambda_m < \nu_1 = 0$, in which case $\lim\limits_{N\rightarrow\infty}{M^{\lambda(N)}_m(\lambda)} = 0$. Assume now $\lambda_m \geq \nu_1 = 0$, i.e., $\lambda\in\GTp_m$.

From Theorem~\ref{thm:asymptotics2}, we have the uniform limit
\begin{gather}
\lim_{N\rightarrow\infty}{\frac{P_{\lambda(N)}\big(x_1t^{1-m}, \dots, x_{m-1}t^{-1}, x_m, t^{-m}, \dots, t^{1 - N}\big)}{P_{\lambda(N)}\big(1, t^{-1}, t^{-2}, \dots, t^{1 - N}\big)}} \nonumber\\
\qquad{} = \Phi^{\nu}\big(x_1t^{1-m}, \dots, x_{m-1}t^{-1}, x_m; q, t\big)\label{eqn:limitmacdonald0}
\end{gather}
on $\TT^m$. Correspondingly, the Fourier coef\/f\/icients of the normalized Macdonald characters converge to those of $\Phi^{\nu}(x_1 t^{1-m}, \dots, x_{m-1} t^{-1}, x_m; q, t)$. Proposition~ref{prop:coherentsequences} and the expansion of Corollary~\ref{cor:branchingrule} give
\begin{gather}
\frac{P_{\lambda(N)}\big(x_1t^{1-m}, \dots, x_{m-1}t^{-1}, x_m, t^{-m}, \dots, t^{1-N}\big)}{P_{\lambda(N)}\big(1, t^{-1}, t^{-2}, \dots, t^{1-N}\big)}\nonumber\\
\qquad{} = \sum_{ \mu\in\GT_m }{ M_m^{\lambda(N)}(\mu) \frac{P_{\mu}\big(x_1, x_2t, \dots, x_mt^{m-1}\big)}{P_{\mu}\big(1, t, \dots, t^{m-1}\big)} }\nonumber\\
\qquad{}= \sum_{\kappa\in\GT_m}{ m_{\kappa}\big(x_1, x_2t, \dots, x_mt^{m-1}\big) \sum_{\mu\in\GT_m}{ \frac{c_{\mu, \kappa}M_m^{\lambda(N)}(\mu)}{P_{\mu}\big(1, t, \dots, t^{m-1}\big)} } }.\label{eqn:prelimit}
\end{gather}
Let $\kappa\in\GT_m$ be arbitrary, and denote $n(\kappa) \myeq \kappa_2 + 2\kappa_3 + \dots + (m-1)\kappa_m$. Observe that $x_1^{\kappa_1}\cdots x_m^{\kappa}$ appears only in the monomial symmetric polynomial $m_{\kappa}(x_1, x_2t, \dots, x_mt^{m-1})$ and the corresponding term is $t^{n(\kappa)}x_1^{\kappa_1}\cdots x_m^{\kappa_m}$, so~(\ref{eqn:prelimit}) is essentially the Fourier expansion of the prelimit functions in~(\ref{eqn:limitmacdonald0}). Then we have that, for any $\kappa\in\GT_m$, the following limit
\begin{gather}\label{eqn:limitexists}
\lim_{N\rightarrow\infty}{ t^{n(\kappa)} \sum_{\mu\in\GT_m}{ \frac{c_{\mu, \kappa}M_m^{\lambda(N)}(\mu)}{P_{\mu}\big(1, t, \dots, t^{m-1}\big)} } }
\end{gather}
exists and equals the Fourier coef\/f\/icient of $x_1^{\kappa_1}\cdots x_m^{\kappa_m}$ in the function $\Phi^{\nu}(x_1t^{m-1}, \dots, x_m; q, t)$. As mentioned before, $M^{\lambda(N)}_m(\mu) = 0$ unless $\mu_m \geq \nu_1 \geq 0$. Thus we can restrict the sum in~(\ref{eqn:limitexists}) to $\mu\in\GTp_m$.
From the same analysis as in Proposition~\ref{prop:lemmaconv}, we can obtain that the limit $\lim\limits_{N\rightarrow\infty}{M^{\lambda(N)}_m(\mu)}$ exists for any $\mu\in\GTp_m$ with $|\lambda| = n$; let us denote $M_m(\mu) \myeq \lim\limits_{N\rightarrow\infty}{M_m^{\lambda(N)}(\mu)}$.
Immediately from the def\/inition (and Fatou's lemma) it follows that
\begin{gather*}
M_m(\lambda) \geq 0 \qquad \forall\, \lambda\in\GT_m,\\
\sum_{\lambda\in\GT_m}{ M_m(\lambda) } \leq 1.
\end{gather*}
Recall that we observed $M^{\lambda(N)}_m(\lambda) = 0$ if $\lambda\notin\GTp_m$, therefore $M_m(\lambda) = 0$ if $\lambda\notin\GTp_m$.

Next consider the following function
\begin{gather*}
F(x_1, \dots, x_m) = \sum_{\mu\in\GT_m}{ M_m(\mu) \frac{P_{\mu}\big(x_1, x_2t, \dots, x_m t^{m-1}\big)}{P_{\mu}\big(1, t, \dots, t^{m-1}\big)} }.
\end{gather*}
Clearly $F$ def\/ines a function on $\TT^m$, as it is def\/ined by an absolutely convergent series on the $m$-dimensional torus. Moreover its absolute value is upper bounded by~$1$, therefore $F \in L^{\infty}(\TT^m) \subset L^2(\TT^m)$.
By the same argument preceding~\eqref{rem:fouriercoeffs}, one shows, for any $\kappa\in\GT_m$, that the Fourier coef\/f\/icient of $x_1^{\kappa_1}\cdots x_m^{\kappa_m}$ in $F$ is
\begin{gather}\label{eqn:fouriers}
t^{n(\kappa)}\sum_{\mu\in\GT_m}{ \frac{c_{\mu, \kappa}M_m(\mu)}{P_{\mu}\big(1, t, \dots, t^{m-1}\big)} }.
\end{gather}
Observe that, if $\kappa\in\GT_m$ is f\/ixed, the sums in both (\ref{eqn:limitexists}) and (\ref{eqn:fouriers}) are f\/inite because $c_{\mu, \kappa} M_m(\mu) = 0$ unless $\mu\geq\kappa$ and $\mu \in \GTp_m$. Therefore $\lim\limits_{N\rightarrow\infty}{M_m^{\lambda(N)}(\mu)} = M_m(\mu)$ for all $\mu\in\GTp_m$ implies the equality between (\ref{eqn:limitexists}) and (\ref{eqn:fouriers}). In other words, the following convergence holds
\begin{gather}\label{eqn:fourierconvergence}
\frac{P_{\lambda(N)}\big(x_1t^{1-m}, \dots, x_{m-1}t^{-1}, x_m, t^{-m}, \dots, t^{1 - N}\big)}{P_{\lambda(N)}\big(1, t^{-1}, t^{-2}, \dots, t^{1 - N}\big)} \xrightarrow{\textrm{Fourier}} F(x_1, \dots, x_m)
\end{gather}
in the sense that all Fourier coef\/f\/icients of the left side of~(\ref{eqn:fourierconvergence}) converge to the corresponding Fourier coef\/f\/icients of $F(x_1, \dots, x_m)$, as $N$ tends to inf\/inity. But we already knew that the Fourier coef\/f\/icients of the left side of~(\ref{eqn:fourierconvergence}) converge to the corresponding Fourier coef\/f\/icients of $\Phi^{\nu}(x_1t^{1-m}, \dots, x_m; q, t)$. Therefore all the Fourier coef\/f\/icients of the dif\/ference $F(x_1, \dots, x_m) - \Phi^{\nu}(x_1t^{1-m}, \dots, x_m; q, t)$ of square-integrable functions on~$\TT^m$ must be zero. It follows that $F(x_1, \dots, x_m) = \Phi^{\nu}(x_1t^{1-m}, \dots, x_m; q, t)$. In particular, the equality holds for $x_1 = \dots = x_m = 1$, resulting in
\begin{gather*}
\sum_{\mu\in\GT_m}{M_m(\mu)} = F(1, \dots, 1) = \Phi^{\nu}\big(t^{1-m}, t^{2-m}, \dots, 1; q, t\big)\\
\qquad{} = \lim_{N\rightarrow\infty}{ \left.\frac{P_{\lambda(N)}\big( x_1t^{1-m}, \dots, x_{m-1}t^{-1}, x_m, t^{-m}, \dots, t^{1 - N}\big)}{P_{\lambda(N)}\big(1, t^{-1}, t^{-2}, \dots, t^{1 - N}\big)}\right|_{x_1 = \dots = x_m = 1} } = \lim_{N\rightarrow\infty}{ 1 } = 1.
\end{gather*}
Moreover $\Phi^{\nu}(x_1t^{1-m}, \dots, x_m; q, t) = F(x_1, \dots, x_m)$ is the Macdonald generating function of $M_m$.

We are almost done. For each $m\in\N$, we have constructed probability measures $M_m$ as limits of $M^{\lambda(N)}_m$ and shown that $\Phi^{\nu}(x_1t^{1-m}, \dots, x_m; q, t)$ is the generating function of $M_m$. Complete the sequence with $M_0 = \delta_{\varnothing}$, the delta mass at $\varnothing$. We claim that $\{M_m\}_{m\geq 0}$ is a $(q, t)$-coherent sequence. In fact, since $\{M_m^{\lambda(N)}\}_{m = 0, 1, \dots, N}$ is a $(q, t)$-coherent sequence, then
\begin{gather}\label{eqn:prelimitcoherence}
M_m^{\lambda(N)}(\mu) = \sum_{\lambda\in\GT_{m+1}}{ M_{m+1}^{\lambda(N)}(\lambda)\Lambda^{m+1}_m(\lambda, \mu) }
\end{gather}
for any $0\leq m < N$, $\mu\in\GT_m$.
As $N$ goes to inf\/inity, the left side of~(\ref{eqn:prelimitcoherence}) tends to $M_m(\mu)$.
By an argument similar to that in the proof of Proposition~\ref{prop:convergencemacdonalds}, one shows that the right side of~(\ref{eqn:prelimitcoherence}) converges to $\sum\limits_{\lambda\in\GT_{m+1}} M_{m+1}(\lambda) \Lambda^{m+1}_m(\lambda, \mu)$.
Indeed, the argument simply relies on the weak convergence $M_{m+1}^{\lambda(N)} \rightarrow M_{m+1}$ and the uniform (on $\lambda$) bound $|\Lambda^{m+1}_m(\lambda, \mu)| = \Lambda^{m+1}_m(\lambda, \mu) \leq 1$.
Therefore the limit of~(\ref{eqn:prelimitcoherence}) as $N\rightarrow\infty$ is
\begin{gather*}
M_m(\mu) = \sum_{\lambda\in\GT_m}{ M_{m+1}(\lambda)\Lambda^{m+1}_m(\lambda, \mu) },
\end{gather*}
for any $m\in\Z_{\geq 0}$, $\mu\in\GT_m$. Thus $\{M_m\}_{m\geq 0}$ is a $(q, t)$-coherent sequence and has an associated probability measure~$M$ on $\Tau$, as given by Proposition~\ref{prop:bijection}. By the def\/inition of $\NN$, we conclude $\NN(M) = \nu$.

\textbf{Step 2.} Next we show that $\NN$ is injective.

Let $M, M'\in\Mart$ have the same image $\nu$ under the map $\NN$. The goal is to prove $M = M'$. From Lemma~\ref{lem:commutingops}, we may assume $\nu_1 = 0$ without any loss of generality. Furthermore, we can assume that $M$ is the element of $\Mart$ such that $\NN(M) = \nu$ and that was contructed in the f\/irst step.

Let $\{M_m\}_{m\geq 0}$ and $\{M'_m\}_{m\geq 0}$ be the $(q, t)$-coherent sequences associated to $M, M'$, then
\begin{gather}\label{eqn:equalitymacdonalds}
\PPP_{M_1}(x) = \Phi^{\nu}(x; q, t) = \PPP_{M'_1}(x) \qquad \forall\, x\in\TT.
\end{gather}
As it was mentioned in Section~\ref{sec:generating}, Macdonald generating functions are uniformly bounded on the torus, in particular, $\PPP_{\mathfrak{m}}\in L^{\infty}(\TT) \subset L^2(\TT)$ for any probability measure $\mathfrak{m}$ on $\GT_1 = \Z$. Write the f\/irst equality of~(\ref{eqn:equalitymacdonalds}) for $x = e^{i\theta}$ as follows:
\begin{gather*}
\sum_{n\in\Z}{M_1(n)e^{in\theta}} = \sum_{n\in\Z}{M_1'(n)e^{in\theta}}.
\end{gather*}
Both sums above are expansions of a square integrable function on $\TT$ in terms of the basis $\{e^{in\theta}\}\subset L^2(\TT)$. Thus the (Fourier) coef\/f\/icients in both sums must agree, i.e., $M_1(n) = M'_1(n)$ $\forall \, n\in\Z$, and so $M_1 = M'_1$.

We aim to apply a similar argument to show that $M_m = M'_m$ for any $m\in\N$. We will be done once this is proved, as Proposition~\ref{prop:bijection} would then show $M = M'$. For general $m$, we make use of the fact that $M\in\Mart$ is the probability measure constructed in step~1: we use that each~$M_m$ is supported on~$\GTp_m$.

As above, the def\/inition of the map $\NN$ implies
\begin{gather*}
\PPP_{M_m}(x_1, \dots, x_m) = \Phi^{\nu}\big(x_1t^{1-m}, \dots, x_{m-1}t^{-1}, x_m; q, t\big) \\
\hphantom{\PPP_{M_m}(x_1, \dots, x_m)}{} = \PPP_{M'_m}(x_1, \dots, x_m) \qquad \forall\, (x_1, \dots, x_m)\in\TT^m.
\end{gather*}
The equality of the functions above implies the equality of corresponding Fourier coef\/f\/icients.
It follows that for any $\kappa\in\GT_m$ we have, see~\eqref{rem:fouriercoeffs},
\begin{gather}\label{eqn:step2fouriers}
\sum_{\mu\in\GT_m}{\frac{c_{\mu, \kappa}M_m(\mu)}{P_{\mu}\big(1, t, \dots, t^{m-1}\big)}} = \sum_{\mu\in\GT_m}{\frac{c_{\mu, \kappa}M_m'(\mu)}{P_{\mu}\big(1, t, \dots, t^{m-1}\big)}}.
\end{gather}
When $\kappa\notin\GTp_m$ (or equivalently $\kappa_m < 0$), we claim that the left side of~(\ref{eqn:step2fouriers}) is zero. In fact, for any $\mu\in\GT_m$, either $c_{\mu, \kappa} = 0$ when $\mu_m \geq 0$ or $M_m(\mu) = 0$ when $\mu_m < 0$, because $M_m$ is supported on~$\GTp_m$.
Then also the right side of~(\ref{eqn:step2fouriers}) is zero if $\kappa\notin\GTp_m$. By using also the properties of the coef\/f\/icients $c_{\mu, \kappa}$ stated in Corollary~\ref{cor:branchingrule}, we have
\begin{gather*}
0 = \sum_{\mu\in\GT_m}{\frac{c_{\mu, \kappa}M_m'(\mu)}{P_{\mu}\big(1, t, \dots, t^{m-1}\big)}} \geq \frac{c_{\kappa, \kappa}M_m'(\kappa)}{P_{\kappa}\big(1, t, \dots, t^{m-1}\big)} = \frac{M_m'(\kappa)}{P_{\kappa}\big(1, t, \dots, t^{m-1}\big)} \geq 0.
\end{gather*}
Then we must have $M_m'(\kappa) = 0$ if $\kappa\notin\GTp_m$, i.e., $M_m'$ is supported on $\GTp_m$. Finally an application of Lemma~\ref{cor:generating} says that if $\PPP_{M_m} = \PPP_{M_m'}$ on $\TT^m$, and $M_m$, $M_m'$ are both supported on $\GTp_m$, implies $M_m = M'_m$, thus f\/inishing the proof.
\end{proof}

\subsection{Characterization of the Martin boundary}\label{sec:martinchar}

Our next goal is to characterize the topological space $\Mart$ completely.
Recall the map $\NN\colon \Mart \rightarrow \Nu$, def\/ined above in~(\ref{def:Nmap1}), by letting $\NN(M) = \nu$ be the unique element of $\Nu$ such that
\begin{gather*}
\Phi^{\nu}\big(x_1t^{1-m}, \dots, x_{m-1}t^{-1}, x_m; q, t\big) = \mathcal{P}_{M_m}(x_1, \dots, x_m)\! \qquad \forall\, (x_1, \dots, x_m) \in \TT^m,\! \quad \forall\, m\geq 1.
\end{gather*}

Proposition~\ref{thm:martinboundary0} shows that $\NN$ is a bijection, so the inverse map $\NN^{-1}$ is well-def\/ined.
\begin{df}
For any $\nu\in\Nu$, we denote $\NN^{-1}(\nu)$ by $M^{\nu}$ and the corresponding $(q, t)$-coherent sequence by $\{M^{\nu}_m\}_{m \geq 0}$.
\end{df}
From step 1 of the proof of Proposition~\ref{thm:martinboundary0}, and Lemma~\ref{lem:commutingops}, we have that for the sequence of signatures $\{\lambda(N) = (\nu_N \geq \cdots \geq \nu_1)\}_{N \geq 1}$ which stabilizes to $\nu$, the following weak convergence holds
\begin{gather*}
\mathbb{P}-\lim_{N\rightarrow\infty}{\Lambda^N_m \delta_{\lambda(N)}} = M^{\nu}_m \qquad \forall\, m\in\N,
\end{gather*}
i.e.,
\begin{gather*}
\lim_{N\rightarrow\infty}{\Lambda^N_m \delta_{\lambda(N)}}(\mu) = \lim_{N\rightarrow\infty}{\Lambda^N_m(\lambda(N), \mu)} = M^{\nu}_m(\mu) \qquad \forall \, m\in\N, \quad \forall\, \mu\in\GT_m.
\end{gather*}

In fact, we note the same analysis as in step 1 of the proof of Proposition~\ref{thm:martinboundary0} shows that the weak convergence above holds for \textit{any} sequence of signatures $\{\lambda(N)\}_{N\geq 1}$ stabilizing to $\nu$.

We need the following lemma to prove that $\NN$ is a homeomorphism.

\begin{lem}\label{lem:posprob}\quad

\begin{enumerate}\itemsep=0pt
	\item[$1.$] For any $\nu\in\Nu$, $m\in\N$, $M_m^{\nu}$ is supported on $\{\mu = (\mu_1 \geq \dots \geq \mu_{m-1} \geq \mu_m) \in \GT_m\colon \mu_m \geq \nu_1,\, \mu_{m-1} \geq \nu_2, \, \dots,\, \mu_1 \geq \nu_m\}$.

	\item[$2.$] For any $m\in\N$, there exists $c_m \in (0, 1)$ such that for any integers $n_1 \leq n_2 \leq \dots \leq n_m$ and $\nu\in\Nu$ with $\nu_1 = n_1,\, \nu_2 = n_2,\, \dots,\, \nu_m = n_m$, we have $M^{\nu}_m(n_m \geq \dots \geq n_2 \geq n_1) \geq c_m$.

	\item[$3.$] Let $m\in\N$ and let $\widetilde{\nu}, \nu\in\Nu$ be such that $\widetilde{\nu}_i \geq \nu_i$ $\forall\, i > m$ and $\widetilde{\nu}_i = \nu_i$ $\forall\, 1\leq i\leq m$, then $M^{\widetilde{\nu}}_m(\nu_m \geq \dots \geq \nu_2 \geq \nu_1) \leq M^{\nu}_m(\nu_m \geq \cdots \geq \nu_2 \geq \nu_1)$.
\end{enumerate}
\end{lem}
\begin{proof}
Let us prove (1). We use the weak convergence of probabilities $\mathbb{P}-\lim\limits_{N\rightarrow\infty}{\Lambda^N_m\delta_{\lambda(N)}} = M^{\nu}_m$ $\forall\, m\in\N$, for the sequence of signatures $\{\lambda(N) = (\nu_N \geq \dots \geq \nu_2 \geq \nu_1)\}_{N\geq 1}$. Because of the def\/inition~(\ref{kernelsmeasures}) for the maps $\Lambda^N_m$, it follows that $\Lambda^N_m\delta_{\lambda(N)}(\mu) = \Lambda^N_m(\lambda(N), \mu) = 0$, unless $\mu_m \geq \lambda(N)_N = \nu_1, \dots, \mu_1 \geq \lambda(N)_{N-m+1}=\nu_m$. Thus also $M^{\nu}_m(\mu) = 0$ unless $\mu_m \geq \nu_1, \dots, \mu_1 \geq \nu_m$.

Next we show (2); to get started f\/ix $n_1\in\N$ and let us show that $M^{\nu}_1(n_1) \geq ((t; t)_{\infty})^{\theta}$ for any $\nu\in\Nu$ with $\nu_1 = n_1$ (so we can set $c_1 = ((t; t)_{\infty})^{\theta}$). The weak convergence mentioned above gives $M_1^{\nu}(n_1) = \lim\limits_{N\rightarrow\infty}{\Lambda^N_1\delta_{\lambda(N)}(n_1)} = \lim\limits_{N\rightarrow\infty}{\Lambda^N_1\delta_{\lambda(N)}(\lambda(N)_N)}$, for any $\nu\in\Nu$ with $\nu_1 = n_1$ and $\{\lambda(N)\}_{N\geq 1}$ as constructed above. The calculations in the proof of Lemma~\ref{lem:boundedparts0} show $\Lambda^N_1 \delta_{\lambda(N)}(\lambda(N)_N) \geq ( (t; t)_{\infty} )^{\theta}$ for all $N\in\N$, and thus also $M^{\nu}_1(n_1) \geq c_1 \myeq ( (t; t)_{\infty} )^{\theta} > 0$, for any $\nu\in\Nu$ with $\nu_1 = n_1$.

For a general $m\in\N$, we have
\[
M_m^{\nu}(n_m \geq \dots \geq n_2 \geq n_1) = \lim_{N\rightarrow\infty}{\Lambda^N_m\delta_{\lambda(N)}(\lambda(N)_{N-m+1} \geq \dots \geq \lambda(N)_N)},
\]
for $\{\lambda(N)\}_{N\geq 1}$ as constructed above.
In the proof of Lemma~\ref{lem:boundedparts}, we showed $\Lambda^N_m\delta_{\lambda(N)}(\lambda(N)_{N-m+1}$ $\geq \cdots \geq \lambda(N)_N) \geq c_m \myeq ((t; t)_{\infty})^{\theta m}$ holds for all $N\geq 1$. Therefore $M_m^{\nu}(n_m \geq \dots \geq n_2 \geq n_1) \geq c_m > 0$, for any $\nu\in\Nu$ with $\nu_1 = n_1$, $\nu_2 = n_2$, $\dots$, $\nu_m = n_m$.

Finally we move on to (3). To get started, we prove it for $m = 1$, so let $\widetilde{\nu}, \nu\in\Nu$ be such that $\widetilde{\nu}_i \geq \nu_i$ $\forall \, i\geq 2$ and $\widetilde{\nu}_1 = \nu_1$. Consider the following pair of sequences of signatures $\{\widetilde{\lambda}(N) = (\widetilde{\nu}_N \geq \dots \geq \widetilde{\nu}_2 \geq \widetilde{\nu}_1) \}_{N\geq 1}$ and $\{\lambda(N) = (\nu_N \geq \dots \geq \nu_2 \geq \nu_1) \}_{N\geq 1}$. Then we have the weak limits
\begin{gather}\label{eqn:someweaks}
\mathbb{P}-\lim_{N\rightarrow\infty}{\Lambda^N_m\delta_{\widetilde{\lambda}(N)}} = M^{\widetilde{\nu}}_m, \qquad \mathbb{P}-\lim_{N\rightarrow\infty}{\Lambda^N_m\delta_{\lambda(N)}} = M^{\nu}_m \hspace{.1in}\forall m\in\N.
\end{gather}
By the calculations in Lemma~\ref{lem:boundedparts0}, we have
\begin{gather*}
\Lambda^N_1\delta_{\lambda(N)}(\nu_1)
= \Lambda^N_1\delta_{\lambda(N)}(\lambda(N)_N)\\
\hphantom{\Lambda^N_1\delta_{\lambda(N)}(\nu_1)}{}
= \prod_{i=1}^{N-1}{\frac{(1 - t^{N-i})(1 - qt^{N-i})\cdots (1 - q^{\theta - 1}t^{N-i})}{(1 - q^{\lambda(N)_i - \lambda(N)_N}t^{N-i})\cdots (1 - q^{\lambda(N)_i - \lambda(N)_N + \theta - 1}t^{N - i})}}\\
\hphantom{\Lambda^N_1\delta_{\lambda(N)}(\nu_1)}{}= \prod_{i=1}^{N-1}{\frac{(1 - t^{N-i})(1 - qt^{N-i})\cdots (1 - q^{\theta - 1}t^{N-i})}{(1 - q^{\nu_{N-i+1} - \nu_1}t^{N-i})\cdots (1 - q^{\nu_{N-i+1} - \nu_1 + \theta - 1}t^{N-i})}}\\
\hphantom{\Lambda^N_1\delta_{\lambda(N)}(\nu_1)}{}\geq \prod_{i=1}^{N-1}{\frac{(1 - t^{N-i})(1 - qt^{N-i})\cdots (1 - q^{\theta - 1}t^{N-i})}{(1 - q^{\widetilde{\nu}_{N-i+1} - \widetilde{\nu}_1}t^{N-i})\cdots (1 - q^{\widetilde{\nu}_{N-i+1} - \widetilde{\nu}_1 + \theta - 1}t^{N-i})}}\\
\hphantom{\Lambda^N_1\delta_{\lambda(N)}(\nu_1)}{}=\prod_{i=1}^{N-1}{\frac{(1 - t^{N-i})(1 - qt^{N-i})\cdots (1 - q^{\theta - 1}t^{N-i})}{(1 - q^{\widetilde{\lambda}(N)_i - \widetilde{\lambda}(N)_N - \theta}t^{N-i+1})\cdots (1 - q^{\widetilde{\lambda}(N)_i - \widetilde{\lambda}(N)_N - 1}t^{N-i+1})}}\\
\hphantom{\Lambda^N_1\delta_{\lambda(N)}(\nu_1)}{}= \Lambda^N_1\delta_{\widetilde{\lambda}(N)}(\widetilde{\lambda}(N)_N) = \Lambda^N_1\delta_{\widetilde{\lambda}(N)}(\nu_1).
\end{gather*}
Thus taking into account the limits~(\ref{eqn:someweaks}) for $m = 1$, we deduce $M^{\nu}_1(\nu_1) \geq M^{\widetilde{\nu}}_1(\nu_1)$.

The proof of the third item for a general $m\in\N$ follows from similar calculations that are used to prove Lemma~\ref{lem:boundedparts}. We leave the details to the reader.
\end{proof}

\begin{thm}\label{thm:characterization}
The bijective map $\NN \colon \Mart \rightarrow \Nu$ is a homeomorphism.
\end{thm}

\begin{proof}
The f\/irst step shows that $\NN^{-1}$ is continuous and the second one shows that $\NN$ is continuous.

\textbf{Step 1.} Let $\big\{M^{\nu^{(i)}}\big\}_{i\geq 1} \subset \Omega_{q, t}^{{\rm Martin}}$ and $\big\{\nu^{(i)}\big\}_{i\geq 1}$ be the corresponding images under the map~$\NN$. If $\lim\limits_{i\rightarrow\infty}{\nu^{(i)}} = \nu\in\Nu$ pointwise, then we prove the weak limit $\mathbb{P}-\lim\limits_{i\rightarrow\infty}{M^{\nu^{(i)}}} = M^{\nu}$.

By applying some automorphism $A_k$ with large $k$, if necessary, and invoking Lemma~\ref{lem:commutingops}, we can assume $\nu_1 \geq 0$. Observe that $\lim\limits_{i\rightarrow\infty}{\nu^{(i)}_1} = \nu_1 \geq 0$ implies $\nu^{(i)}_1 \geq 0$ for large enough $i$. Thus let us also assume $\nu^{(i)}_1 \geq 0$ $\forall\, i\geq 1$, for simplicity.

We reduce our desired statement to simpler ones. We claim that $\lim\limits_{i\rightarrow\infty}{\nu^{(i)}} = \nu$ implies
\begin{gather}\label{eqn:limitPhis}
\lim_{i\rightarrow\infty}{\Phi^{\nu^{(i)}}(x_1, \dots, x_m; q, t)} = \Phi^{\nu}(x_1, \dots, x_m; q, t) \qquad \forall\, m\in\N
\end{gather}
uniformly on $\TT^m$.

First let us deduce our desired weak convergence from the limit above.
Recall that, due to Theorem~\ref{thm:asymptotics2}, all functions $\big\{\Phi^{\nu^{(i)}}(x_1, \dots, x_m; q, t)\big\}_{i \geq 1}$ and $\Phi^{\nu}(x_1, \dots, x_m; q, t)$ are entire functions and, in particular, they are continuous on the torus~$\TT^m$. So the convergence~(\ref{eqn:limitPhis}) implies the convergence of Fourier coef\/f\/icients. Thus for any $\kappa\in\GT_m$ we obtain, see~\eqref{rem:fouriercoeffs},
\begin{gather}\label{eqn:limitabove}
\lim_{i\rightarrow\infty}{ \sum_{\mu\in\GTp_m}{\frac{c_{\mu, \kappa}M^{\nu^{(i)}}_m(\mu)}{P_{\mu}(1, t, \dots, t^{m-1})}} } = \sum_{\mu\in\GTp_m}{\frac{c_{\mu, \kappa}M^{\nu}_m(\mu)}{P_{\mu}(1, t, \dots, t^{m-1})}}.
\end{gather}
Note that in both sides of~(\ref{eqn:limitabove}), the sums have been restricted to~$\GTp_m$. In fact, since we are assuming $\nu^{(i)}_1$, $\nu_1 \geq 0$ $\forall\, i\geq 1$, Lemma~\ref{lem:posprob}(1) implies that all probability measures $\big\{M^{\nu^{(i)}}_m\big\}_{i\geq 1}, M^{\nu}_m$ are supported on $\GTp_m$.

From Proposition~\ref{prop:lemmaconv}, the limit~(\ref{eqn:limitabove}) yields $\lim\limits_{i\rightarrow\infty}{M_m^{\nu^{(i)}}(\mu)} = M_m^{\nu}(\mu)$ $\forall\, \mu\in\GTp_m$. But also $M^{\nu^{(i)}}_m(\mu) = M_m^{\nu}(\mu) = 0$ for any $i\geq 1$ and any $\mu\notin\GTp_m$.
Therefore $\lim\limits_{i\rightarrow\infty}{M_m^{\nu^{(i)}}(\mu)} = M_m^{\nu}(\mu)$ holds for all $\mu\in\GT_m$, and so $\mathbb{P}-\lim\limits_{i\rightarrow\infty}{M^{\nu^{(i)}}_m} = M_m$. Note that we proved the weak convergence for any $m\in\N$. By Proposition~\ref{prop:bijection}, we conclude $\mathbb{P}-\lim\limits_{i\rightarrow\infty}{M^{\nu^{(i)}}} = M$.

We are left with the task of proving~(\ref{eqn:limitPhis}).
We show the uniform convergence in a neighborhood of the torus~$\TT^m$.
But it suf\/f\/ices to prove the limit on compact subsets of $\U_m$.
From the def\/inition of $\Phi^{\nu}(x_1, \dots, x_m; q, t)$ on $\U_m$, see~(\ref{def:Phinu}), and observing that $\TT^m\subset\U_m$, it follows that the result holds for any $m\in\N$ provided it holds for $m = 1$.
We prove~(\ref{eqn:limitPhis}) for $m = 1$ on an open neighborhood of~$\TT$.

Clearly a small enough open neighborhood of $\TT$ is a subset of~$\U$.
By the def\/inition of $\Phi^{\nu}(x; q, t)$ on $\U$, see~(\ref{eqn:Phi1}), the desired uniform convergence will hold if we verify
\begin{gather}\label{eqn:limitIntegrals}
\lim_{i\rightarrow\infty}{ \int_{\CC^+} {x^z\prod_{j=1}^{\infty}{\frac{(q^{-z+\nu_j^{(i)}}t^j; q)_{\infty}}{(q^{-z+\nu_j^{(i)}}t^{j-1}; q)_{\infty}}} } }= \int_{\CC^+} {x^z\prod_{j=1}^{\infty}{\frac{(q^{-z+\nu_j}t^j; q)_{\infty}}{(q^{-z+\nu_j}t^{j-1}; q)_{\infty}}} }
\end{gather}
uniformly on compact subsets of $x\in\C\setminus\{0\}$. Note that, since all $\big\{\nu^{(i)}_1\big\}_{i\geq 1}$, $\nu_1$, are nonnegative, we can take the same contour $\CC^+$ for both of the integrals in~(\ref{eqn:limitIntegrals}).

The pointwise convergence of integrands in~(\ref{eqn:limitIntegrals}) is clear. We still need some uniform estimates for the contribution of the tails of the left side in~(\ref{eqn:limitIntegrals}). This is similar to the proof of Theorem~\ref{thm:asymptotics1}.

Let $K\subset\C\setminus\{0\}$ be any compact set. Parametrize the tails of $\CC^+$ as $z = r \pm \frac{\pi\sqrt{-1}}{\ln{q}}$; for $r$ ranging from some large $R>0$ to $+\infty$, we want to show that the contribution of each of these lines is small. We have
\begin{gather*}
\sup_{x\in K}\left| x^z \cdot \prod_{j=1}^{\infty}{\frac{\big(q^{-z+\nu_j^{(i)}}t^j; q\big)_{\infty}}{\big(q^{-z+\nu_j^{(i)}}t^{j-1}; q\big)_{\infty}}} \right| \\
\qquad {} \leq {\rm const} \times |x|^r \times \prod_{j=1}^{\infty}{\frac{1}{\big(1 + q^{-r + \nu_j^{(i)} + \theta(j-1)}\big)\cdots \big(1 + q^{-r + \nu_j^{(i)} + \theta j - 1}\big)}}\\
\qquad{}\leq {\rm const}\cdot \frac{|x|^r}{q^{-kr}q^{\nu^{(i)}_1 + \nu^{(i)}_2+\dots +\nu^{(i)}_k}} = {\rm const}\cdot\frac{\big(|x|q^k\big)^r}{q^{\nu_1^{(i)} + \dots + \nu_k^{(i)}}},
\end{gather*}
for any $z = r\pm\frac{\pi\sqrt{-1}}{\ln{q}}$, and any $k\in\N$.
The constant above depends on $K$ but not on $i$.
Choose $k\in\N$ large enough so that $a \myeq \sup\limits_{x\in K}{|x|}\cdot q^k \in (0, 1)$.
Since $\lim\limits_{i\rightarrow\infty}{\nu_j^{(i)}} = \nu_j$ for all $j = 1, 2, \dots, k$, then $\sup\limits_{i\geq 1}{q^{-\nu_1^{(i)}-\nu_2^{(i)}-\dots-\nu_k^{(i)}}} \leq {\rm const} < \infty$.
Then there exists a constant $c = c_K > 0$, independent of $i$, such that
\begin{gather*}
\sup_{i\geq 1}\sup_{x\in K}\left| x^z \cdot \prod_{j=1}^{\infty}{\frac{\big(q^{-z+\nu_j^{(i)}}t^j; q\big)_{\infty}}{\big(q^{-z+\nu_j^{(i)}}t^{j-1}; q\big)_{\infty}}} \right| \leq c\cdot a^r, \qquad \textrm{if} \quad z = r \pm \frac{\pi\sqrt{-1}}{\ln{q}}.
\end{gather*}
Since $\int_R^{\infty}{a^r {\rm d}r} = -e^{R\ln{a}}/\ln{a} \xrightarrow{R\rightarrow\infty} 0$, we have just shown that the contribution of the tails of~$\CC^+$ is uniformly small. We can then apply dominated convergence theorem to conclude~(\ref{eqn:limitIntegrals}), as desired.

\textbf{Step 2.} As in the step above, let $\big\{M^{\nu^{(i)}}\big\}_{i\geq 1}$ be a sequence in $\Omega_{q, t}^{{\rm Martin}}$, whose images under~$\NN$ are $\big\{\nu^{(i)}\big\}_{i\geq 1}$.
If the weak limit $\mathbb{P}-\lim\limits_{i\rightarrow\infty}{M^{\nu^{(i)}}} = M$ holds, then we show that $M\in\Omega_{q, t}^{{\rm Martin}}$, that there a pointwise limit $\lim\limits_{i\rightarrow\infty}{\nu^{(i)}} = \nu\in\Nu$ and moreover $M = M^{\nu}$.

Let $m\in\N$ be arbitrary.
Thanks to Proposition~\ref{prop:bijection}, the limit $\mathbb{P}-\lim\limits_{i\rightarrow\infty}{M^{\nu^{(i)}}} = M$ implies $\mathbb{P}-\lim\limits_{i\rightarrow\infty}{M_m^{\nu^{(i)}}} = M_m$.
By Lemma~\ref{lem:posprob}(2), there exists $c_m \in (0, 1)$ such that $M_m^{\nu^{(i)}}\big(\nu_m^{(i)} \geq \cdots \geq \nu_1^{(i)}\big) \geq c_m > 0$ for all $i\geq 1$.
Since $M_m$ is a probability measure, there exists $N_0\in\N$ large enough such that $M_m((\lambda_1, \dots, \lambda_m)\in\GT_m\colon N_0 \geq \lambda_1, \dots, \lambda_m \geq -N_0) > 1 - \frac{c_m}{2}$.
Then $\mathbb{P}-\lim\limits_{i\rightarrow\infty}{M_m^{\nu^{(i)}}} = M_m$ implies $M_m^{\nu^{(i)}}((\lambda_1, \dots, \lambda_m)\in\GT_m \colon N_0 \geq \lambda_1, \dots, \lambda_m \geq -N_0) > 1 - c_m$ for all $i$ large enough.
It follows that $N_0 \geq \nu_m^{(i)}, \dots, \nu_1^{(i)} \geq -N_0$ for all $i$ large enough. In particular, the sequence $\big\{\nu_m^{(i)}\big\}_{i\geq 1}$ is bounded.

By using the boundedness of the sequences $\big\{\nu^{(i)}_m\big\}_{i\geq 1}$ and the ``diagonal argument'' of Proposition~\ref{prop:PhiProb}, there exists a subsequence $1\leq i_1 < i_2 < i_3 < \cdots$ such that the limits $\lim\limits_{k\rightarrow\infty}{\nu^{(i_k)}_m}$ exist, for all $m\in\N$. In other words, we have the pointwise limit $\lim\limits_{k\rightarrow\infty}{\nu^{(i_k)}} = \nu$, for some $\nu\in\Nu$.

We have the obvious implication
\begin{gather*}
\mathbb{P}-\lim_{i\rightarrow\infty}{M^{\nu^{(i)}}} = M \Longrightarrow \mathbb{P}-\lim_{k\rightarrow\infty}{M^{\nu^{(i_k)}}} = M,
\end{gather*}
But step 1 in this proof shows that $\lim\limits_{k\rightarrow\infty}{\nu^{(i_k)}} = \nu$ implies $\mathbb{P}-\lim\limits_{k\rightarrow\infty}{M^{\nu^{(i_k)}}} = M^{\nu}$. By uniqueness of weak limits, we have $M = M^{\nu}$.

We are left to show that $\lim\limits_{i\rightarrow\infty}{\nu^{(i)}} = \nu$. Above we only showed the existence of a subsequence $\{i_k\}_{k\geq 1}$ such that $\lim\limits_{k\rightarrow\infty}{\nu^{(i_k)}} = \nu$. Nevertheless the same argument gives us more. In fact, it shows that any subsequence $\big\{\nu^{(i_r)}\big\}_{r\geq 1} \subset \big\{\nu^{(i)}\big\}_{i\geq 1}$ must have a subsubsequence $\big\{\nu^{(i_{r(s)})}\big\}_{s\geq 1} \subset \big\{\nu^{(i_r)}\big\}_{r\geq 1}$ such that a pointwise limit $\lim\limits_{s\rightarrow\infty}{\nu^{(i_{r(s)})}} = \nu'$ exists. Then $M = M^{\nu'} = M^{\nu}$, but since the map~$\NN$ is bijective, we have $\nu = \nu'$. We therefore conclude that the sequence $\big\{\nu^{(i)}\big\}_{i\geq 1}$ itself must converge to $\nu$, f\/inishing the proof.
\end{proof}

\subsection{Relation between the Martin boundary and the (minimal) boundary}

The basic relation between the Martin and minimal boundary of the $(q, t)$-GT graph is the following statement, which actually holds in a much greater generality.

\begin{prop}[{consequence of \cite[Theorem~6.1]{OO1}}]\label{thm:martinboundary}
The following inclusion holds $\Omega_{q, t} \subseteq \Mart$. In other words, for any $M\in\Omega_{q, t}$, there exists a sequence $\{\lambda(N)\}_{N\geq 1}$, $\lambda(N)\in\GT_N$, such that
\begin{gather*}
\mathbb{P}-\lim_{N\rightarrow\infty}{M^{\lambda(N)}_m} = M_m \qquad \forall \, m = 0, 1, 2, \dots.
\end{gather*}
\end{prop}

In many examples, especially in the context of asymptotic representation theory, it is known that the Martin boundary of a branching graph is \textit{equal} to its minimal boundary, e.g.,~\cite{G, OO1}. In our case, we will also prove that this is the case by following the ideas in \cite{G, Ol0}. The following statement, which also holds in greater generality, will be useful.

\begin{prop}[{consequence of \cite[Theorem~9.2]{OO1}}]\label{thm:simplexcons}
Let $M'$ be any probability measure on~$M_{\rm prob}(\Tau)$. There exists a unique Borel probability measure $\pi$ belonging to~$\Omega_{q, t}$ such that{\samepage
\begin{gather*}
M'(S_{\phi}) = \int_{M\in\Omega_{q, t}}{M(S_{\phi}) \pi({\rm d}M)},
\end{gather*}
for any finite path $\phi = \big(\phi^{(0)} \prec \phi^{(1)}\prec \cdots \prec \phi^{(n)}\big)$ in the $(q, t)$-GT graph.}
\end{prop}

\subsection[Characterization of the boundary of the $(q, t)$-Gelfand--Tsetlin graph]{Characterization of the boundary of the $\boldsymbol{(q, t)}$-Gelfand--Tsetlin graph}

\begin{proof}[Proof of Theorem~\ref{thm:mainapplication}]
Let us now prove the items (1), (2) stated in the theorem.

(1) Thanks to Proposition~\ref{thm:martinboundary}, we have $\Omega_{q, t} \subseteq \Mart \subseteq M_{\rm prob}(\Tau)$. We claim that $\Mart \subseteq \Omega_{q, t}$. Because of Theorem~\ref{thm:characterization}, the Martin boundary $\Mart$ (with its induced topology from $M_{\rm prob}(\Tau)$) is homeomorphic to $\Nu$, under the map~$\NN$.
The claim would then show that the minimal boundary $\Omega_{q, t}$ of the $(q, t)$-GT graph is equal to the Martin boundary~$\Mart$ and therefore homeomorphic to~$\Nu$, under the map~$\NN$.

Let us prove the claim above. Let $\nu\in\Nu$ be arbitrary and let $M^{\nu}\in\Mart$ be the corresponding element of $\Mart$. Let $\mathfrak{i} \colon \Omega_{q, t} \hookrightarrow \Mart$ be the natural inclusion, considered as a~measurable map. By Proposition~\ref{thm:simplexcons}, there exists a unique probability measure $\pi$ on $\Omega_{q, t}$ such that
\begin{gather}\label{eqn1:proofthm}
M^{\nu} = \int_{M\in\Omega_{q, t}}{M\pi({\rm d}M)} = \int_{M\in\Mart}{M(\mathfrak{i}_*\pi)({\rm d}M)}.
\end{gather}
Note that $i_*\pi$ is a probability measure on $\Mart$ with $(i_*\pi)(\Omega_{q, t}) = 1$. Let $\widetilde{\pi}$ be the pushforward of $\mathfrak{i}_*\pi$ under the homeomorphism $\NN\colon \Mart \rightarrow \Nu$, so $\widetilde{\pi}$ is a Borel probability measure on~$\Nu$. Equation~\eqref{eqn1:proofthm} can be rewritten as
\begin{gather}\label{eqn2:proofthm}
M^{\nu} = \int_{\widetilde{\nu}\in\Nu}{M^{\widetilde{\nu}}\widetilde{\pi}({\rm d}\widetilde{\nu})}.
\end{gather}

We make a subclaim: $\widetilde{\pi}$ is the delta mass at $\nu\in\Nu$. Let us f\/irst deduce $\Mart \subseteq \Omega_{q, t}$ from this latter claim. In fact, if $\widetilde{\pi}$ is the delta mass at $\nu\in\Nu$, then $\mathfrak{i}_*\pi$ is the delta mass at $M^{\nu}$. But since we had $(i_*\pi)(\Omega_{q, t}) = 1$, then $M^{\nu} \in \Omega_{q, t}$. Since $M^{\nu}$ was an arbitrary element of $\Mart$, then we conclude $\Mart \subseteq \Omega_{q, t}$.

Let us now prove the subclaim that the probability measure $\widetilde{\pi}$ on $\Nu$ satisfying~(\ref{eqn2:proofthm}) must be the delta mass at $\nu\in\Nu$.
We show f\/irst that $\widetilde{\pi}$ is supported on $\{\widetilde{\nu}\in\Nu \colon \widetilde{\nu} \geq \nu\} \myeq \{\widetilde{\nu}\in\Nu \colon \widetilde{\nu}_1 \geq \nu_1,\, \widetilde{\nu}_2 \geq \nu_2,\, \widetilde{\nu}_3 \geq \nu_3, \dots\}$.
Since $\widetilde{\pi}$ is a Borel measure, the opposite would mean the existence of $m\in\N$ and $\kappa_1 \leq \kappa_2 \leq \cdots \leq \kappa_m$, such that $\kappa_i < \nu_i$ for some $1\leq i\leq m$, and
\begin{gather*}
\widetilde{\pi}\big( \big\{\widetilde{\nu}\in\Nu \colon \widetilde{\nu}_1 = \kappa_1, \dots, \widetilde{\nu}_m = \kappa_m \big\} \big) > 0.
\end{gather*}

As a consequence of~(\ref{eqn2:proofthm}) we have, for all $m\in\N$,
\begin{gather}\label{eqn3:proofthm}
M_m^{\nu} = \int_{\widetilde{\nu}\in\Nu}{M_m^{\widetilde{\nu}}\widetilde{\pi}({\rm d}\widetilde{\nu})}.
\end{gather}

We can now apply~(\ref{eqn3:proofthm}) to $\kappa = (\kappa_m \geq \dots \geq \kappa_1)\in\GT_m$:
\begin{gather*}
M_m^{\nu}(\kappa) = \int_{\widetilde{\nu}\in\Nu}{M_m^{\widetilde{\nu}}(\kappa)\widetilde{\pi}({\rm d}\widetilde{\nu})}.
\end{gather*}
From Lemma~\ref{lem:posprob}(1), the left-hand side of the equality above vanishes, while Lemma~\ref{lem:posprob}(2) shows that the right-hand side is at least $c_m \cdot \widetilde{\pi} ( \{\widetilde{\nu}\in\Nu \colon\widetilde{\nu}_1 = \kappa_1, \dots, \widetilde{\nu}_m = \kappa_m\} ) > 0$, thus there is a~contradiction.

From the fact that $\widetilde{\pi}$ is supported on $\{\widetilde{\nu}\in\Nu \colon \widetilde{\nu} \geq \nu\} \myeq \{\widetilde{\nu}\in\Nu \colon \widetilde{\nu}_1 \geq \nu_1, \widetilde{\nu}_2 \geq \nu_2, \dots\}$, and Lemma~\ref{lem:posprob}, parts (1) (3), we have that~(\ref{eqn3:proofthm}) evaluated at $(\nu_m \geq \cdots \geq \nu_1)$ is
\begin{gather*}
M^{\nu}_m(\nu_m \geq \dots \geq \nu_1)
= \int_{\widetilde{\nu}\in\Nu}{{M_m^{\widetilde{\nu}}(\nu_m \geq \dots \geq \nu_1)\widetilde{\pi}({\rm d}\widetilde{\nu})}}
= \int_{\substack{\widetilde{\nu}\in\Nu \\ \widetilde{\nu} \geq \nu}}{{M_m^{\widetilde{\nu}}(\nu_m \geq \dots \geq \nu_1)\widetilde{\pi}({\rm d}\widetilde{\nu})}}\\
\qquad{} = \int_{\substack{\widetilde{\nu}\in\Nu,\, \widetilde{\nu} \geq \nu \\ \widetilde{\nu}_i = \nu_i \ \forall i = 1, \dots, m}}{{M_m^{\widetilde{\nu}}(\nu_m \geq \dots \geq \nu_1)\widetilde{\pi}({\rm d}\widetilde{\nu})}}\\
\qquad\quad{} + \int_{\substack{\widetilde{\nu}\in\Nu, \, \widetilde{\nu} \geq \nu \\ \widetilde{\nu}_i \neq \nu_i \textrm{ for some }i\in\{1, \dots, m\}}}{{M_m^{\widetilde{\nu}}(\nu_m \geq \dots \geq \nu_1)\widetilde{\pi}({\rm d}\widetilde{\nu})}}\\
\qquad{}= \int_{\substack{\widetilde{\nu}\in\Nu, \widetilde{\nu} \geq \nu \\ \widetilde{\nu}_i = \nu_i \ \forall\, i = 1, \dots, m}}{{M_m^{\widetilde{\nu}}(\nu_m \geq \dots \geq \nu_1)\widetilde{\pi}({\rm d}\widetilde{\nu})}}\\
\qquad \leq M_m^{\nu}(\nu_m \geq \dots \geq \nu_1)\cdot\widetilde{\pi}(\{\widetilde{\nu} \in \Nu\colon \widetilde{\nu}_1 = \nu_1, \dots, \widetilde{\nu}_m = \nu_m\}).
\end{gather*}
Next Lemma~\ref{lem:posprob}(2) says that $M_m^{\nu}(\nu_m \geq \dots \geq \nu_1) \geq c_m > 0$, so we must have $\widetilde{\pi}(\{\widetilde{\nu} \in \Nu\colon \widetilde{\nu}_1 = \nu_1, \dots, \widetilde{\nu}_m = \nu_m\}) = 1$. Since this is true for any $m\in\N$, it follows that $\widetilde{\pi}$ must be the delta mass at $\nu$, thus proving our second claim and the full characterization of $\Omega_{q, t}$.

Let us return to the statement of item (1) in Theorem~\ref{thm:mainapplication}.
By def\/inition of the map $\NN$, the relations~(\ref{eq:macdonaldgenerating}) hold.
We have already observed that the Macdonald generating function def\/ining $\PPP_{M_m^{\nu}}(x_1, \dots, x_m)$ is absolutely convergent on $\TT^m$.
The last statement that says $M^{\nu}$ is determined by the relations~(\ref{eq:macdonaldgenerating}) follows from the uniqueness statement in Proposition~\ref{prop:PhiProb}.

(2) Let $\{M_m^{\nu}\}_{m \geq 0}$, $\{M_m^{A_k \nu}\}_{m \geq 0}$, be the $(q, t)$-coherent sequences associated to $M^{\nu}$ and $M^{A_k \nu}$, respectively.
By Lemma~\ref{lem:commutingops}, we have the f\/irst statement $M^{A_k\nu}(S_{A_k \phi}) = M^{\nu}(S_{\phi})$, for any f\/inite path $\phi$.
Next by virtue of Lemma~\ref{lem:qtcoherency}, $M^{A_k\nu}_m = A_k M^{\nu}_m$ for all $m\geq 0$.
Thus by following the def\/initions, $M^{A_k\nu}_m (A_k\lambda) = A_k M^{\nu}_m(A_k\lambda) = M^{\nu}_m(A_{-k}A_k \lambda) = M^{\nu}_m(\lambda)$, for any $\lambda\in\GT_m$ and $m\in\Z_{\geq 0}$.
\end{proof}

\appendix

\section[Basics on $q$-analysis]{Basics on $\boldsymbol{q}$-analysis}\label{app:qtheory}

A good reference for the material on $q$-analysis is \cite[Chapter~10]{AAR}. Assume $|q| < 1$ is an arbitrary complex number. Most statements work if $q$ is an indeterminate too.

The \textit{$q$-numbers} and the \textit{$q$-factorial} are def\/ined by
\begin{gather*}
[n]_q \myeq \frac{1 - q^n}{1 - q}, \qquad n\in\N,\\
[n]_q! \myeq [n]_q\cdots[2]_q[1]_q, \qquad n\in\N, \qquad [0]_q! \myeq 1.
\end{gather*}
It is evident that $[n]_q \rightarrow n$ and $[n]_q!\rightarrow n!$, as $q\rightarrow 1$, for any $n\in\Z_{\geq 0}$. Observe that we can also def\/ine $[x]_q$ for any $x\in\C$, as before, and it also holds that $[x]_q \rightarrow x$ as $q\rightarrow 1$. The \textit{$q$-Gamma function} is def\/ined by
\begin{gather*}
\Gamma_q(z) \myeq (1 - q)^{1-z}\frac{(q; q)_{\infty}}{(q^z; q)_{\infty}}.
\end{gather*}
From the def\/inition, the $q$-functional equation
\begin{gather*}
\Gamma_q(z + 1) = [z]_q\Gamma_q(z), \qquad z\notin\{\dots, -2, -1, 0\},
\end{gather*}
is evident. The $q$-Gamma function is a meromorphic function with simple poles at $z = 0, -1, -2, \dots$ and all their shifts by an integral multiple of $2\pi\sqrt{-1}/\ln{q}$.
The $q$-Gamma function has no zeroes in~$\C$. Moreover, we have the following convergence to the Gamma function.

\begin{thm}[{\cite[Corollary~10.3.4]{AAR}}]\label{qtheory1}
For any $z\in\C \setminus \{\dots, -2, -1, 0\}$, we have
\begin{gather*}
\lim_{q\rightarrow 1}{\Gamma_q(z)} = \Gamma(z).
\end{gather*}
\end{thm}
\begin{rem}\label{qtheory1remark}
As a consequence of Stieltjes--Vitali theorem, the convergence $\lim\limits_{q\rightarrow 1}{\Gamma_q(z)} = \Gamma(z)$ holds uniformly on compact subsets of $\C \setminus \{\dots, -2, -1, 0\}$.
\end{rem}

Other important identities we use in our paper are the \textit{$q$-binomial theorems}. To state them, we need to def\/ine the \textit{$q$-Pochhammer symbols} $(x; q)_n$ and $(x; q)_{\infty}$, for any $x\in\C$ and $n\in\Z_{\geq 0}$ by
\begin{gather*}
(x; q)_n \myeq
 \begin{cases}
 1 & \textrm{if } n = 0,\\
 \displaystyle \prod\limits_{i=1}^n{\big(1 - xq^{i-1}\big)} & \textrm{if } n\geq 1,\\
 \displaystyle \prod\limits_{i=1}^{\infty}{\big(1 - xq^{i-1}\big)} & \textrm{if } n = \infty.
 \end{cases}
\end{gather*}
Note that $|q| < 1$ implies that the product def\/ining $(x; q)_{\infty}$ is uniformly convergent for $x\in\C$, and thus $(x; q)_{\infty}$ is an entire function.

The $q$-binomial formula is the following

\begin{thm}[{\cite[Theorem~10.2.1]{AAR}}]\label{qtheory2}
For $|z| < 1$,
\begin{gather*}
\sum_{n=0}^{\infty}{\frac{(a; q)_n}{(q; q)_n}z^n} = \frac{(az; q)_{\infty}}{(z; q)_{\infty}}
\end{gather*}
\end{thm}

\begin{cor}\label{qtheory25}
For $z\in\C$, $m\in\N$,
\begin{gather*}
\sum_{n=0}^M{\frac{(q^{-1}; q^{-1})_M}{(q^{-1}; q^{-1})_n(q^{-1}; q^{-1})_{M-n}}(-1)^nq^{-{n \choose 2}}z^n} = \big(z; q^{-1}\big)_M.
\end{gather*}
\end{cor}
\begin{proof}
Let $a = q^{-M}$ in Theorem \ref{qtheory2}, and use $(q; q)_n = (-1)^n q^{{n + 1 \choose 2}} (q^{-1}; q^{-1})_n$. The statement is then proved for any $|z| < 1$; therefore it also holds for any $z\in\C$ because both sides are polynomials on $z$.
\end{proof}

Another application of the $q$-binomial theorem is the following limit.

\begin{thm}[{\cite[Theorem~10.2.4]{AAR}}]\label{qtheory3}
For any $a, b \in\R$ such that $b-a\notin\Z$, the following limit
\begin{gather*}
\lim_{q\rightarrow 1}{\frac{(xq^a; q)_{\infty}}{(xq^b; q)_{\infty}}} = (1 - x)^{b-a}
\end{gather*}
holds uniformly on compact subsets of $\{x\in\C \colon |x|\leq 1, x\neq 1\}$.
\end{thm}

\begin{rem}
If $b - a\in\Z$, the limit in Theorem {\rm \ref{qtheory3}} holds uniformly on compact subsets of~$\C\setminus\{1\}$.
\end{rem}

\section[Some properties of the rational functions $C^{(q, t)}_{\tau_1, \dots, \tau_n}(u_1, \dots, u_n)$]{Some properties of the rational functions $\boldsymbol{C^{(q, t)}_{\tau_1, \dots, \tau_n}(u_1, \dots, u_n)}$}\label{app:cfunctions}

\begin{lem}\label{Cprop1}
Assume $t = q^{\theta}$, for some $\theta\in\N$. Let $\tau_1, \dots, \tau_n\in\Z_{\geq 0}$ and $u_1, \dots, u_n$ be $n$ variables, then $C^{(q, t)}_{\tau_1, \dots, \tau_n}(u_1, \dots, u_n) = 0$ if some of the integers $\tau_1, \dots, \tau_n$ is strictly larger than $\theta$.
\end{lem}
\begin{proof}
Let us f\/irst rewrite the expression for $C^{(q, t)}_{\tau_1, \dots, \tau_n}(u_1, \dots, u_n)$, as it was done in \cite[Sec\-tion~6]{LS}. Let $v_i = q^{\tau_i}u_i$ for $i = 1, \dots, n$. Also, given $\tau = (\tau_1, \dots, \tau_n)$, let $T = T_{\tau} \myeq \{k\in\{1, 2, \dots, n\}\colon \tau_k \neq 0\}$. Then
\begin{gather}
C^{(q, t)}_{\tau}(u_1, \dots, u_n) = \prod_{k\in T}{t^{\tau_k-1}\frac{(q/t; q)_{\tau_k-1}}{(q; q)_{\tau_k-1}}\frac{(qu_k; q)_{\tau_k}}{(qtu_k; q)_{\tau_k}}} \nonumber\\
 \hphantom{C^{(q, t)}_{\tau}(u_1, \dots, u_n) =}{}\times \prod_{1\leq i<j\leq n}{\frac{(qu_i/tu_j; q)_{\tau_i}}{(qu_i/u_j; q)_{\tau_i}}\frac{(tu_i/v_j; q)_{\tau_i}}{(u_i/v_j; q)_{\tau_i}}} F_{\tau}(u; q, t),\label{Cnewform}
\end{gather}
where
\begin{gather*}
F_{\tau}(u; q, t) \myeq \sum_{K\subset T} (-1)^{|K|}(1/t)^{{|K| \choose 2}}\prod_{j\in T - K}{\frac{t - q^{\tau_j}}{1 - q^{\tau_j}}}\\
\hphantom{F_{\tau}(u; q, t) \myeq}{}\times \prod_{\substack{k\in K \\ j\in T - K}}{\frac{v_j - v_k/t}{v_j - v_k}}\prod_{k\in K}{\left( \frac{1 - tv_k}{1-v_k}\prod_{\substack{i\in T \\ i\neq k}}{\frac{u_i - v_k}{u_i - v_k/t}} \right)}.
\end{gather*}
Due to the factor $\prod\limits_{k\in T}{(q/t; q)_{\tau_k - 1}}$ in \eqref{Cnewform}, it follows that $\tau_k - \theta\in\N = \{1, 2, \dots\}$, for some $k$, implies $C^{(q, t)}_{\tau_1, \dots, \tau_n}(u_1, \dots, u_n) = 0$, as desired.
\end{proof}

\begin{lem}[{\cite[Lemma 6.1]{LS}}]\label{Cprop2}
Assume $t = q$. Let $\tau_1, \dots, \tau_n\in\Z_{\geq 0}$ and $u_1, \dots, u_n$ be $n$ variables, then $C^{(q, t)}_{\tau_1, \dots, \tau_n}(u_1, \dots, u_n) = 0$ if some of the integers $\tau_1, \dots, \tau_n$ is strictly larger than~$1$. If all $\tau_1, \dots, \tau_n\in\{0, 1\}$, then $C^{(q, t)}_{\tau_1, \dots, \tau_n}(u_1, \dots, u_n)$ does not depend on the variables $u_1, \dots, u_n$ and
\begin{gather*}
C_{\tau_1, \dots, \tau_n}^{(q, t)} = (-1)^{\tau_1 + \dots + \tau_n}.
\end{gather*}
\end{lem}

\begin{lem}\label{Cprop3}Assume $\theta\in\N$. If we let $a_n^{(\theta)} = \frac{C_n^{(q, q^{\theta})}(x_2^{-1}x_1q^{-\theta})}{\prod\limits_{i=0}^{\theta-1}{(x_1 - q^ix_2)}}$, for all $0\leq n\leq \theta$, then these expressions satisfy the relations
\begin{gather*}
a_n^{(\theta)} = \frac{1}{x_1 - x_2} \big( T_{q, x_2} a_n^{(\theta - 1)} - T_{q, x_1} a_{n-1}^{(\theta - 1)} \big), \qquad 1\leq n\leq \theta-1,\\
a_0^{(\theta)} = \frac{1}{\prod\limits_{i=0}^{\theta - 1}{(x_1 - q^i x_2)}}, \qquad a_{\theta}^{(\theta)} = \frac{1}{\prod\limits_{i=0}^{\theta - 1}{(x_2 - q^i x_1)}}.
\end{gather*}
\end{lem}
\begin{proof}
The expression $C_n^{(q, t)}(u)$, for $n\in\Z_{\geq 0}$, is much simpler than the general expression (\ref{Ccoeff}). It was f\/irst found by Jing and Jo\'zef\/iak in~\cite{JJ} and it reads $C_n^{(q, t)}(u) = t^n\frac{(1/t; q)_n}{(q; q)_n}\frac{(u; q)_n}{(qtu; q)_n}\frac{1 - q^{2n}u}{1-u}$. Then, for $0\leq n\leq \theta$:
\begin{gather}\label{simplifiedC}
C_n^{(q, q^{\theta})}\big(x_1/\big(q^{\theta}x_2\big)\big) = \frac{q^{\theta}x_2 - q^{2n}x_1}{q^{\theta}x_2 - x_1}\frac{1}{t^n}\prod_{i=1}^n{\left\{\frac{q^{\theta} - q^{i-1}}{1 - q^i}\frac{q^{\theta}x_2 - q^{i-1}x_1}{x_2 - q^ix_1}\right\}}.
\end{gather}
From (\ref{simplifiedC}), it is only a matter of tedious computation to check the three identities given in the lemma.
\end{proof}

\subsection*{Acknowledgements}

It is my pleasure to thank Alexei Borodin for his generous sharing of time and ideas. I am equally indebted to Vadim Gorin, for his interest in my work, many helpful discussions and for sharing some of his notes on the $q$-GT graph. This work would not exist without them. I would also like to thank Jiaoyang Huang, for being an excellent sounding board at the beginning stage of this project, Konstantin Matveev for his expert help with the software Mathematica, and Grigori Olshanski for comments in a previous draft of this paper and for asking a question that led to my proof of Theorem~\ref{thm:mainapplication}.
The suggestions of the referees helped improved this text greatly; many thanks are due to them.

\pdfbookmark[1]{References}{ref}
\LastPageEnding

\end{document}